\documentclass[12pt,letterpaper]{amsart}

\usepackage{euscript,amsfonts,amssymb,amsmath,amscd,amsthm,enumerate,hyperref}

\usepackage{tikz,bm,float}
\usetikzlibrary{calc}
\usetikzlibrary{patterns}
\colorlet{lgray}{white!85!black}
\colorlet{lred}{white!75!red}

\usepackage{comment}

\usepackage{graphicx}

\usepackage{color}

\usepackage[margin=0.9in]{geometry}
\newtheorem{theorem}{Theorem} 
\newtheorem*{theorem*}{Theorem}
\newtheorem{lemma}[theorem]{Lemma}
\newtheorem{definition}[theorem]{Definition}
\newtheorem{proposition}[theorem]{Proposition}

\theoremstyle{remark}
\newtheorem{remark}[theorem]{Remark}

\numberwithin{equation}{section} \numberwithin{theorem}{section}

\newcommand{\la}{\lambda}

\newcommand{\N}{\mathbb N}

\newcommand{\R}{\mathbb R}

\newcommand{\C}{\mathbb C}

\newcommand{\Z}{\mathbb Z}

\newcommand{\Y}{\mathbb Y}

\newcommand{\Q}{\mathbb Q}

\newcommand{\ii}{{\mathbf i}}

\newcommand{\eps}{\varepsilon}

\newcommand{\p}{\mathbf \Pi}

\newcommand{\mm}{\mathbb{MM}}

\newcommand{\Sig}{\mathrm{Sig}}

\usepackage{MnSymbol}

\usepackage{skak}

\usepackage{adjustbox}
\usetikzlibrary{arrows}
\usetikzlibrary{decorations.pathreplacing}

\title[Hall-Littlewood RSK field]
{Hall-Littlewood RSK field}

\author{Alexey Bufetov}

\address[Alexey Bufetov]{Department of Mathematics, Massachusetts Institute of Technology, Cambridge, MA, USA. E-mail: alexey.bufetov@gmail.com}

\author{Konstantin Matveev}

\address[Konstantin Matveev]{ Department of Mathematics, Brandeis University, Waltham, MA, USA. E-mail: kosmatveev@gmail.com}

\begin{document}

\maketitle

\begin{abstract}
We introduce a randomized Hall-Littlewood RSK algorithm and study its combinatorial and probabilistic properties. On the probabilistic side, a new model --- the Hall-Littlewood RSK field --- is introduced. Its various degenerations contain known objects (the stochastic six vertex model, the asymmetric simple exclusion process) as well as a variety of new ones. We provide formulas for a rich class of observables of these models, extending existing results about Macdonald processes. On the combinatorial side, we establish analogs of properties of the classical RSK algorithm: invertibility, symmetry, and a ``bijectivization'' of the skew-Cauchy identity.
\end{abstract}

\section{Introduction}

\subsection{Overview}

This paper is motivated by two directions of recent developments. First, in the last twenty years there was a lot of progress in applications of the technique of symmetric functions to probability. The frameworks of Schur processes \cite{O}, \cite{OR}, and Macdonald processes \cite{BC} have proved useful in numerous stochastic models. However, it seems that some important connections still remain to be discovered. For example, only very recently it was realized that the asymmetric stochastic exclusion process (ASEP) falls into the formalism of Macdonald processes.

Second, there is a purely combinatorial motivation for generalizing the classical Robinson-Schensted-Knuth (RSK) algorithm to Macdonald symmetric functions (RSK itself is related to Schur symmetric functions). Perhaps, the main obstacle is that the generalization of RSK has to be randomized, and one needs to find suitable analogs for properties of the classical RSK algorithm.

We develop both these directions for the Hall-Littlewood symmetric functions. The main results of this paper are:

\begin{itemize}

\item We introduce a Hall-Littlewood RSK algorithm (Definition \ref{def:RSK-main1}), and prove its combinatorial properties: invertibility (Proposition \ref{lemma:flip}), symmetry (\ref{prop:symmetry}), Markov evolution of first columns (\ref{prop:Markov-proj}). We show that it provides a bijective proof of the skew-Cauchy identity for Hall-Littlewood functions (Section \ref{sec:bij-proof}).

\item We introduce new integrable probabilistic models: the Hall-Littlewood RSK field (Section \ref{sec:field}), a multi-layer stochastic vertex model (Sections \ref{sec:2vertMod} and Section \ref{sec:multi-layer-vertMod}), and a multi-layer ASEP (Sections \ref{sec:2asep} and \ref{sec:multi-asep}). The last two naturally generalize a stochastic six vertex model and ASEP, respectively. We give a new proof to the main result of \cite{BBW} which claims that the distribution of the height function of the stochastic six vertex model is equal to the distribution of the first column of Young diagrams distributed according to a Hall-Littlewood process (Section \ref{sec:st6v}). The results of \cite{BBW} have already lead to new applications for ASEP (\cite{CD}), and we hope that more applications will be available with the use of results of this paper.

\item Extending the results of \cite{BCGS}, we prove formulas for observables of a general Hall-Littlewood process (Theorem \ref{prop:gen-HL-proc}). This general result provides formulas for observables of the Hall-Littlewood RSK field, a multi-layer stochastic vertex model, and a multi-layer ASEP. In particular, they provide formulas for observables for a stochastic six vertex model and ASEP (similar formulas for these objects are known: \cite{TW1}, \cite{TW2}, \cite{BCG}, but their derivation with the use of Hall-Littlewood processes is new). We also provide formulas for mixed q-Whittaker / Hall-Littlewood measures (Proposition \ref{prop:qWHLandHL}).
\end{itemize}

Section \ref{sec:2} contains results about observables of Macdonald processes. In Section \ref{sec:RSK} we deal with a Hall-Littlewood RSK algoritm. Section \ref{sec:degen} deals with integrable models by combining the results of Sections \ref{sec:2} and \ref{sec:RSK}.
Note that the material of Section \ref{sec:RSK} is independent of Sections \ref{sec:2} and \ref{sec:degen}; a reader interested in combinatorics only might read Section \ref{sec:RSK} alone.

In the rest of the introduction we discuss our results in more detail.

\subsection{Observables of Macdonald processes}

Macdonald functions (see \cite[Section 6]{M}) are an important family of symmetric functions which depend on parameters $q,t$; in probabilistic applications one typically assumes $0 \le q, t < 1$.
Macdonald processes (as introduced in \cite{BC}) are probability measures on sequences of Young diagrams determined by products of (skew) Macdonald functions. The main tool for an analysis of this process is the following: The application of Macdonald difference operators allows to obtain exact formulas for certain observables of the process, which can subsequently be analyzed in asymptotic regimes. Starting from \cite{BC}, this idea has found many applications in the last five years, see e.g. \cite{BP1}, \cite{Cor}, and references therein.

While this scheme produced many rich results in the cases of q-Whittaker ($t=0$; \cite{BC}, \cite{BP1}, \cite{Cor}) and Jack ($q=t^{\alpha}$ and the limit transition $q,t \to 1$; \cite{BG}) functions, the application of it to the case of Hall-Littlewood functions ($q=0$) was restricted by dealing with the first column of Young diagrams only (\cite{D}, \cite{CD}). The reason for this is simple --- the Macdonald difference operators do not provide convenient observables for other columns.

Section \ref{sec:2} fills this gap. We provide (Theorem \ref{prop:gen-HL-proc}) observables of Hall-Littlewood processes which are as suitable for an asymptotic analysis as observables for q-Whittaker processes. We do this by combining the general formalism of \cite{BCGS} with the application of the Macdonald involution operator. The derivation is easy, yet it seems to be an important part of the framework of Macdonald processes which is missing in the current literature.

The same idea allows to obtain formulas for observables of the mixed q-Whittaker / Hall-Littlewood measure (see Proposition \ref{prop:qWHLandHL}). We believe that these formulas will find their applications as well.

\subsection{Hall-Littlewood RSK algorithm}

The classical RSK algorithm (see e.g. \cite[Chapter 7]{Sta}) is a profound combinatorial object which is closely related to properties of Schur (Macdonald parameters $q=t$) symmetric functions. One way to formulate the connection is to say that the RSK algorithm provides a bijective proof of a skew Cauchy identity for Schur functions. Also the RSK algorithm enjoys many other nice combinatorial features.

It is natural to ask whether the algorithm can be naturally generalized to Macdonald functions beyond the Schur case. The answer is non trivial: It turns out that the algorithm has to become random.

The idea of the randomization of the RSK algorithm is by no means new. For the Robinson-Schensted algorithm (an important particular case of the RSK algorithm) the natural randomized versions were constructed in \cite{OCP}, \cite{BP3} for the q-Whittaker functions and in \cite{BP} for the Hall-Littlewood functions. In the full generality the randomized version of RSK for q-Whittaker functions was constructed in \cite{MP} and further studied in \cite{P}.

The main result of Section \ref{sec:RSK} is a construction of a randomized Hall-Littlewood RSK (HL-RSK) algorithm. Let us review properties of our construction.

Let $\Y$ be the set of all Young diagrams, and for $\la \in \Y$ let $|\la|$ be the number of boxes in $\Y$. One way ({\it Fomin growth diagram}, due to Fomin \cite{F1}) to define the classical RSK algorithm is to define a certain operation on Young diagrams. Namely, consider Young diagrams $\la, \mu, \nu$ such that $\mu / \la$ and $\nu / \la$ are horizontal strips, and $r \in \Z_{\ge 0}$. Then there exists a map $U: \Y \times \Y \times \Y \times \Z_{\ge 0} \mapsto \Y$, $(\la, \mu, \nu, r) \mapsto U^r( \la;\mu, \nu)$ with many remarkable properties (in fact, there are several possible choices of the map: row or column insertion, usual or dual RSK, etc, which lead to maps with similar properties). We refer to \cite{Sta} for a detailed exposition of the classical RSK.

In the setting of Hall-Littlewood functions the map $U$ has to become random (see Section \ref{sec:RSKgoal} for a discussion of this point). That is, a randomized Hall-Littlewood RSK algorithm takes as an input $\la, \mu, \nu, r$ and outputs $\rho \in \Y$ with a certain probability $U^r ( \mu \to \rho \mid \la \to \nu)$ which needs to be defined. In this paper we also define a (randomized) inverse algorithm which takes as an input Young diagrams $\mu, \nu, \rho$ such that $\rho / \mu$ and $\rho / \nu$ are horizontal strips, and outputs a pair $(\mu, r) \in \Y \times \Z_{\ge 0}$ with a certain probability $\hat U^r ( \nu \to \la \mid \rho \to \mu )$ which needs to be defined.

We define coefficients $U^r ( \mu \to \rho \mid \la \to \nu)$ and $\hat U^r ( \nu \to \la \mid \rho \to \mu )$ in Section \ref{sec:RSK}. They have the following properties:

1) for any fixed $\la, \mu, \nu, r$ such that $\mu/\la$, $\nu / \la$ are horizontal strips we have
$$
\sum_{\rho} U^r ( \mu \to \rho \mid \la \to \nu) =1, \qquad U^r ( \mu \to \rho \mid \la \to \nu) \ge 0,
$$
and for $\mu, \nu, \rho$ such that $\rho / \mu$ and $\rho / \nu$ are horizontal strips we have
$$
\sum_{\la, r} \hat U^r ( \nu \to \la \mid \rho \to \mu ) =1, \qquad \hat U^r ( \nu \to \la \mid \rho \to \mu ) \ge 0.
$$
These equalities guarantee that we indeed have probabilistic algorithms.

2) The key property which relates the randomized HL-RSK and its inverse is
$$
\left( 1 - t \mathbf{1}_{r \ge 1} \right) \phi_{\nu / \la} \psi_{\mu / \la} U^r (\mu \to \rho \mid \la \to \nu) = \phi_{\rho / \mu} \psi_{\rho / \nu} \hat U^r ( \nu \to \la \mid \rho \to \mu ),
$$
see Section \ref{sec:3-prelim} for definitions of $\phi$- and $\psi$-functions. This property gives a bijective proof of the (generalized) skew Cauchy identity for Hall-Littlewood functions (see Section \ref{sec:bij-proof}).

Other properties are very similar for the (direct) HL-RSK algorithm and the inverse HL-RSK, so we formulate them for the former one only.

3) $U^r ( \mu \to \rho \mid \la \to \nu)=0$ unless $|\nu| - |\la| +r = |\rho| - |\mu|$, and $\rho / \nu$, $\rho / \mu$ are horizontal strips. 

4) The coefficients $U^r ( \mu \to \rho \mid \la \to \nu)$ are defined as a result of an iterative procedure. Informally, this iterative procedure ``glues'' elementary steps which are formed by cases when $r=0$, $|\mu|-|\la|=1$, $|\nu|-|\la|=1$. See Definition \ref{def:RSK-main1} and Definition \ref{def:refined} for precise statements.

5) Symmetry property: $U^r ( \mu \to \rho \mid \la \to \nu) = U^r ( \nu \to \rho \mid \la \to \mu)$.

6) For any $k$, the projection of the HL-RSK algorithm to the first $k$ columns of Young diagrams is Markovian.

All of these properties naturally generalize the properties of the classical RSK; one recovers the classical setting if $t$ is equal to $0$.

Our approach to the HL-RSK is significantly different from the approach of \cite{MP}, \cite{P} to the q-Whittaker RSK and provides a new way for constructing and working with such dynamics. We regard the consideration of the inverse HL-RSK algorithm, and properties 2) and 4) above as important novelties of the current paper. The symmetry property for the q-Whittaker RSK algorithm was proved in \cite{P}.

Iteratively applying the randomized HL-RSK algorithm for a two-dimensional array of inputs we define a new integrable probabilistic model --- the Hall-Littlewood RSK field (Section \ref{sec:field}), which is a probability measure on a two-dimensional array of Young diagrams. We show that it samples Hall-Littlewood processes (Proposition \ref{prop:pathHLprocess}). Thus, the results of Section \ref{sec:2} allow to compute rich families of observables of the Young diagrams and various degenerations of this large combinatorial object.


\subsection{Integrable degenerations}

Section \ref{sec:degen} deals with some available degenerations of the Hall-Littlewood RSK field.

At the beginning, we identify the well-known objects --- the stochastic six vertex model (introduced in \cite{GS} and recently studied in \cite{BCG}) and ASEP as degenerations of this field which appear when one considers the first column of random Young diagrams from the field only. This provides a new proof of the main result of \cite{BBW} which relates first columns of Young diagrams distributed according to Hall-Littlewood processes and the height function of the stochastic six vertex model (see also \cite{Bor} for connections between one point distributions of Macdonald measures and certain vertex models). As a corollary, results of Section \ref{sec:2} provide formulas for observables of these models. These formulas are essentially equivalent to the ones obtained in \cite{TW1}, \cite{TW2} for ASEP with step initial condition, and in \cite{BCG}, \cite{BP2} for the stochastic six vertex model by different methods, yet their derivation through Macdonald processes might be conceptually important.

Full Young diagrams from the Hall-Littlewood RSK field contain more information than just their first columns, and thus provide more integrable models. In this paper we introduce the multi-layer stochastic vertex model and the multi-layer ASEP, which are natural generalizations of the stochastic six vertex model and ASEP, respectively. For any $k \in \Z_{\ge 1}$ a $k$-layer stochastic vertex model appears when one restricts the Hall-Littlewood RSK field to the first $k$ columns of Young diagrams. Results of Section \ref{sec:2} provide observables for these models. We mainly focus on the two-layer stochastic vertex model for which we give an explicit description (see Section \ref{sec:2vertMod}).

We use the term ``multi-layer'' for our models by an analogy with \cite{PrSp} (see also \cite{BO}, \cite{OR}) where similar multi-layer objects were introduced in the case of the classical RSK algorithm related to Schur functions. The multi-layer ASEP seems to be of a quite different nature than the extensively studied multi-species ASEP (see e.g. \cite{TW3} and references therein about the latter).

Let us describe the two-layer ASEP dynamics which we introduce. This is a continuous time dynamics of particles which are placed in $\Z+1/2$. The particles are of two types. Each point of $\Z+1/2$ might contain either one or zero particles of each type; in particular, it might contain two particles of different types simultaneously. Let $\mathbf{\hat X (\tau)} \subset \Z$, $\tau \in \R_{\ge 0}$, be positions of particles of the first type at time $\tau$, and let $\mathbf{\hat Y (\tau)}$ be positions of particles of the second type at time $\tau$. At $\tau=0$ let $\mathbf{\hat X} (0) = \mathbf{\hat Y} (0) = \{ -1/2, -3/2, -5/2, \dots \}$ (step initial condition).

Let $h_0 (n; \tau)$, $n \in \Z$, be the number of particles of the first type which are to the right of $n$ at moment $\tau$, and let $h_1 (n;\tau)$ be the number of particles of the second type which are to the right of $n$ at moment $\tau$. These are convenient height functions of the model. We will also need the following quantity
$$
k (n;\tau) := h_1 (n;\tau) - h_0 (n;\tau).
$$
The rules of dynamics will imply that $k (n;\tau) \ge 0$ for all $n$ and $\tau$.

At each moment of time and for each $n \in \Z$ there is a certain rate with which the particles in $(n-1/2, n+1/2)$ change their positions according to rules of the dynamics. This is a local change --- the filling of $(\Z+1/2) \backslash \{ n-1/2, n+1/2 \}$ remains the same. All possible local changes and their rates are shown in Figure \ref{fig:2levAsepII}. These rules completely determine the dynamics.

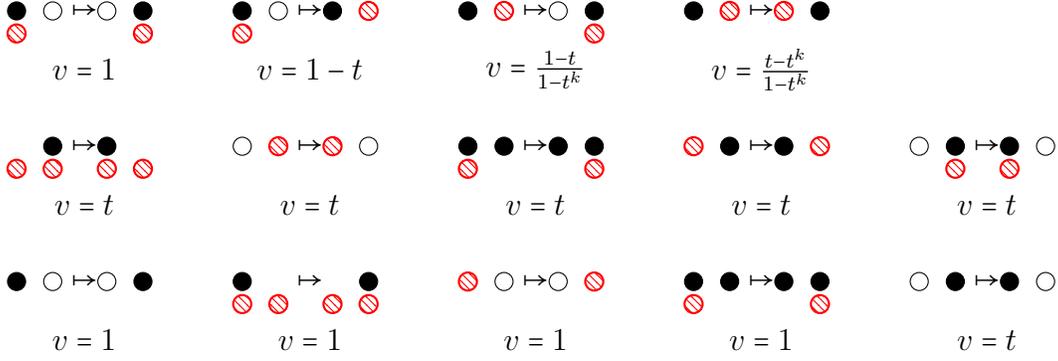
\begin{figure}
\begin{tikzpicture}[>=stealth,scale=0.6]
\draw[black,fill=black] (1,1) circle (0.2);
\draw[black] (1.8,1) circle (0.2);
\node at (2.5,1) {$\mapsto$};
\draw[black] (3,1) circle (0.2);
\draw[black,fill=black] (3.8,1) circle (0.2);
\node at (2.5,-0.3) {$v=1$};
\draw[black,fill=black] (6,1) circle (0.2); \fill[draw=red, thick, pattern=north west lines, pattern color=red] (6,0.5) circle (0.2);
\fill[draw=red, thick, pattern=north west lines, pattern color=red] (6.8,0.5) circle (0.2);
\node at (7.5,1) {$\mapsto$};
\fill[draw=red, thick, pattern=north west lines, pattern color=red] (8,0.5) circle (0.2);
\draw[black,fill=black] (8.8,1) circle (0.2); \fill[draw=red, thick, pattern=north west lines, pattern color=red] (8.8,0.5) circle (0.2);
\node at (7.5,-0.3) {$v=1$};
\fill[draw=red, thick, pattern=north west lines, pattern color=red] (11,1) circle (0.2);
\draw[black] (11.8,1) circle (0.2);
\node at (12.5,1) {$\mapsto$};
\draw[black] (13,1) circle (0.2);
\fill[draw=red, thick, pattern=north west lines, pattern color=red] (13.8,1) circle (0.2);
\node at (12.5,-0.3) {$v=1$};
\draw[black,fill=black] (16,1) circle (0.2); \fill[draw=red, thick, pattern=north west lines, pattern color=red] (16,0.5) circle (0.2);
\draw[black,fill=black] (16.8,1) circle (0.2);
\node at (17.5,1) {$\mapsto$};
\draw[black,fill=black] (18,1) circle (0.2);
\draw[black,fill=black] (18.8,1) circle (0.2); \fill[draw=red, thick, pattern=north west lines, pattern color=red] (18.8,0.5) circle (0.2);
\node at (17.5,-0.3) {$v=1$};
\draw[black] (21,1) circle (0.2);
\draw[black, fill=black] (21.8,1) circle (0.2);
\node at (22.5,1) {$\mapsto$};
\draw[black, fill=black] (23,1) circle (0.2);
\draw[black] (23.8,1) circle (0.2);
\node at (22.5,-0.3) {$v=t$};
\draw[black,fill=black] (1.8,4) circle (0.2); \fill[draw=red, thick, pattern=north west lines, pattern color=red] (1,3.5) circle (0.2);
\fill[draw=red, thick, pattern=north west lines, pattern color=red] (1.8,3.5) circle (0.2);
\node at (2.5,4) {$\mapsto$};
\fill[draw=red, thick, pattern=north west lines, pattern color=red] (3,3.5) circle (0.2);
\draw[black,fill=black] (3,4) circle (0.2); \fill[draw=red, thick, pattern=north west lines, pattern color=red] (3.8,3.5) circle (0.2);
\node at (2.5,2.7) {$v=t$};
\fill[draw=red, thick, pattern=north west lines, pattern color=red] (6.8,4) circle (0.2);
\draw[black] (6,4) circle (0.2);
\node at (7.5,4) {$\mapsto$};
\draw[black] (8.8,4) circle (0.2);
\fill[draw=red, thick, pattern=north west lines, pattern color=red] (8,4) circle (0.2);
\node at (7.5,2.7) {$v=t$};
\draw[black,fill=black] (11,4) circle (0.2); \fill[draw=red, thick, pattern=north west lines, pattern color=red] (11,3.5) circle (0.2);
\draw[black,fill=black] (11.8,4) circle (0.2);
\node at (12.5,4) {$\mapsto$};
\draw[black,fill=black] (13,4) circle (0.2);
\draw[black,fill=black] (13.8,4) circle (0.2); \fill[draw=red, thick, pattern=north west lines, pattern color=red] (13.8,3.5) circle (0.2);
\node at (12.5,2.7) {$v=t$};
\draw[black,fill=black] (16.8,4) circle (0.2);
\fill[draw=red, thick, pattern=north west lines, pattern color=red] (16,4) circle (0.2);
\node at (17.5,4) {$\mapsto$};
\fill[draw=red, thick, pattern=north west lines, pattern color=red] (18.8,4) circle (0.2);
\draw[black,fill=black] (18,4) circle (0.2);
\node at (17.5,2.7) {$v=t$};
\draw[black,fill=black] (21.8,4) circle (0.2); \fill[draw=red, thick, pattern=north west lines, pattern color=red] (21.8,3.5) circle (0.2);
\draw[black] (21,4) circle (0.2);
\node at (22.5,4) {$\mapsto$};
\draw[black] (23.8,4) circle (0.2);
\draw[black,fill=black] (23,4) circle (0.2); \fill[draw=red, thick, pattern=north west lines, pattern color=red] (23,3.5) circle (0.2);
\node at (22.5,2.7) {$v=t$};
\draw[black,fill=black] (1,7) circle (0.2); \fill[draw=red, thick, pattern=north west lines, pattern color=red] (1,6.5) circle (0.2);
\draw[black] (1.8,7) circle (0.2);
\node at (2.5,7) {$\mapsto$};
\draw[black] (3,7) circle (0.2);
\draw[black,fill=black] (3.8,7) circle (0.2); \fill[draw=red, thick, pattern=north west lines, pattern color=red] (3.8,6.5) circle (0.2);
\node at (2.5,5.7) {$v=1$};
\fill[draw=red, thick, pattern=north west lines, pattern color=red] (6,6.5) circle (0.2);
\draw[black,fill=black] (6,7) circle (0.2); \draw[black] (6.8,7) circle (0.2);
\node at (7.5,7) {$\mapsto$};
\fill[draw=red, thick, pattern=north west lines, pattern color=red] (8.8,7) circle (0.2);
\draw[black,fill=black] (8,7) circle (0.2);
\node at (7.5,5.7) {$v=1-t$};
\draw[black,fill=black] (11,7) circle (0.2); \fill[draw=red, thick, pattern=north west lines, pattern color=red] (11.8,7) circle (0.2);
\node at (12.5,7) {$\mapsto$};
\draw[black,fill=black] (13.8,7) circle (0.2); \fill[draw=red, thick, pattern=north west lines, pattern color=red] (13.8,6.5) circle (0.2);
\draw[black] (13,7) circle (0.2);
\node at (12.5,5.7) {$v=\frac{1-t}{1-t^k}$};
\fill[draw=red, thick, pattern=north west lines, pattern color=red] (16.8,7) circle (0.2);
\draw[black,fill=black] (16,7) circle (0.2);
\node at (17.5,7) {$\mapsto$};
\draw[black,fill=black] (18.8,7) circle (0.2);
\fill[draw=red, thick, pattern=north west lines, pattern color=red] (18,7) circle (0.2);
\node at (17.5,5.7) {$v=\frac{t-t^k}{1-t^k}$};
\end{tikzpicture}
\caption{Jump rates of the two-layer ASEP. An empty circle denotes an empty site, while filled black and shaded red circles denote particles of first and second type, respectively. The change happens in points $(n-1/2)$, $(n+1/2)$, and $k = k(n;\tau)$ is a parameter of the configuration.}
\label{fig:2levAsepII}
\end{figure}

Note that the particles of the first type form a Markov dynamics. It is visible from the rules that this dynamics is ASEP in which the rate of a jump to the right is $1$ and the rate of a jump to the left is $t<1$. The evolution of particles of the second type depends on the evolution of particles of the first type. A new feature is that some rates depend on the parameter $k(n;\tau)$.

Let $[z^k] A(z)$ denote the coefficient of $z^k$ in a Laurent series $A(z)$. Results of Section \ref{sec:2} provide the following formulas for observables of these height functions.

\begin{proposition}
\label{prop:2ASEPform2}
Fix an integer $k$, and consider an arbitrary sequence $(s(1), \dots, s(k))$, $s(i) \in \{0,1 \}$, $i=1, \dots, k$. Let $\{ z_{u;f} \}$, $u=1, \dots, k$, $f=1, \dots, s(u)+1$, (so $f$ takes either one or two values depending on $s(u)$) be formal variables.
For any integers $m_1 \ge m_2 \ge \dots \ge m_k$ and $\tau \in \R_{>0}$ the height functions of the two-layer ASEP with step initial condition satisfy
\begin{multline*}
\mathbf E t^{ h_{s(1)} (m_1;\tau) + h_{s(2)} (m_2; \tau)+ \dots + h_{s(k)} (m_k;\tau)} = \left[ \prod_{u=1}^k \prod_{f=1}^{s(u)+1} z_{u;f}^{-1} \right] \prod_{1 \le i<j \le k} \prod_{f=1}^{s(i)+1} \prod_{\hat f=1}^{s(j)+1} \frac{1- z_{i;f}^{-1} z_{j;\hat f}}{1 - t^{-1} z_{i;f}^{-1} z_{j;\hat f}}
\\ \times \prod_{l=1}^k \left( \mathbf{1}_{s(l)=1} \frac{- (z_{l;1} - z_{l;2})^2}{2 z_{l;1}^2 z_{l;2}^2} + \mathbf{1}_{s(l)=0} \frac{1}{z_{l;1}} \right) \prod_{l=1}^k \prod_{f=1}^{s(l)+1} \exp \left( \frac{ - \tau z_{l;f} (1-t)^2}{(1+z_{l;f})(1+t z_{l;f})} \right) \left( \frac{z_{l;f} + t^{-1}}{z_{l;f} +1} \right)^{m_l}.
\end{multline*}
\end{proposition}

If we consider the height function $h_0$ only, this proposition provides formulas which are essentially equivalent to the formulas of \cite{TW1}, \cite{TW2} for ASEP.

\subsection*{Further directions}

There are two ongoing projects which are based on the results of the present paper. The first project studies the structure and the asymptotic behavior of the Hall-Littlewood RSK field in the $t\to 1$ limit. The second one deals with a RSK algorithm in the mixed q-Whittaker / Hall-Littlewood case (see Section \ref{sec:obs-mes} and Remark \ref{rem:mixed}).

\subsection*{Acknowledgements}

We are grateful to A.~Borodin for useful discussions. We are grateful to I.~Corwin for useful comments. 
A.~Bufetov was partially supported by The Foundation Sciences Mathematiques de Paris.

\section{Observables of Macdonald processes}
\label{sec:2}

\subsection{The completed algebra of symmetric functions}

In this section we briefly review some facts about the algebra of symmetric functions. We refer to \cite{M} for a detailed exposition and proofs of mentioned properties. Also, we attempt to mostly follow notations of \cite{BCGS}.

Let $X=(x_1, x_2, \dots)$ be a countable set of formal variables. Define Newton power sums via
$$
p_k := \sum_{i=1}^{\infty} x_i^k, \qquad k=1,2, \dots,
$$
and set the degree of $p_k$ to be equal to $k$.

The algebra of symmetric functions $\Lambda_X$ is the $\Z_{\ge 0}$ graded algebra $\C[1, p_1, p_2,\dots]$; any element $f$ of $\Lambda_X$ can be written as a sum $\sum_{k=0}^{\infty} f_k$, where each $f_k$ is a homogeneous symmetric function of degree $k$ and the sum contains only finitely many non-zero terms. For such a sum we define the lower degree $\mathrm{ldeg} (f)$ as the maximal $N$ such that $f_k=0$ for all $k<N$.

Equip $\Lambda_X$ with the graded topology in which a sequence $\{ f^{(n)} \}_{n \ge 1}$, $f^{(i)} \in \Lambda_X$, is a Cauchy sequence if and only if
$$
\lim_{\min\{m, n\} \to \infty} \mathrm{ldeg}(f^{(m)} - f^{(n)}) = + \infty.$$
The completion of $\Lambda_X$ in the graded topology is referred to as the completed algebra of symmetric functions and is denoted by $\bar \Lambda_X$.

Let $\Y$ be the set of all partitions. We will deal with families of Macdonald symmetric functions $\{ P_{\la} (X;q,t) \}_{\la \in \Y}$, $\{ Q_{\la} (X;q,t) \}_{\la \in \Y}$ and their skew versions $\{ P_{\la / \mu} (X;q,t) \}_{\la, \mu \in \Y}$, $\{ Q_{\la / \mu} (X;q,t) \}_{\la, \mu \in \Y}$ parameterized by two complex numbers $q$ and $t$ (see the definitions of these functions in \cite[Chapter 6]{M}). For two sets of formal variables $X =(x_1, x_2, \dots)$ and $Y=(y_1,y_2, \dots)$ Macdonald functions satisfy the Cauchy identity:
\begin{equation}
\label{eq:Cauchy-simp}
\sum_{\la} P_{\la} (X; q,t) Q_{\la} (Y;q,t) = \prod_{i,j} \frac{(t x_i y_j;q)_{\infty}}{(x_i y_j;q)_{\infty}} =: \p (X,Y;q,t);
\end{equation}
and the identity:
\begin{equation*}
\sum_{\mu} P_{\la / \mu} (X; q,t) P_{\mu} (Y;q,t) = P_{\la} (X,Y; q,t),
\end{equation*}
where
$$
(a;q)_{\infty} := (1-a) (1-aq) (1-aq^2) \dots
$$
is the $q$-Pochhammer symbol, and we denote $P_{\la} (X \cup Y; q,t)$ by $P_{\la} (X,Y; q,t)$.

Let $\omega_{q,t}^X: \bar \Lambda_X \to \bar \Lambda_X$ be the Macdonald automorphism in variables $X$, which satisfies the following properties:
\begin{equation*}
\omega_{q,t}^X \p (X,Y;q,t) = \prod_{i,j} \left(1 + x_i y_j \right), \qquad \omega_{q,t}^X \prod_{i,j} \left(1 + x_i y_j \right) = \p (X,Y;t,q),
\end{equation*}
(note that parameters $t$ and $q$ are swapped in the second formula), and
\begin{equation*}
\omega_{q,t}^{X} P_{\la} (X; q,t) = Q_{\la'} (X; t,q), \qquad \omega_{q,t}^{X} Q_{\la} (X; q,t) = P_{\la'} (X; t,q), \qquad \omega_{t,q} \omega_{q,t}= id,
\end{equation*}
where $\la'$ denotes the transposition of the partition $\la$. In particular, for a single formal variable $w$ we have
\begin{equation}
\label{eq:one-var-inv}
\omega_{q,t}^X \left( \prod_{i} \frac{1- t w x_i}{1 - w x_i} \right) = \omega_{q,t}^X \left( \frac{\p (X,w;q,t)}{\p (X,qw;q,t)} \right) = \prod_{i} \frac{1+ w x_i}{1+q w x_i}.
\end{equation}

\subsection{Macdonald measures and processes}

The \textit{formal Macdonald measure} is a function $\Y \to \bar \Lambda_X \otimes \bar \Lambda_Y$ defined by
\begin{equation}
\mm_{X,Y} (\la) := \frac{ P_{\la} (X;q,t) Q_{\la} (Y;q,t)}{ \p (X, Y; q, t)}.
\end{equation}
Note that the sum of these weights over all $\la$ equals $1$ due to \eqref{eq:Cauchy-simp}, that is why we refer to it as measure.

For two partitions $\la$, $\mu$ let
\begin{equation*}
\Psi_{\la,\mu} ( X,Y) := \sum_{\nu \in \Y} P_{\la / \nu} (X;q,t) Q_{\mu / \nu} (Y;q,t).
\end{equation*}

Fix $2N$ countable sets of formal variables $X_{1}, \ldots, X_{N}, Y_{1}, \ldots, Y_{N}$, let $\mathbf X = (X_1, \dots, X_N)$, $\mathbf Y = (Y_1, \dots, Y_N)$.  The  \textit{formal Macdonald process} is a function
$$
\mathbb Y^{N} \to \bar \Lambda_{X_1} \otimes \cdots \bar \Lambda_{X_N} \otimes \bar \Lambda_{Y_1} \otimes \cdots \bar \Lambda_{Y_N}.
$$
defined by
\begin{equation*}
\mm_{\mathbf X, \mathbf Y} \left(\la^{1}, \dots, \la^{N}\right) := \frac{ P_{\la^{1}} (X_1) \Psi_{\la^2, \la^1} (X_2, Y_1) \Psi_{\la^3, \la^2} (X_3, Y_2) \dots \Psi_{\la^N, \la^{N-1}} (X_N, Y_{N-1}) Q_{\la^N} (Y_N)}{ \prod_{1 \le i \le j \le N} \p \left( X^i;Y^j \right)}.
\end{equation*}
We refer to \cite[Section 3.1]{BCGS} for basic properties of these formal processes.

An algebra homomorphism $\Lambda_X \to \C$ is called a specialization of $\Lambda_X$. A specialization of $\Lambda_X$ is called Macdonald-positive, if its values on all Macdonald functions are nonnegative reals. In this paper we address only the so-called \textit{finite alpha specializations}; such specialization $\rho (a_1, a_2, \dots, a_m)$ is defined by
$$
p_k \mapsto a_1^k+ a_2^k + \dots + a_m^k,
$$
for any (finitely many) $a_1,\dots, a_m \ge 0$. If $0 \leq q, t < 1$, then any finite alpha specialization is Macdonald-positive.

Suppose $\rho_1^+$, $\dots$, $\rho_N^+$, $\rho_1^-, \dots, \rho_N^-$ are finite alpha specializations of $\Lambda_{X_1}, \dots \Lambda_{X_N}$, $\Lambda_{Y_1}, \dots, \Lambda_{Y_N}$, respectively, such that $ab<1$  for any pair $(a,b)$, where $a$ is a parameter of some $\rho_i^+$ and $b$ is a parameter of some $\rho_j^-$. Then one can consider the specialization
$$
\mm_{\mathbf X, \mathbf Y} \left(\la^{1}, \dots, \la^{N}\right) \xrightarrow[\rho_1^+, \dots, \rho_N^+, \rho_1^-, \dots, \rho_N^-]{} MM_{\rho_1^+, \dots, \rho_N^+, \rho_1^-, \dots, \rho_N^-} \left(\la^{1}, \dots, \la^{N}\right) \in \R.
$$

It is known (see, e.g., \cite{BC})  that collection of numbers $\left \{MM_{\rho_1^+, \dots, \rho_N^+, \rho_1^-, \dots, \rho_N^-} \left(\la^{1}, \dots, \la^{N}\right) \right \}$ is well defined under our assumptions  and provides a probability measure on sequences (of length $N$) of partitions --- it is referred to as a Macdonald process (without the word ``formal''). Probabilistically, one is interested in properties of these measures. However, as we will see, it is more efficient to study first a more general formal setting, and only then extract results about probabilistic objects by applying specializations.

\subsection{Observables of formal Macdonald measures}

The main tool (introduced in \cite{BC} and used in \cite{BCGS}) in the study of the Macdonald processes is the method to compute its observables by using Macdonald difference operators. In this and subsequent sections we show how the combination of this idea with the use of the Macdonald involution $w_{q,t}$ leads to new probabilistic results.

Let us set
$$
\mathcal{E}_r (\la) := \sum_{1 \le i_1 < i_2 < \dots < i_r} q^{-\la_{i_1} - \la_{i_2} - \dots - \la_{i_r} } t^{(i_1-1) + (i_2-1) \dots + (i_r-1)}, \qquad \la \in \Y, r \in \N,
$$
where the indices $i$'s range over positive integers yet the sum converges due to $0 \leq  t <1$.

For a Laurent power series
$$A(w_{1}, \ldots, w_{r}) = \sum_{(k_{1}, \ldots, k_{r}) \in \mathbb{Z}^{r}} A_{k_{1}, \ldots, k_{r}} w_{1}^{k_{1}}\cdots w_{r}^{k_{r}}$$ we denote  the coefficient $A_{k_{1}, \ldots, k_{r}}$ by $\left[w_{1}^{k_{1}}\cdots w_{r}^{k_{r}}\right]A(w_{1}, \ldots, w_{r})$. For operations with Laurent power series we adopt the usual convention:
\begin{equation}
\label{eq:Laur-conv}
\frac{f}{1-g} = f(1+g+g^{2}+g^{3}+\cdots).
\end{equation}

The following proposition was proved in \cite[Proposition 3.8]{BCGS}, see also \cite[Section 9]{Sh} and \cite[Proposition 3.6]{FHHSY} for similar statements.

\begin{proposition}
\label{prop:obs-meas}
We have
\begin{multline}
\label{eq:form-var-mes-1}
\sum_{\la \in \mathbb Y} \mathcal E_r (\la) P_{\la} (X; q,t) Q_{\la} (Y;q,t) = \Pi (X,Y;q,t) [ w_1^{-1} \dots w_r^{-1} ] \frac{t^{r(r-1)/2}} {r! (1-t)^r w_1 \dots w_r} \\ \times \prod_{k \ne l} \frac{1 - w_k^{-1} w_l}{1 - t w_k^{-1} w_l} \prod_{s=1}^r \left( \prod_{i \ge 1} \frac{1- t w_s x_i}{1-w_s x_i} \prod_{j \ge 1} \frac{1- t q^{-1} w_s^{-1} y_j}{1- q^{-1} w_s^{-1} y_j} \right),
\end{multline}
\end{proposition}

\begin{remark}
It is an important part of Proposition \ref{prop:obs-meas} that according to \eqref{eq:Laur-conv} we expand the right-hand side of \eqref{eq:form-var-mes-1} into a Laurent series in $w$ variables with the use of
\begin{equation*}
\frac{1}{w_k - t w_l} = w_k^{-1} \sum_{a=0}^{\infty} t^a w_l^a w_k^{-a},
\end{equation*}
and expanding
\begin{equation*}
\prod_{i \ge 1} \frac{1- t w_s x_i}{1-w_s x_i} , \qquad  \frac{1- t q^{-1} w_s^{-1} y_j}{1- q^{-1} w_s^{-1} y_j},
\end{equation*}
into a power series with non-negative powers of $(w_s x_i)$ and $(q^{-1} w_s^{-1} y_j)$, respectively. We apply \eqref{eq:Laur-conv} in all subsequent formulas as well.
\end{remark}

Note that both sides of equation \eqref{eq:form-var-mes-1} belong to $\bar \Lambda_X \otimes \bar \Lambda_Y$. Then the application of $\omega_{q,t}^X$  and the use of \eqref{eq:one-var-inv} immediately gives the following statement.

\begin{proposition}
We have
\begin{multline}
\label{eq:form-var-mes-2}
\sum_{\la \in \mathbb Y} \mathcal E_r (\la) Q_{\la'} (X; t,q) Q_{\la} (Y;q,t) = \prod_{i,j \ge 1} \left(1 + x_i y_j \right) [ w_1^{-1} \dots w_r^{-1} ] \frac{t^{r(r-1)/2}} {r! (1-t)^r w_1 \dots w_r} \\ \times \prod_{k \ne l} \frac{w_k - w_l}{w_k - t w_l} \prod_{s=1}^r \left( \prod_{i \ge 1} \frac{1 + w_s x_i}{1 + q w_s x_i} \prod_{j \ge 1} \frac{1- t q^{-1} w_s^{-1} y_j}{1- q^{-1} w_s^{-1} y_j} \right),
\end{multline}
\end{proposition}

Alternatively, we can apply $\omega_{q,t}^Y$ to both sides of \eqref{eq:form-var-mes-1} and obtain
\begin{proposition}
We have
\begin{multline}
\label{eq:form-var-mes-3}
\sum_{\la \in \mathbb Y} \mathcal E_r (\la) P_{\la} (X; q,t) P_{\la'} (Y;t,q) = \prod_{i,j \ge 1} \left(1 + x_i y_j \right) [ w_1^{-1} \dots w_r^{-1} ] \frac{t^{r(r-1)/2}} {r! (1-t)^r w_1 \dots w_r} \\ \times \prod_{k \ne l} \frac{w_k - w_l}{w_k - t w_l} \prod_{s=1}^r \left( \prod_{i \ge 1} \frac{1- t w_s x_i}{1-w_s x_i} \prod_{j \ge 1} \frac{1 + q^{-1} w_s^{-1} y_j}{1 + w_s^{-1} y_j} \right),
\end{multline}
\end{proposition}

Finally, applying $\omega_{q,t}^Y$ to \eqref{eq:form-var-mes-2} we get
\begin{proposition}
We have
\begin{multline}
\label{eq:form-var-mes-4}
\sum_{\la \in \mathbb Y} \mathcal E_r (\la) Q_{\la'} (X; t,q) P_{\la'} (Y;t,q) = \Pi (X,Y;t,q) [ w_1^{-1} \dots w_r^{-1} ] \frac{t^{r(r-1)/2}} {r! (1-t)^r w_1 \dots w_r} \\ \times \prod_{k \ne l} \frac{w_k - w_l}{w_k - t w_l} \prod_{s=1}^r \left( \prod_{i \ge 1} \frac{1 + w_s x_i}{1 + q w_s x_i} \prod_{j \ge 1} \frac{1 + q^{-1} w_s^{-1} y_j}{1 + w_s^{-1} y_j} \right),
\end{multline}
\end{proposition}

\subsection{Observables of mixed and Hall-Littlewood measures}
\label{sec:obs-mes}

The \textit{Hall-Littlewood function} is the Macdonald function $P_{\la} (X; q,t)$ with parameter $q$ set to $0$. Similarly, the \textit{q-Whittaker function} is the Macdonald function $P_{\la} (X; q,t)$ with parameter $t$ set to $0$. When one of the parameters is set to $0$, the observables computed in the previous subsection simplify and provide a significant amount of information about the corresponding measures.

Note that as $t \to 0$ we have
$$
\mathcal{E}_r (\la) = q^{-\la_1 - \la_2 - \dots - \la_r} t^{r(r-1)/2} + o \left( t^{r(r-1)/2} \right).
$$
$$
P_{\la} (X;q,t) = P_{\la} (X;q,0) + o(1), \qquad Q_{\la} (Y;q,t) = Q_{\la} (Y;q,0) + o(1)
$$

Therefore, after dividing both sides of equations \eqref{eq:form-var-mes-1} -- \eqref{eq:form-var-mes-4} by $t^{r(r-1)/2}$ and taking the limit $t \to 0$, we obtain the following statement.

\begin{proposition}
\label{prop:meas-ttt}
We have
\begin{multline}
\label{eq:form-var-mes-11}
\sum_{\la \in \mathbb Y} q^{-\la_1 - \la_2 - \dots - \la_r} P_{\la} (X; q,0) Q_{\la} (Y;q,0) = \p (X,Y;q,0) [ w_1^{-1} \dots w_r^{-1} ] \frac{ (-1)^{r(r-1)/2} \prod_{k<l} (w_k-w_l)^2} {r! w_1^r \dots w_r^r} \\ \times \prod_{s=1}^r \left( \prod_{i \ge 1} \frac{1}{1-w_s x_i} \prod_{j \ge 1} \frac{1}{1- q^{-1} w_s^{-1} y_j} \right),
\end{multline}
\begin{multline}
\label{eq:form-var-mes-22}
\sum_{\la \in \mathbb Y} q^{-\la_1 - \la_2 - \dots - \la_r} Q_{\la'} (X; 0,q) Q_{\la} (Y;q,0) = \prod_{i,j \ge 1} \left(1 + x_i y_j \right) [ w_1^{-1} \dots w_r^{-1} ] \frac{ (-1)^{r(r-1)/2} \prod_{k<l} (w_k-w_l)^2} {r! w_1^r \dots w_r^r} \\ \times \prod_{s=1}^r \left( \prod_{i \ge 1} \frac{1 + w_s x_i}{1 + q w_s x_i} \prod_{j \ge 1} \frac{1}{1- q^{-1} w_s^{-1} y_j} \right),
\end{multline}
\begin{multline}
\label{eq:form-var-mes-33}
\sum_{\la \in \mathbb Y} q^{-\la_1 - \la_2 - \dots - \la_r} P_{\la} (X; q,0) P_{\la'} (Y;0,q) = \prod_{i,j \ge 1} \left(1 + x_i y_j \right) [ w_1^{-1} \dots w_r^{-1} ] \frac{ (-1)^{r(r-1)/2} \prod_{k<l} (w_k-w_l)^2} {r! w_1^r \dots w_r^r} \\ \times \prod_{s=1}^r \left( \prod_{i \ge 1} \frac{1}{1-w_s x_i} \prod_{j \ge 1} \frac{1 + q^{-1} w_s^{-1} y_j}{1 + w_s^{-1} y_j} \right),
\end{multline}
\begin{multline}
\label{eq:form-var-mes-44}
\sum_{\la \in \mathbb Y} q^{-\la_1 - \la_2 - \dots - \la_r} Q_{\la'} (X; 0,q) P_{\la'} (Y;0,q) = \p (X,Y;0,q) [ w_1^{-1} \dots w_r^{-1} ] \frac{ (-1)^{r(r-1)/2} \prod_{k<l} (w_k-w_l)^2} {r! w_1^r \dots w_r^r} \\ \times \prod_{s=1}^r \left( \prod_{i \ge 1} \frac{1 + w_s x_i}{1 + q w_s x_i} \prod_{j \ge 1} \frac{1 + q^{-1} w_s^{-1} y_j}{1 + w_s^{-1} y_j} \right),
\end{multline}

\end{proposition}

Equation \eqref{eq:form-var-mes-11} was derived in \cite{BC}. Equations \eqref{eq:form-var-mes-22}--\eqref{eq:form-var-mes-44} seem to be published for the first time\footnote{ After the completion of this project, we have learned that equation \eqref{eq:form-var-mes-44} was present in unpublished notes of Shamil Shakirov.}.

Let us now reformulate Proposition \ref{prop:meas-ttt} in a probabilistic setting.

Assume that $x_1, \dots, x_M$ and $y_1, \dots, y_N$ are non-negative reals such that $x_i y_j <1$, for $1 \le i \le M$, $1 \le j \le N$. Then the corresponding \textit{Hall-Littlewood probability measure} is given by
\begin{equation}
\label{eq:HL-meas-def}
\mathbb{HL} (\la) := \prod_{i=1}^M \prod_{j = 1}^N \frac{1- x_i y_j}{1 - t x_i y_j} P_{\la} (x_1, \dots, x_M;0,t) Q_{\la} (y_1, \dots, y_N;0,t), \qquad \la \in \Y.
\end{equation}
Similarly, we can define the corresponding \textit{mixed} Hall-Littlewood / q-Whittaker probability measure via
\begin{equation}
\label{eq:mixed-meas-def}
\mathbb{WHL} (\la) := \prod_{i=1}^M \prod_{j=1}^N (1+x_i y_j)^{-1} P_{\la} (x_1, \dots, x_M;q,0) P_{\la'} (y_1, \dots, y_N;0,q), \qquad \la \in \Y.
\end{equation}
One of the properties of the Macdonald functions (see \cite[Chapter 6, Equation (5.3)]{M}) implies that we also have
\begin{equation}
\label{eq:transpos-mix}
\mathbb{WHL} (\la) = \prod_{i=1}^M \prod_{j=1}^N (1+x_i y_j)^{-1} Q_{\la} (x_1, \dots, x_M;q,0) Q_{\la'} (y_1, \dots, y_N;0,q), \qquad \la \in \Y.
\end{equation}

\begin{proposition}
\label{prop:qWHLandHL}

a) Assume that $\la$ is a random partition distributed according to the mixed measure \eqref{eq:mixed-meas-def}. We have
\begin{multline}
\label{eq:mix-meas-obs}
\langle q^{-\la_1 - \dots - \la_r} \rangle_{\mathbb{WHL}} = 
\frac{1}{(2 \pi \ii)^r} \oint_{w_1 \in \mathcal C} \dots \oint_{w_r \in \mathcal C} \frac{(-1)^{r(r-1)/2} \prod_{k <l} (w_k-w_l)^2} {r! w_1^r \dots w_r^r} \\ \times \prod_{s=1}^r \left( \prod_{i \ge 1} \frac{1}{1-w_s x_i} \prod_{j \ge 1} \frac{w_s + q^{-1} y_j}{w_s + y_j} \right) d w_1 \dots d w_r,
\end{multline}
where the contour of integration $\mathcal C$ (common for all variables) encircles $(-y_j)$'s, $0$, 
and no other poles of the integrand.

b) Assume that $\la$ is a random partition distributed according to the Hall-Littlewood measure \eqref{eq:HL-meas-def}. We have
\begin{multline}
\label{eq:HL-meas-obs}
\langle t^{-\la'_1 - \dots - \la'_r} \rangle_{\mathbb{HL}} = \frac{1}{(2 \pi \ii)^r} \oint_{w_1 \in \mathcal C} \dots \oint_{w_r \in \mathcal C} \frac{(-1)^{r(r-1)/2} \prod_{k <l} (w_k-w_l)^2} {r! w_1^r \dots w_r^r} \\ \times \prod_{s=1}^r \left( \prod_{i \ge 1} \frac{1 + w_s x_i}{1 + t w_s x_i} \prod_{j \ge 1} \frac{w_s + t^{-1} y_j}{w_s + y_j} \right) d w_1 \dots d w_r,
\end{multline}
where the contour of integration (common for all variables) encircles $(-y_j)$'s, $0$, and no other poles of the integrand.
\end{proposition}
\begin{proof}
a) 
The equality is a corollary of \eqref{eq:form-var-mes-3}; the conditions on contours are obtained from a specific way how we expand the right-hand side of \eqref{eq:form-var-mes-3} into a Laurent series.

b) This statement follows from \eqref{eq:form-var-mes-4}; we replaced $q$ from \eqref{eq:form-var-mes-4} into $t$, since it is conventional to denote the parameter of a Hall-Littlewood function as $t$. Again, the conditions on contours are obtained from a specific way how we expand the right-hand side of \eqref{eq:form-var-mes-4} into a Laurent series.

\end{proof}

\begin{remark}
Equation \eqref{eq:form-var-mes-11} (which deals with q-Whittaker measures defined analogous to Hall-Littlewood measures) was actively used in \cite{BC} and subsequently in \cite{BCF}, \cite{BCR}, and many other works for an asymptotic analysis of a number of probabilistic models. Formulas \eqref{eq:form-var-mes-22}--\eqref{eq:form-var-mes-44} as well as \eqref{eq:HL-meas-obs}, \eqref{eq:mix-meas-obs} might allow a similar analysis of several new models. Also, as we will see in Section \ref{sec:degen}, Hall-Littlewood processes allow to obtain many formulas for observables of ASEP and stochastic six-vertex model. These provides a new approach to the formulas obtained for ASEP in \cite{TW1}, \cite{TW2}, see also \cite{BCS}, and for a stochastic six vertex model in \cite{BCG} and \cite{BP2}.


\end{remark}

\begin{remark}
\label{rem:mixed}
It is interesting to note that while the probabilistic models related to q-Whittaker measures and
and to Hall-Littlewood measures are already widely present in the literature, the case of a mixed measure seems to lead to new probabilistic models.

In the Schur case (($q=t=0$) one of interesting examples of the mixed measure is given by the domino tilings of the Aztec diamond. In the case of nontrivial $q$ we obtain a new $q$-deformation of the measure on the set of domino tilings that leads to a mixed measure \eqref{eq:mixed-meas-def}. Using \eqref{eq:mix-meas-obs} it might be possible to obtain KPZ asymptotics of this measure (we do not address the asymptotics in the present paper; but a number of similar computations were performed in recent years, see \cite{BC}, \cite{BCF}, \cite{BCR}).
\end{remark}

\subsection{Observables of formal Macdonald processes}

Using the same idea as in Section \ref{sec:obs-mes} one can obtain much more general observables of Hall-Littlewood processes. In order to formulate the results, let us introduce one more piece of notation. For two sets of variables $U$ and $V$ set
\begin{equation}
W(U;V) := \prod_{i,j} \frac{ (1 - t u_i v_j) (1 - q u_i v_j)}{(1 - u_i v_j) (1 - q t u_i v_j)},
\end{equation}
which is understood as a formal power series in nonnegative powers of $(u_i v_j)$.

We will use $2 N$ countable sets of formal variables $\mathbf X = (X_1, \dots, X_N)$, $\mathbf Y = (Y_1, \dots, Y_N)$, where $X_a = (x_{1;a}, x_{2,a}, \dots)$, $Y_a = (y_{1;a}, y_{2,a}, \dots)$.

The following result was proved in \cite[Theorem 3.10 and Corollary 3.11]{BCGS}.

\begin{proposition}
\label{prop:gen-Mac-form}
Let $\{ \mm_{\mathbf X, \mathbf Y} (\la^{1}, \dots, \la^{N}) \}_{\la^{i} \in \Y; 1\le i \le N}$ be a formal Macdonald process, let $0<M \le N$, and let $r_1, \dots, r_M$ be $M$ positive integers. For $1 \le m \le N$ set $V^m := \{v_{1;m}, \dots, v_{r_m;m} \}$, and define
$$
DV^m := \frac{t^{r_m (r_m-1)/2}}{r_m ! (1-t)^{r_m} v_{1;m} \dots v_{r_m;m}} \prod_{1 \le i \ne j \le r_m} \frac{ 1 - v_{i;m}^{-1} v_{j;m}}{ 1 - t v_{i;m}^{-1} v_{j;m}}
$$
We have
\begin{multline}
\sum_{\la^1, \dots, \la^N} \mathcal{E}_{r_1} (\la^1) \dots \mathcal{E}_{r_M} (\la^M) \mm_{\mathbf X, \mathbf Y} \left(\la^{1}, \dots, \la^{N}\right) = [ v_{1;1}^{-1} \dots v_{r_1;1}^{-1} v_{1;2}^{-1} \dots v_{1;M}^{-1} \dots v_{r_M; M}^{-1} ] \prod_{m=1}^M DV^m
\\ \times \prod_{1 \le \alpha < \beta \le M} \left( \prod_{s=1}^{r_{\beta}} \prod_{i \ge 1} \frac{1- t v_{s;\beta} x_{i;\alpha}}{1-v_{s;\beta} x_{i;\alpha}} \right) \left( \prod_{s=1}^{r_{\alpha}} \prod_{j \ge 1} \frac{1- t q^{-1} v_{s;\alpha}^{-1} y_{j;\beta} }{1- q^{-1} v_{s;\alpha}^{-1} y_{j;\beta}} \right) W ( (qV^{\alpha})^{-1};V^{\beta}),
\end{multline}
where for a set of variables $V = \{v_1, \dots, v_r \}$, $(q V)^{-1}$ stands for the set $\{(q v_1)^{-1}, \dots, (q v_r)^{-1} \}$.
\end{proposition}

Proposition \ref{prop:gen-Mac-form} claims an equality in $\bar \Lambda_{X_1} \otimes \dots \otimes \bar \Lambda_{X_N} \otimes \bar \Lambda_{Y_1} \otimes \dots \otimes \bar \Lambda_{Y_N}$. The application of $\omega_{q,t}$ in all these $2N$ sets of variables immediately gives the following proposition.

\begin{proposition}
\label{prop:gen-Mac-form-dual}
In notations of Proposition \ref{prop:gen-Mac-form} we have
\begin{multline}
\sum_{\la^1, \dots, \la^N} \mathcal{E}_{r_1} \left(\la^1\right) \dots \mathcal{E}_{r_M} \left(\la^M\right) \mm_{\mathbf X, \mathbf Y} \left(\la^{1}, \dots, \la^{N}; t,q \right) = [ v_{1;1}^{-1} \dots v_{r_1;1}^{-1} v_{1;2}^{-1} \dots v_{1;M}^{-1} \dots v_{r_M; M}^{-1} ] \prod_{m=1}^M DV^m
\\ \times \prod_{1 \le \alpha < \beta \le M} \left( \prod_{s=1}^{r_{\beta}} \prod_{i \ge 1} \frac{1 + v_{s;\beta} x_{i;\alpha}}{1 + q v_{s;\beta} x_{i;\alpha}} \right) \left( \prod_{s=1}^{r_{\alpha}} \prod_{j \ge 1} \frac{1 + q^{-1} v_{s;\alpha}^{-1} y_{j;\beta} }{1 + v_{s;\alpha}^{-1} y_{j;\beta}} \right) W ( (qV^{\alpha})^{-1};V^{\beta}),
\end{multline}

\end{proposition}

By making $t \to 0$ transition, we get
\begin{theorem}
\label{prop:gen-HL-proc}
In notations of Proposition \ref{prop:gen-Mac-form} we have
\begin{multline}
\sum_{\la^1, \dots, \la^N} q^{-\la_1^1 - \dots - \la_{r_1}^1} \dots q^{-\la_1^M - \dots - \la_{r_M}^M} \mm_{\mathbf X, \mathbf Y} (\la^{1}, \dots, \la^{N}; 0,q) = [ v_{1;1}^{-1} \dots v_{r_1;1}^{-1} v_{1;2}^{-1} \dots v_{1;M}^{-1} \dots v_{r_M; M}^{-1} ] \\ \times \prod_{m=1}^M \frac{(-1)^{r_m(r_m-1)/2} \prod_{1 \le i<j \le r_m} (v_{i;r_m} - v_{j;r_m})^2 }{r_m ! v_{1;m}^{r_m} v_{2;m}^{r_m} \dots v_{r_m;m}^{r_m} }
\prod_{1 \le \alpha < \beta \le M} \left( \prod_{s=1}^{r_{\beta}} \prod_{i \ge 1} \frac{1 + v_{s;\beta} x_{i;\alpha}}{1 + q v_{s;\beta} x_{i;\alpha}} \right) \left( \prod_{s=1}^{r_{\alpha}} \prod_{j \ge 1} \frac{1 + q^{-1} v_{s;\alpha}^{-1} y_{j;\beta} }{1 + v_{s;\alpha}^{-1} y_{j;\beta}} \right)
\\ \times \prod_{i=1}^{r_{\alpha}} \prod_{j=1}^{r_{\beta}} \frac{1 - v_{i;\alpha}^{-1} v_{j;\beta}}{1 - q^{-1} v_{i;\alpha}^{-1} v_{j;\beta}},
\end{multline}

\end{theorem}

Note that similarly to the transition between \eqref{eq:form-var-mes-44} and \eqref{eq:HL-meas-obs}, Theorem \ref{prop:gen-HL-proc} can be interpreted as a statement about observables of first columns of Hall-Littlewood processes. We will not rewrite it in such a form since it is straightforward.

\begin{remark}
In particular, for a Hall-Littlewood measure \eqref{eq:HL-meas-def} Theorem \ref{prop:gen-HL-proc} allows to compute expectations of any product of expressions of the form $t^{-\la'_1 - \dots - \la'_r}$.
The knowledge of such observables provides a lot of information about the random signatures. In particular, for any fixed $r$ these observables completely determine the distribution of $\la'_1, \dots, \la'_r$ if we \textit{a priori} know that $\la'_1 \le N$ for a certain $N \in \Z_{\ge 0}$ (which is often the case). The reason for this is that we know all joint moments of random variables $t^{N-\la'_1}, t^{2N -\la'_1 - \la'_2}, \dots, t^{rN - \la'_1 - \dots - \la'_r} $ which take values between 0 and 1 (thus, they are determined by their moments).
\end{remark}

\section{Hall-Littlewood RSK algorithm}
\label{sec:RSK}

\subsection{Preliminaries}
\label{sec:3-prelim}

In the remaining sections we will deal only with the Hall-Littlewood polynomials (HL for short), so we shall denote them by $P_{\la}$, $Q_{\la}$, omitting the dependence on the Macdonald parameter $t$ (and $q=0$).

A \textit{signature} $\la$ is a finite non-increasing sequence of integers $\la_1 \ge \la_2 \ge \dots \ge \la_N$; $N$ is called the \textit{length} of $\la$. Let $|\la| := \la_1 + \dots + \la_N$. Denote by $\Sig_N$ the set of all signatures of length $N$. We will say that two signatures $\la \in \Sig_{N-1}$ and $\mu \in \Sig_N$ \textit{interlace} (notation $\la \prec \mu$) if $\mu_1 \ge \la_1 \ge \mu_2 \ge \dots \ge \la_{N-1} \ge \mu_N$. In a similar vein, we will say that two signatures $\la \in \Sig_{N}$ and $\mu \in \Sig_N$ \textit{interlace} (same notation $\la \prec \mu$) if $\mu_1 \ge \la_1 \ge \mu_2 \ge \dots \ge \la_{N-1} \ge \mu_N \ge \la_N$. For a signature $\la \in \Sig_{N}$ we will use the convention that $\la_{k} = -\infty$ for $k > N$ and $\lambda_{0} = +\infty$.

Note that any Young diagram $\la \in \Y$ for any $N \geq \la_{1}'$ can be naturally identified with a signature of length $N$ by adding $N-\la'_{1}$ zeros to the sequence of its row lengths. On the other hand, any signature with nonnegative parts can be naturally identified with a Young diagram. For the rest of the paper we will mostly use signatures of large length instead of Young diagrams.

Following \cite{M}, define 
\begin{equation}
\label{eq:def-psi}
\psi_{\mu/\la}: = 1_{\la \prec \mu} \cdot \prod_{1 \leq k \leq j \leq i: \ \mu_{j+1} < \la_{j} = \la_{k} < \mu_{k}} \left(1-t^{j-k+1}\right), \qquad \ \mu \in \Sig_i, \quad  \la \in \Sig_{i-1} \ \text{or} \ \Sig_{i},
\end{equation}
\begin{equation}
\label{eq:def-phi}
\phi_{\mu/\la}: = 1_{\la \prec \mu} \cdot \prod_{1 \leq k \leq j \leq i: \ \la_{j} < \mu_{j} = \mu_{k} < \la_{k-1}} \left(1-t^{j-k+1}\right), \qquad \ \mu \in \Sig_i, \quad  \la \in \Sig_{i-1} \ \text{or} \ \Sig_{i}.
\end{equation}
Also, if $\la \subset \mu$ are Young diagrams, then define $\psi_{\mu/\la}, \phi_{\mu/\la}$ by considering $\mu \in \Sig_{N}$ and $\la \in \Sig_{N-1}$, where $N-1$ is the number of non-zero rows in $\la$, and applying \eqref{eq:def-psi}, \eqref{eq:def-phi}.

\begin{remark}
\label{rem:YDvsS}
It is important for this definition (for Young diagrams) that we use signatures of different lengths. The usage of signatures of the same length will not give conventional Hall-Littlewood functions.
\end{remark}

For $\la \prec \mu$, $\la \in \Sig_k$, $\mu \in \Sig_N$ ($k < N$), combinatorial formulae for the skew Hall-Littlewood polynomials read
\begin{equation}
\label{eq:HLcombFormQ}
Q_{\mu / \la} (x_1, \dots, x_{N-k}) = \sum_{\la = \la^{(k)} \prec \la^{(k+1)} \prec \dots \prec \la^{(N-1)} \prec \la^{N} = \mu} \prod_{i=k+1}^N \phi_{ \la^{(i)} / \la^{(i-1)}} x_{i-k}^{|\la^{(i)}| - |\la^{(i-1)}|},  \quad \la^{(i)} \in \Sig_i,
\end{equation}
\begin{equation}
\label{eq:HLcombFormP}
P_{\mu / \la} (x_1, \dots, x_{N-k}) = \sum_{\la = \la^{(k)} \prec \la^{(k+1)} \prec \dots \prec \la^{(N-1)} \prec \la^{N} = \mu} \prod_{i=k+1}^N \psi_{ \la^{(i)} / \la^{(i-1)}} x_{i-k}^{|\la^{(i)}| - |\la^{(i-1)}|}, \quad \la^{(i)} \in \Sig_i.
\end{equation}

We will need the (generalised) Cauchy identity for skew HL functions (see \cite[Chapter 3.5]{M}). Typically, this identity is stated in terms of Young diagrams; however, we will use statements in terms of signatures. It is straightforward to check that these variations \eqref{eq:skew-Cauchy-1varA}-- \eqref{eq:skew-Cauchy-1varBB} are equivalent to the conventional generalised Cauchy identity (we will also give independent proofs of \eqref{eq:skew-Cauchy-1varA}-- \eqref{eq:skew-Cauchy-1varBB} in Section \ref{sec:bij-proof}).

The first statement is that for formal variables $x$,$y$ and for any $\mu \in \Sig_N$, $\nu \in \Sig_{N-1}$ one has
\begin{equation}
\label{eq:skew-Cauchy-1varA}
\frac{1-txy}{1-xy} \sum_{\la \in Sig_{N-1}: \la \prec \mu, \la \prec \nu} P_{\mu / \la} (x) Q_{\nu / \la} (y) = \sum_{\rho \in Sig_{N}: \mu \prec \rho, \nu \prec \rho} P_{\rho / \nu} (x) Q_{\rho / \mu} (y);
\end{equation}

Note that the summation on the right-hand side of \eqref{eq:skew-Cauchy-1varA} is infinite ($\rho_{1}$ can be arbitrarily large), while the summation on the left-hand side  is finite (since $\mu_{N} \leq \la_{N-1} \leq \la_{1} \leq \mu_{1}$). So for arbitrary $k,\ell \in \Z_{\ge 0}$ the coefficient of $x^k y^\ell$ is the same on both sides. Thus, \eqref{eq:skew-Cauchy-1varA} can be rewritten in an equivalent form:
\begin{equation}
\label{eq:skew-Cauchy-1varAA}
\sum_{r \ge 0} ( 1 - t 1_{r \ge 1}) \sum_{\la \in Sig_{N-1}: \la \prec \mu, \la \prec \nu, |\mu|- |\la| = k-r, |\nu|-|\la| = \ell-r} \phi_{\nu / \la} \psi_{\mu / \la} = \sum_{\rho \in Sig_{N}: \mu \prec \rho, \nu \prec \rho, |\rho|-|\nu|=k, |\rho|-|\mu| = \ell} \psi_{\rho / \nu} \phi_{\rho / \mu}.
\end{equation}
A minor variation of this formula states that for any $\mu \in \Sig_N$ and $\nu \in \Sig_{N}$ we have
\begin{equation}
\label{eq:skew-Cauchy-1varBB}
\sum_{\la \in Sig_{N}: \la \prec \mu, \la \prec \nu, |\mu|- |\la| = k, |\nu|-|\la| = \ell} \phi_{\nu / \la} \psi_{\mu / \la} = \sum_{\rho \in Sig_{N}: \mu \prec \rho, \nu \prec \rho, |\rho|-|\nu| = k, |\rho|-|\mu|=\ell} \psi_{\rho / \nu} \phi_{\rho / \mu}.
\end{equation}
Note that all summations in both \eqref{eq:skew-Cauchy-1varAA} and  \eqref{eq:skew-Cauchy-1varBB} are finite (for fixed $k$ and $\ell$).

\subsection{Significance of randomization. Signature/ Young diagrams variations of RSK}
\label{sec:RSKgoal}

The classical RSK (Robinson-Schensted-Knuth) algorithm was introduced in \cite{R}, \cite{She}, \cite{K}.
As was mentioned in the introduction, a {\it Fomin growth diagram} is a map $U: \Y \times \Y \times \Y \times \Z_{\ge 0} \mapsto \Y$, $(\la, \mu, \nu, r) \mapsto U^r( \la;\mu, \nu)$ which takes as an input a triple of Young diagrams $\la, \mu, \nu$ such that $\mu / \la$ and $\nu / \la$ are horizontal strips, and $r \in \Z_{\ge 0}$, and outputs a single Young diagram. This map has many remarkable properties. One of its most important applications is that it provides a bijective proof of the (generalized) skew Cauchy identity \eqref{eq:skew-Cauchy-1varAA} for the Schur case $t=0$ (the usage of signatures in \eqref{eq:skew-Cauchy-1varAA} in this case is equivalent to the usage of Young diagrams).

Our goal is to naturally generalize this map to the setting of Hall-Littlewood functions with $0 \le t<1$. It is not immediate to see how to properly generalize the most important properties of the classical RSK. In particular, for general $t$ the identity \eqref{eq:skew-Cauchy-1varAA} cannot be proved bijectively in a naive way. Though it still contains the same number of summands on both sides, the sets of weights are not  the same in general.

We believe that the ``correct'' way to get a bijective proof of \eqref{eq:skew-Cauchy-1varAA} is to make an output of the algorithm \textit{random}. Then an algorithm will be constructed by defining for each $\rho$ a nonnegative coefficient $U^r (\mu \to \rho \mid \la \to \nu)$, which is the probability that the output of the algorithm is $\rho$. Then we must have
$$
\sum_{\rho} U^r (\mu \to \rho \mid \la \to \nu) =1, \qquad \text{for any valid $\la, \mu, \nu, r$}.
$$
Such construction provides a ``bijective'' (in a quite natural sense) proof of the Cauchy identity, see Section \ref{sec:bij-proof}.

Random RSK algorithms for generalizations of Schur functions were constructed in \cite{OCP}, \cite{BP3} in a ``continuous time'' setting (in our notations one may interpret it as the special case with $|\nu|-|\la|=1$), and in \cite{MP} in a general ``discrete time'' setting (arbitrary $\nu,\la, \mu$) for the $q$-Whittaker functions. Also, a continuous time random RSK algorithm in the Hall-Littlewood case was constructed in \cite{BP}. In this section we present a ``discrete time'' RSK for the Hall-Littlewood functions.








We will consider two closely related yet slightly different settings. First, we will define a probabilistic algorithm (that is, coefficients $U (\mu \to \rho \mid \la \to \nu)$) for signatures $\la, \mu, \nu, \rho \in \Sig_N$ such that $|\nu|-|\la| = |\rho|-|\mu|$. It can be thought of as a bijective proof of a skew Cauchy identity in the form \eqref{eq:skew-Cauchy-1varBB}.

Next, for signatures $\la, \nu \in \Sig_{N-1}$, $\mu, \rho \in \Sig_N$ and an integer $r \ge 0$, we define coefficients $U^r (\mu \to \rho \mid \la \to \nu)$ which are related to a bijective proof of \eqref{eq:skew-Cauchy-1varAA}. It is this setting that directly generalises the classical setting of Young diagrams and Schur functions (see Remark \ref{rem:YDvsS}).

Both these variations of the HL-RSK algorithm seem to be important. The first setting ($\la, \nu, \mu, \rho \in \Sig_N$) can be more naturally described combinatorially, and it deals with arguably more natural modification of the Cauchy identity, which does not involve any extra functions (see \eqref{eq:skew-Cauchy-1varBB} and \eqref{eq:commut-op}). The second setting ($\la, \nu \in \Sig_{N-1}$, $\mu, \rho \in \Sig_N$, $r \ge 0$ or Young diagrams instead of signatures) is more classical and leads to integrable models of Section \ref{sec:degen}.

We define and study the HL-RSK in the first setting in Sections \ref{sec:descRSK} -- \ref{sec:1cDescr}, and then deduce all properties of the HL-RSK in the second setting in Section \ref{sec:rsk-input}.


\subsection{Description of the dynamics}
\label{sec:descRSK}

The main goal of this section is to define the probability $U(\mu \to \rho \mid \la \to \nu)$ for any fixed signatures $\la, \nu, \mu, \rho\in \Sig_N$ such that $\la \prec \nu, \mu \prec \rho$ and $|\nu| - |\la| = |\rho| - |\mu|$ \footnote{If these conditions on signatures are not satisfied, $U(\mu \to \rho \mid \la \to \nu)$ is set to be $0$.}, and discuss its immediate properties.

It will be convenient to think about an arbitrary signature $\kappa \in \Sig_N$ as about a collection of $N$ particles on $\mathbb{Z}$ with positions $\kappa_{N} \leq \cdots \leq \kappa_{1}$. Denote by $\kappa'_{k}$ the number of particles with positions $\geq k$. If $k > 0$ and $\kappa_{N} \geq 0$, then $\kappa'_{k}$ is the length of the $k$-th column of the Young diagram corresponding to $\kappa$.

Let $m = |\nu|-|\la|$ and $n = |\mu|-|\la|$. Fix a $[0, m] \times [0, n]$ rectangle on the square lattice $\mathbb{Z}^{2}$. Denote  by $V_{m, n}$ the set of its vertices, $|V_{m, n}| = (m+1)(n+1)$, and by $B_{m, n}$ the set of its boxes, $|B_{m, n}| = mn$. We will consider functions $\Lambda: \mathcal V \to \Sig_N$, such that
\begin{enumerate}

\item
$\mathcal V \subset V_{m,n}$, and all points of the form $(0,j)$, $j=0, 1,\ldots,n$, and $(i,0)$, $i=1, \ldots, m$, belong to $\mathcal V$.

\item
If $(I,J) \in \mathcal V$, then $(i,j) \in \mathcal V$ for all $i,j \in \Z_{\ge 0}$ such that $i \le I, j \le J$.

\item
$\Lambda(0, 0) = \la$, $\Lambda(m, 0) = \nu$, $\Lambda(0, n) = \mu$. 

\smallskip

\item
$\Lambda(i, j)$ is obtained from $\Lambda(i-1, j)$ by moving one particle by $+1$ (for $i>0$). Denote by $h_{i-1, j}$ the initial position of this particle.

\smallskip

\item
$\Lambda(i, j)$ is obtained from $\Lambda(i, j-1)$ by moving one particle by $+1$ (for $j>0$). Denote by $v_{i, j-1}$ the initial position of this particle.

\end{enumerate}

We will call such a function $\Lambda$ {\it admissible} if

\begin{enumerate}
\item
$h_{i, j} < h_{i+1, j}$ for all $i, j \geq 0$, $(i+2,j) \in \mathcal V$, and $v_{i, j} < v_{i, j+1}$ for all $i, j \geq 0$, $(i,j+2) \in \mathcal V$.

\smallskip

\item
If $h_{i, j} = v_{i, j} = k$, then $h_{i, j+1} = v_{i+1, j} = k$ or $k+1$ (for all $i, j \geq 0$, $(i+1, j+1) \in \mathcal V$).
\end{enumerate}

An example of an admissible function is given in Figure \ref{fig:example-admF1}.

\begin{figure}
\begin{tikzpicture}[>=stealth,scale=0.7]
\node at (0,0) {$(3,3,1,0)$};
\draw (1.5,0) -- (3.5,0); \node at (2.5,-0.4) {$1$};
\node at (5,0) {$(3,3,2,0)$};
\draw (6.5,0) -- (8.5,0); \node at (7.5,-0.4) {$3$};
\node at (10,0) {$(4,3,2,0)$};
\draw (11.5,0) -- (13.5,0); \node at (12.5,-0.4) {$4$};
\node at (15,0) {$(5,3,2,0)$};
\draw (0,0.5) -- (0,2.5); \draw (5,0.5) -- (5,2.5); \draw (10,0.5) -- (10,2.5); \draw (15,0.5) -- (15,2.5);
\node at (-0.3,1.5) {$0$}; \node at (4.7,1.5) {$0$}; \node at (9.7,1.5) {$0$}; \node at (14.7,1.5) {$0$};
\node at (0,3) {$(3,3,1,1)$};
\draw (1.5,3) -- (3.5,3); \node at (2.5,2.6) {$1$};
\node at (5,3) {$(3,3,2,1)$};
\draw (6.5,3) -- (8.5,3); \node at (7.5,2.6) {$3$};
\node at (10,3) {$(4,3,2,1)$};
\draw (11.5,3) -- (13.5,3); \node at (12.5,2.6) {$4$};
\node at (15,3) {$(5,3,2,1)$};
\draw (0,3.5) -- (0,5.5); \draw (5,3.5) -- (5,5.5); \draw (10,3.5) -- (10,5.5); \draw (15,3.5) -- (15,5.5);
\node at (-0.3,4.5) {$1$}; \node at (4.7,4.5) {$1$}; \node at (9.7,4.5) {$1$}; \node at (14.7,4.5) {$1$};
\node at (0,6) {$(3,3,2,1)$};
\draw (1.5,6) -- (3.5,6); \node at (2.5,5.6) {$1$};
\node at (5,6) {$(3,3,2,2)$};
\draw (6.5,6) -- (8.5,6); \node at (7.5,5.6) {$3$};
\node at (10,6) {$(4,3,2,2)$};
\draw (11.5,6) -- (13.5,6); \node at (12.5,5.6) {$4$};
\node at (15,6) {$(5,3,2,2)$};
\draw (0,6.5) -- (0,8.5); \draw (5,6.5) -- (5,8.5); 
\node at (-0.3,7.5) {$3$}; \node at (4.7,7.5) {$3$}; 
\node at (0,9) {$(4,3,2,1)$};
\draw (1.5,9) -- (3.5,9); \node at (2.5,8.6) {$1$};
\node at (5,9) {$(4,3,2,2)$};
\draw (0,9.5) -- (0,11.5); \draw (5,9.5) -- (5,11.5); 
\node at (-0.3,10.5) {$4$}; \node at (4.7,10.5) {$4$}; 
\node at (0,12) {$(5,3,2,1)$};
\draw (1.5,12) -- (3.5,12); \node at (2.5,11.6) {$1$};
\node at (5,12) {$(5,3,2,2)$};
\end{tikzpicture}

\caption{An example of an admissible function for $\la=(3,3,1,0)$, $\mu= (5,3,2,1)$, $\nu=(5,3,2,0)$. }
\label{fig:example-admF1}
\end{figure}
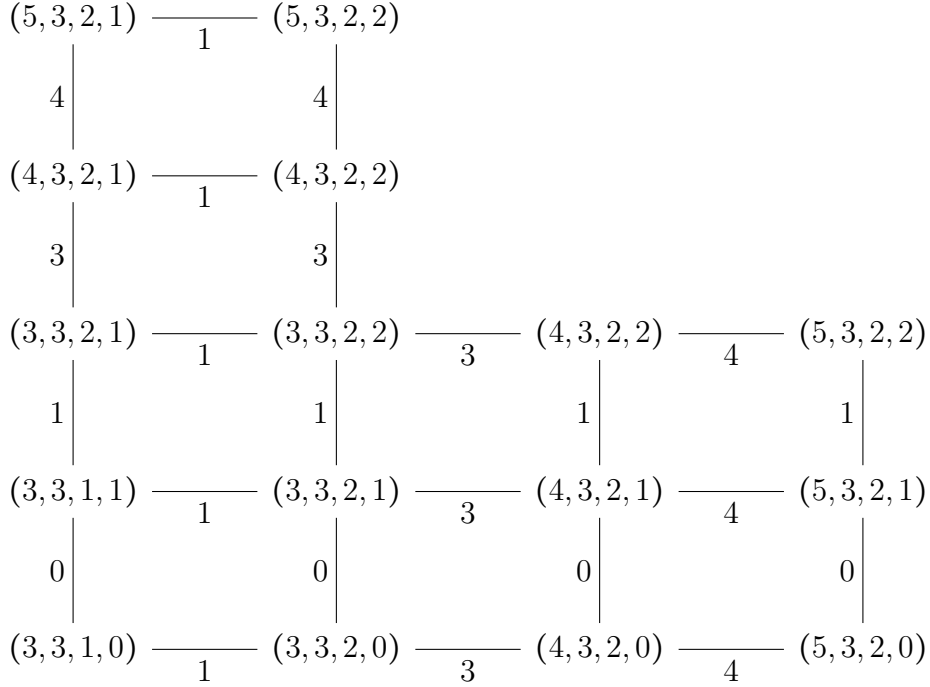

Denote by $b_{i, j}$ the lattice box with south-western corner at $(i, j)$. Given an admissible function $\Lambda: \mathcal V \to \Sig_N$, each box $b_{i, j}$ with $(i+1,j+1) \in \mathcal V$ belongs to one of the following three classes:
\begin{itemize}
\item
We call it a {\it zero-box} if $h_{i, j} = v_{i, j} = h_{i, j+1} = v_{i+1, j}$.

\smallskip

\item
We call it a {\it one-box} if $h_{i, j} = v_{i, j} = h_{i, j+1}-1 = v_{i+1, j}-1$.

\smallskip

\item
We call it a {\it trivial box} if $h_{i, j} = h_{i, j+1} \neq v_{i, j} = v_{i+1, j}$ (or if it is outside of the rectangle).

\end{itemize}
See Figure \ref{fig:exampBoxes}.

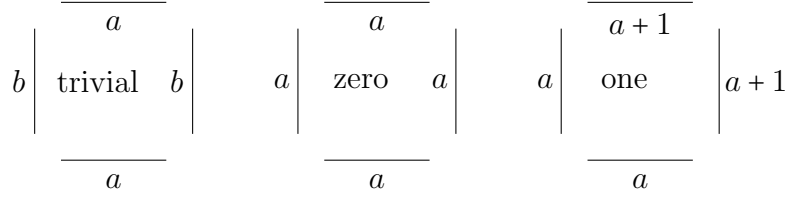
\begin{figure}
\begin{tikzpicture}[>=stealth,scale=0.7]
\draw (0.5,0) -- (2.5,0); \node at (1.5,-0.4) {$a$}; \draw (0,0.5) -- (0,2.5); \node at (-0.3,1.5) {$b$};
\draw (0.5,3) -- (2.5,3); \node at (1.5,2.6) {$a$}; \draw (3,0.5) -- (3,2.5); \node at (2.7,1.5) {$b$};
\node at (1.2,1.5) {trivial};
\draw (5.5,0) -- (7.5,0); \node at (6.5,-0.4) {$a$}; \draw (5,0.5) -- (5,2.5); \node at (4.7,1.5) {$a$};
\draw (5.5,3) -- (7.5,3); \node at (6.5,2.6) {$a$}; \draw (8,0.5) -- (8,2.5); \node at (7.7,1.5) {$a$};
\node at (6.2,1.5) {zero};
\draw (10.5,0) -- (12.5,0); \node at (11.5,-0.4) {$a$}; \draw (10,0.5) -- (10,2.5); \node at (9.7,1.5) {$a$};
\draw (10.5,3) -- (12.5,3); \node at (11.5,2.6) {$a+1$}; \draw (13,0.5) -- (13,2.5); \node at (13.7,1.5) {$a+1$};
\node at (11.2,1.5) {one};
\end{tikzpicture}
\caption{Three possible types of boxes.}
\label{fig:exampBoxes}
\end{figure}

If $B = b_{i, j}$ is a zero-box or a one-box, then set
\begin{align}
r(B):= h_{i, j} = v_{i, j}, \quad \overline{r}(B) := h_{i, j+1}=v_{i+1, j}, \quad c(B) := \exp_{t}\left(\Lambda(i, j)'_{r(B)} - \Lambda(i, j)'_{r(B)+1}\right),
\end{align}
where $\exp_{t}(a):=t^{a}$. Note that $\Lambda(i, j)'_{r(B)} - \Lambda(i, j)'_{r(B)+1}$ is the number of coordinates of $\Lambda(i, j)$ which are equal to $r(B)$.

There exists at most one admissible function such that $\mathcal V = V_{m,n}$ and $\Lambda (m,n) = \rho$. Indeed, there is a unique way to prescribe values of $\Lambda$ to all boundary vertices of the rectangle to satisfy the first condition of admissibility. Then there is at most one way to extend $\Lambda$ to interior vertices. If there is no admissible function, then define $U(\mu \to \rho \mid \la \to \nu) := 0$. Suppose now that such function $\Lambda$ exists. Note that all boxes from $B_{m,n}$ belong to one of the three classes defined above, and that any box adjacent to a zero-box is trivial.

\begin{figure}[h]
\includegraphics[width=0.8\textwidth]{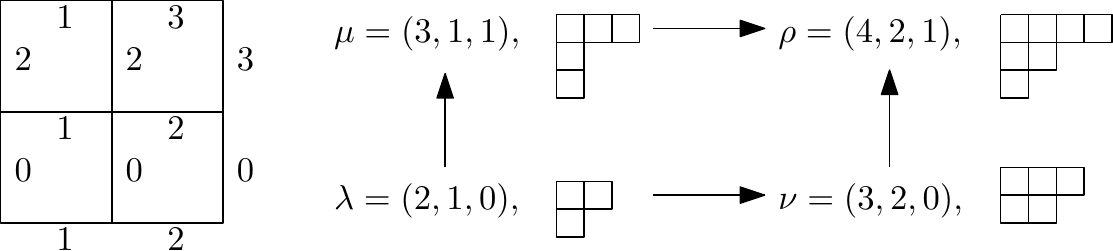}
\caption{On the right: an example of four signatures. On the left: values $h_{i, j}$, $v_{i, j}$ for the admissible function on the whole $V_{2,2}$.}
\end{figure}

\begin{definition}
\label{def:RSK-main1}
Let $\Lambda$ be the admissible function with $\Lambda (m,n) = \rho$. Consider $\{ h_{i,j} \}$, $\{ v_{i,j} \}$ (for the purpose of this definition we set $v_{i, -1} = h_{-1, j} = -\infty$). To each box $B=b_{i, j}$ we associate weight $w(B)$ according to the following rules:
\begin{enumerate}
\item
$w(B) = 1$, if

\begin{enumerate}
\item
$B$ is a trivial box.

\item
$B$ is a one-box and at least one of $b_{i-1, j}$, $b_{i, j-1}$ is a one-box.

\item
$B$ is a one-box and $b_{i-1, j-1}$ is a one-box with $\overline{r}(b_{i-1, j-1}) = r(B)-1$.

\item
$B$ is a zero-box and $b_{i-1, j-1}$ is a zero-box with $\overline{r}(b_{i-1, j-1}) = r(B)-1$.

\end{enumerate}

\medskip

\item
$w(B)=t$, if $B$ is a zero-box, and  exactly one of $h_{i-1, j}$, $v_{i, j-1}$ is equal to $r(B)-1$.

\medskip

\item
$w(B)=1-t$, if $B$ is a one-box, and boxes $b_{i-1, j}$, $b_{i, j-1}$ are trivial, and  exactly one of $h_{i-1, j}$, $v_{i, j-1}$ is equal to $r(B)-1$.

\medskip

\item
$w(B)=\frac{t-c(B)}{1-c(B)}$, if $B$ is a zero-box, and $h_{i-1, j}$, $v_{i, j-1} < r(B)-1$.

\medskip

\item
$w(B)=\frac{1-t}{1-c(B)}$, if $B$ is a one-box, and $h_{i-1, j}$, $v_{i, j-1} < r(B)-1$.

\medskip

\item
$w(B) = 0$, if
\begin{enumerate}
\item
$B$ is a one-box and $b_{i-1, j-1}$ is a zero-box with $\overline{r}(b_{i-1, j-1}) = r(B)-1$.

\item
$B$ is a zero-box and $b_{i-1, j-1}$ is a one-box with $\overline{r}(b_{i-1, j-1}) = r(B)-1$.

\end{enumerate}

\end{enumerate}

Then define
\begin{align}
U(\mu \to \rho \mid \la \to \nu) := \prod_{B \in B_{m, n}} w(B)
\end{align}

\end{definition}

\medskip

An example of this construction is given in Figure \ref{fig:weigh-full-example}.

\begin{figure}
\begin{tikzpicture}[>=stealth,scale=0.7]
\node at (0,0) {$(3,3,1,0)$};
\draw (1.5,0) -- (3.5,0); \node at (2.5,-0.4) {$1$};
\node at (5,0) {$(3,3,2,0)$};
\draw (6.5,0) -- (8.5,0); \node at (7.5,-0.4) {$3$};
\node at (10,0) {$(4,3,2,0)$};
\draw (11.5,0) -- (13.5,0); \node at (12.5,-0.4) {$4$};
\node at (15,0) {$(5,3,2,0)$};
\draw (0,0.5) -- (0,2.5); \draw (5,0.5) -- (5,2.5); \draw (10,0.5) -- (10,2.5); \draw (15,0.5) -- (15,2.5);
\node at (-0.3,1.5) {$0$}; \node at (4.7,1.5) {$0$}; \node at (9.7,1.5) {$0$}; \node at (14.7,1.5) {$0$};
\node at (0,3) {$(3,3,1,1)$};
\draw (1.5,3) -- (3.5,3); \node at (2.5,2.6) {$1$}; \node at (2.5,4.5) {$\mathbf{t}$};
\node at (5,3) {$(3,3,2,1)$};
\draw (6.5,3) -- (8.5,3); \node at (7.5,2.6) {$3$};
\node at (10,3) {$(4,3,2,1)$};
\draw (11.5,3) -- (13.5,3); \node at (12.5,2.6) {$4$};
\node at (15,3) {$(5,3,2,1)$};
\draw (0,3.5) -- (0,5.5); \draw (5,3.5) -- (5,5.5); \draw (10,3.5) -- (10,5.5); \draw (15,3.5) -- (15,5.5);
\node at (-0.3,4.5) {$1$}; \node at (4.7,4.5) {$1$}; \node at (9.7,4.5) {$1$}; \node at (14.7,4.5) {$1$};
\node at (0,6) {$(3,3,2,1)$};
\draw (1.5,6) -- (3.5,6); \node at (2.5,5.6) {$1$};
\node at (5,6) {$(3,3,2,2)$};
\draw (6.5,6) -- (8.5,6); \node at (7.5,5.6) {$3$}; \node at (7.5,7.5) {$\mathbf{\frac{1-t}{1-t^2}}$};
\node at (10,6) {$(4,3,2,2)$};
\draw (11.5,6) -- (13.5,6); \node at (12.5,5.6) {$4$}; \node at (12.5,7.5) {$\mathbf{1}$};
\node at (15,6) {$(5,3,2,2)$};
\draw (0,6.5) -- (0,8.5); \draw (5,6.5) -- (5,8.5); \draw (10,6.5) -- (10,8.5); \draw (15,6.5) -- (15,8.5);
\node at (-0.3,7.5) {$3$}; \node at (4.7,7.5) {$3$}; \node at (9.7,7.5) {$4$}; \node at (14.7,7.5) {$5$};
\node at (0,9) {$(4,3,2,1)$};
\draw (1.5,9) -- (3.5,9); \node at (2.5,8.6) {$1$};
\node at (5,9) {$(4,3,2,2)$};
\draw (6.5,9) -- (8.5,9); \node at (7.5,8.6) {$4$}; \node at (7.5,10.5) {$\mathbf{1}$};
\node at (10,9) {$(5,3,2,2)$};
\draw (11.5,9) -- (13.5,9); \node at (12.5,8.6) {$5$}; \node at (12.5,10.5) {$\mathbf{1}$};
\node at (15,9) {$(6,3,2,2)$};
\draw (0,9.5) -- (0,11.5); \draw (5,9.5) -- (5,11.5); \draw (10,9.5) -- (10,11.5); \draw (15,9.5) -- (15,11.5);
\node at (-0.3,10.5) {$4$}; \node at (4.7,10.5) {$4$}; \node at (9.7,10.5) {$5$}; \node at (14.7,10.5) {$6$};
\node at (0,12) {$(5,3,2,1)$};
\draw (1.5,12) -- (3.5,12); \node at (2.5,11.6) {$1$};
\node at (5,12) {$(5,3,2,2)$};
\draw (6.5,12) -- (8.5,12); \node at (7.5,11.6) {$5$};
\node at (10,12) {$(6,3,2,2)$};
\draw (11.5,12) -- (13.5,12); \node at (12.5,11.6) {$6$};
\node at (15,12) {$(7,3,2,2)$};
\end{tikzpicture}
\caption{The admissible function corresponding to $\la=(3,3,1,0)$, $\mu= (5,3,2,1)$, $\nu=(5,3,2,0)$, and $\rho = (7,3,2,2)$. The weights of all non-trivial boxes are indicated in bold. As a result, $U(\mu \to \rho \mid \la \to \nu) = t \frac{1-t}{1-t^2}$ for these signatures.}
\label{fig:weigh-full-example}
\end{figure}
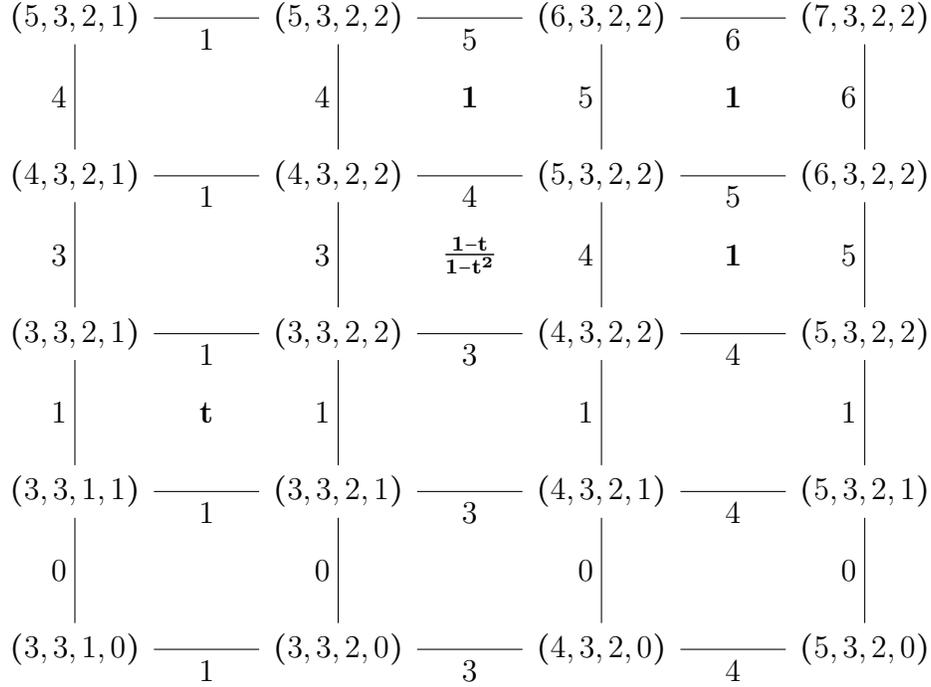

Our next goal is to give a natural procedure which takes the input $\la, \mu, \nu$ and outputs the random signature $\tilde \rho$ with probability $U(\mu \to \tilde \rho \mid \la \to \nu)$. The idea behind it is to construct an admissible function on the whole $V_{m,n}$ iteratively by assigning value to one vertex at a time.
In order to formulate the procedure, we need a technical lemma first.

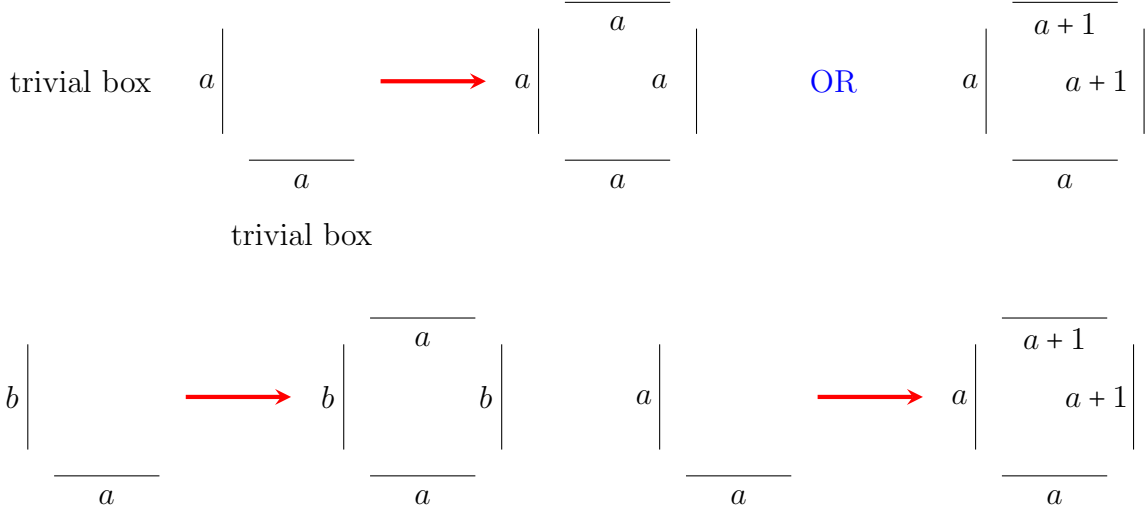
\begin{figure}
\begin{tikzpicture}[>=stealth,scale=0.7]
\draw (0.5,0) -- (2.5,0); \node at (1.5,-0.4) {$a$}; \draw (0,0.5) -- (0,2.5); \node at (-0.3,1.5) {$b$};
\draw[ultra thick,->,red] (3,1.5) -- (5,1.5);
\draw (6.5,0) -- (8.5,0); \node at (7.5,-0.4) {$a$}; \draw (6,0.5) -- (6,2.5); \node at (5.7,1.5) {$b$};
\draw (6.5,3) -- (8.5,3); \node at (7.5,2.6) {$a$}; \draw (9,0.5) -- (9,2.5); \node at (8.7,1.5) {$b$};
\draw (12.5,0) -- (14.5,0); \node at (13.5,-0.4) {$a$}; \draw (12,0.5) -- (12,2.5); \node at (11.7,1.5) {$a$};
\draw[ultra thick,->,red] (15,1.5) -- (17,1.5);
\draw (18.5,0) -- (20.5,0); \node at (19.5,-0.4) {$a$}; \draw (18,0.5) -- (18,2.5); \node at (17.7,1.5) {$a$};
\draw (18.5,3) -- (20.5,3); \node at (19.5,2.6) {$a+1$}; \draw (21,0.5) -- (21,2.5); \node at (20.3,1.5) {$a+1$};
\draw (4.2,6) -- (6.2,6); \node at (5.2,5.6) {$a$}; \draw (3.7,6.5) -- (3.7,8.5); \node at (3.4,7.5) {$a$};
\node at (5.2,4.6) {trivial box}; \node at (1,7.5) {trivial box};
\draw[ultra thick,->,red] (6.7,7.5) -- (8.7,7.5);
\draw (10.2,6) -- (12.2,6); \node at (11.2,5.6) {$a$}; \draw (9.7,6.5) -- (9.7,8.5); \node at (9.4,7.5) {$a$};
\draw (10.2,9) -- (12.2,9); \node at (11.2,8.6) {$a$}; \draw (12.7,6.5) -- (12.7,8.5); \node at (12,7.5) {$a$};
\node at (15.3,7.5) {\textcolor{blue}{OR}};
\draw (18.7,6) -- (20.7,6); \node at (19.7,5.6) {$a$}; \draw (18.7,9) -- (20.7,9); \node at (17.9,7.5) {$a$};
\draw (18.2,6.5) -- (18.2,8.5); \node at (19.7,8.6) {$a+1$}; \draw (21.2,6.5) -- (21.2,8.5); \node at (20.3,7.5) {$a+1$};
\end{tikzpicture}
\caption{Possible ways to extend an admissible function to a new vertex, see Lemma \ref{lem:refined}. Cases A (we assume $a \ne b$) and B of the lemma are depicted in the bottom, case C -- in the top.}
\label{fig:exampLemmaTec}
\end{figure}

\begin{lemma}
\label{lem:refined}
Let $\Lambda: \mathcal V \to \Sig_N$ be an admissible function with $\mathcal V \subset V_{m,n}$, and assume that $(i,j) \in V_{m,n}$
satisfies $(i+1,j+1) \notin \mathcal V$, $(i,j), (i+1,j), (i,j+1) \in \mathcal V$. Then one can define value of $\Lambda$ at $(i+1, j+1)$ (it is enough to assign $h_{i,j+1}$ and $v_{i+1,j}$ for this) such that $\Lambda$ remains admissible. Moreover, this can be done in the following specific way (see Figure \ref{fig:exampLemmaTec}) :

A) If $h_{i,j} \ne v_{i,j}$, then assign $h_{i,j+1} := h_{i,j}$, $v_{i+1,j} := v_{i,j}$.

B) If $h_{i,j} = v_{i,j}$ and at least one of the boxes $b_{i,j-1}$ and $b_{i-1,j}$ is a one-box. Then $h_{i,j+1} := h_{i,j}+1$, $v_{i+1,j} := h_{i,j}+1$.

C) $h_{i,j} = v_{i,j}$ and both $b_{i,j-1}, b_{i-1,j}$ are trivial boxes. Then both assignments: $h_{i,j+1} := h_{i,j}+1$, $v_{i+1,j} := h_{i,j}+1$ and $h_{i,j+1} := h_{i,j}$, $v_{i+1,j} := h_{i,j}$, produce an admissible function.

These three cases exhaust all possible situations.

\end{lemma}

\begin{proof}

Throughout the proof we will extensively use the fact that each box is of one of the three types from Figure \ref{fig:exampBoxes}. We need to check that assignment of $h_{i,j+1}$ and $v_{i+1,j}$ specified in the lemma's statement satisfies two conditions from the definition of admissible functions. The second condition is satisfied automatically. Let us check the first condition. It is enough to show that $h_{i,j+1} > h_{i-1,j+1}$; the proof of $v_{i+1,j} > v_{i+1, j-1}$ is analogous (by the symmetry).

\begin{enumerate}

\item
Suppose we are in case A. Then there are two subcases. First, if $h_{i-1,j+1} \ne v_{i,j}$ or $b_{i-1,j}$ is a zero-box, then  $h_{i, j+1} = h_{i, j} > h_{i-1, j} = h_{i-1, j+1}$. Second, if $b_{i-1,j}$ is a one-box, then
$h_{i, j+1} = h_{i, j} > h_{i-1, j}+1 = h_{i-1, j+1}$.

\item
Suppose we are in case B or C, and $h_{i,j+1} := h_{i,j}+1$, $v_{i+1,j} := h_{i,j}+1$. Then $h_{i, j+1} = h_{i, j}+1 > h_{i-1, j}+1 \ge h_{i-1, j+1}$.

\item
Suppose we are in case C, and $h_{i,j+1} := h_{i,j}$, $v_{i+1,j} := h_{i,j}$. Then $h_{i, j+1} = h_{i, j} > h_{i-1, j} = h_{i-1, j+1}$.
\end{enumerate}
Thus the first condition is satisfied in all cases.
\end{proof}
Now we are able to reformulate Definition \ref{def:RSK-main1} in terms of a probabilistic algorithm.

\begin{definition}[Refined HL-RSK algorithm]
\label{def:refined}

\noindent

\smallskip

\textbf{Input}: $\la, \mu, \nu \in \Sig_N$: $\la \prec \mu$, $\la \prec \nu$.

\smallskip

\textbf{(Random) Output}: A signature $\tilde \rho \in \Sig_N$ appearing with probability $U(\mu \to \tilde \rho \mid \la \to \nu)$.

\smallskip

Step 1: Consider $V_{m,n}$, $B_{m,n}$ as above and assign values $\Lambda (0,j)$, $\Lambda (i,0)$, $0 \le i \le m$, $0 \le j \le n$, in a (unique) way so that they form an admissible function on the left and bottom boundaries of $V_{m,n}$.

\smallskip

Step 2: Consider a vertex $(i+1,j+1)$ at which the value of $\Lambda$ is not yet assigned but at its west, south, and south-west neighbors values of $\Lambda$ are already assigned. Then we are in one of the three cases of Lemma \ref{lem:refined}. Assign value of $\Lambda$ in $(i+1,j+1)$ in the following way (the first four cases are depicted in Figure \ref{fig:exampRefRSK}; the last two cases are depicted in Figure \ref{fig:exampRefRSK2}):

\begin{itemize}

\item
In case A, set $h_{i,j+1} := h_{i,j}$, $v_{i+1,j} := v_{i,j}$.

\item
In case B, set $h_{i,j+1} := h_{i,j}+1$, $v_{i+1,j} := h_{i,j}+1$.

\item
In case C, and if $h_{i,j} -1 = v_{i,j}-1 = h_{i-1,j} = v_{i,j-1} = h_{i-1,j-1}=v_{i-1,j-1}$,  set $h_{i,j+1} := h_{i,j}$, $v_{i+1,j} := h_{i,j}$.

\item
In case C, and if $h_{i,j} -1 = v_{i,j}-1 = h_{i-1,j} = v_{i,j-1} = h_{i-1,j-1}+1=v_{i-1,j-1}+1$, set $h_{i,j+1} := h_{i,j}+1$, $v_{i+1,j} := h_{i,j}+1$.

\item
In case C, and if exactly one of $h_{i-1,j}, v_{i,j-1}$ is equal to $h_{i,j}-1$, then set $h_{i,j+1} := h_{i,j}+1$, $v_{i+1,j} := h_{i,j}+1$ with probability $(1-t)$, or set $h_{i,j+1} := h_{i,j}$, $v_{i+1,j} := h_{i,j}$ with probability $t$.

\item
In case C, and if $h_{i-1,j}, v_{i, j-1} < h_{i,j}-1$, then set $h_{i,j+1} := h_{i,j}+1$, $v_{i+1,j} := h_{i,j}+1$ with probability $(1-t)/(1-c(B))$, or set $h_{i,j+1} := h_{i,j}$, $v_{i+1,j} := h_{i,j}$ with probability $(t-c(B))/(1-c(B))$; recall that $c(B)$ is $t$ raised to the number of coordinates of the signature $\Lambda(i,j)$ which are equal to $h_{i,j}$.

\end{itemize}

Step 3: Iterate Step 2 until we fill all vertices in $V_{m,n}$. Output the value $\Lambda (m,n)$.

\end{definition}

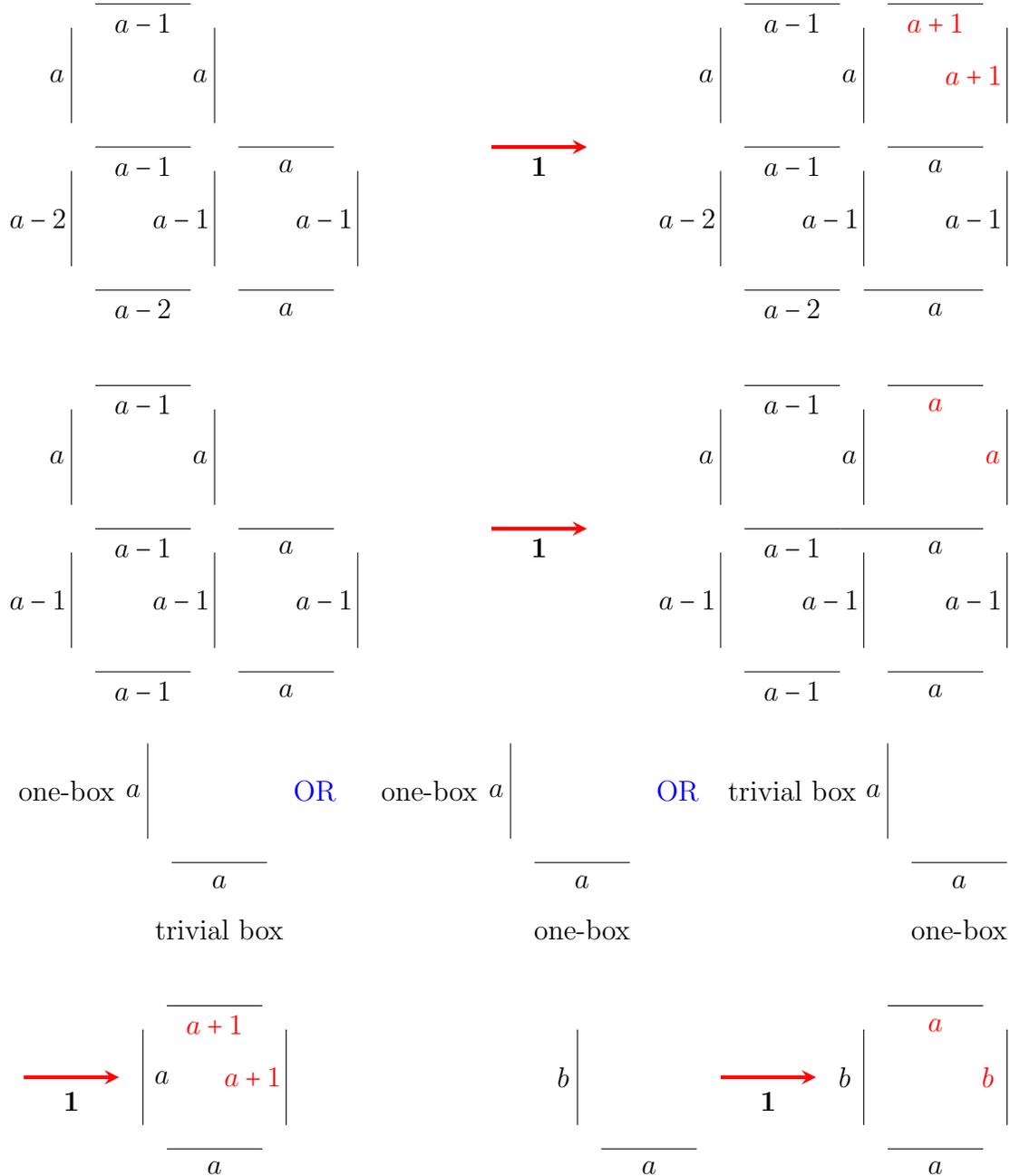
\begin{figure}
\begin{tikzpicture}[>=stealth,scale=0.7]
\draw[ultra thick,->,red] (-2,1.5) -- (0,1.5); \node at (-1,1) {\textbf{1}};
\draw (1,0) -- (3,0); \node at (2, -0.4) {$a$}; \draw (0.5, 0.5) -- (0.5, 2.5); \node at (0.9,1.5) {$a$};
\draw (1,3) -- (3,3); \node at (2,2.6) {\textcolor{red}{$a+1$}}; \draw (3.5,0.5) -- (3.5,2.5); \node at (2.8,1.5) {\textcolor{red}{$a+1$}};
\draw (10.1,0) -- (12.1,0); \node at (11.1,-0.4) {$a$}; \draw (9.6,0.5) -- (9.6,2.5); \node at (9.3,1.5) {$b$};
\draw[ultra thick,->,red] (12.6,1.5) -- (14.6,1.5); \node at (13.6,1) {\textbf{1}};
\draw (16.1,0) -- (18.1,0); \node at (17.1,-0.4) {$a$}; \draw (15.6,0.5) -- (15.6,2.5); \node at (15.2,1.5) {$b$};
\draw (16.1,3) -- (18.1,3); \node at (17.1,2.6) {$\textcolor{red}{a}$}; \draw (18.6,0.5) -- (18.6,2.5); \node at (18.2,1.5) {$\textcolor{red}{b}$};
\draw (1.1,6) -- (3.1,6); \node at (2.1,5.6) {$a$}; \draw (0.6,6.5) -- (0.6,8.5); \node at (0.3,7.5) {$a$};
\node at (2.1,4.6) {trivial box}; \node at (-1.1,7.5) {one-box};
\node at (4.1,7.5) {\textcolor{blue}{OR}};
\draw (8.7,6) -- (10.7,6); \node at (9.7,5.6) {$a$}; \draw (8.2,6.5) -- (8.2,8.5); \node at (7.9,7.5) {$a$};
\node at (9.7,4.6) {one-box}; \node at (6.5,7.5) {one-box};
\node at (11.7,7.5) {\textcolor{blue}{OR}};
\draw (16.6,6) -- (18.6,6); \node at (17.6,5.6) {$a$}; \draw (16.1,6.5) -- (16.1,8.5); \node at (15.8,7.5) {$a$};
\node at (17.6,4.6) {one-box}; \node at (14.1,7.5) {trivial box};
\draw (15.1,13) -- (18.1,13); \node at (17.1,12.6) {$a$}; \draw (15.6,13.5) -- (15.6,15.5); \node at (15.3,14.5) {$a$};
\draw (13.1,13) -- (15.1,13); \node at (14.1,12.6) {$a-1$}; \draw (15.6,10.5) -- (15.6,12.5); \node at (14.9,11.5) {$a-1$};
\draw (16.1,10) -- (18.1,10); \node at (17.1,9.6) {$a$}; \draw (18.6,10.5) -- (18.6,12.5); \node at (17.9,11.5) {$a-1$};
\draw (13.1,16) -- (15.1,16); \node at (14.1,15.6) {$a-1$}; \draw (12.6,13.5) -- (12.6,15.5); \node at (12.3,14.5) {$a$};
\draw (13.1,10) -- (15.1,10); \node at (14.1,9.6) {$a-1$}; \draw (12.6,10.5) -- (12.6,12.5); \node at (11.9,11.5) {$a-1$};
\draw (16.1,16) -- (18.1,16); \node at (17.1,15.6) {\textcolor{red}{$a$}}; \draw (18.6,13.5) -- (18.6,15.5); \node at (18.3,14.5) {\textcolor{red}{$a$}};
\draw[ultra thick,->,red] (7.8,13) -- (9.8,13); \node at (8.8,12.6) {\textbf{1}};
\draw (2.5,13) -- (4.5,13); \node at (3.5,12.6) {$a$}; \draw (2,13.5) -- (2,15.5); \node at (1.7,14.5) {$a$};
\draw (-0.5,13) -- (1.5,13); \node at (0.5,12.6) {$a-1$}; \draw (2,10.5) -- (2,12.5); \node at (1.3,11.5) {$a-1$};
\draw (2.5,10) -- (4.5,10); \node at (3.5,9.6) {$a$}; \draw (5,10.5) -- (5,12.5); \node at (4.3,11.5) {$a-1$};
\draw (-0.5,16) -- (1.5,16); \node at (0.5,15.6) {$a-1$}; \draw (-1,13.5) -- (-1,15.5); \node at (-1.3,14.5) {$a$};
\draw (-0.5,10) -- (1.5,10); \node at (0.5,9.6) {$a-1$}; \draw (-1,10.5) -- (-1,12.5); \node at (-1.7,11.5) {$a-1$};
\draw (16.1,21) -- (18.1,21); \node at (17.1,20.6) {$a$}; \draw (15.6,21.5) -- (15.6,23.5); \node at (15.3,22.5) {$a$};
\draw (13.1,21) -- (15.1,21); \node at (14.1,20.6) {$a-1$}; \draw (15.6,18.5) -- (15.6,20.5); \node at (14.9,19.5) {$a-1$};
\draw (15.6,18) -- (18.1,18); \node at (17.1,17.6) {$a$}; \draw (18.6,18.5) -- (18.6,20.5); \node at (17.9,19.5) {$a-1$};
\draw (13.1,24) -- (15.1,24); \node at (14.1,23.6) {$a-1$}; \draw (12.6,21.5) -- (12.6,23.5); \node at (12.3,22.5) {$a$};
\draw (13.1,18) -- (15.1,18); \node at (14.1,17.6) {$a-2$}; \draw (12.6,18.5) -- (12.6,20.5); \node at (11.9,19.5) {$a-2$};
\draw (16.1,24) -- (18.1,24); \node at (17.1,23.6) {\textcolor{red}{$a+1$}}; \draw (18.6,21.5) -- (18.6,23.5); \node at (17.9,22.5) {\textcolor{red}{$a+1$}};
\draw[ultra thick,->,red] (7.8,21) -- (9.8,21); \node at (8.8,20.6) {\textbf{1}};
\draw (2.5,21) -- (4.5,21); \node at (3.5,20.6) {$a$}; \draw (2,21.5) -- (2,23.5); \node at (1.7,22.5) {$a$};
\draw (-0.5,21) -- (1.5,21); \node at (0.5,20.6) {$a-1$}; \draw (2,18.5) -- (2,20.5); \node at (1.3,19.5) {$a-1$};
\draw (2.5,18) -- (4.5,18); \node at (3.5,17.6) {$a$}; \draw (5,18.5) -- (5,20.5); \node at (4.3,19.5) {$a-1$};
\draw (-0.5,24) -- (1.5,24); \node at (0.5,23.6) {$a-1$}; \draw (-1,21.5) -- (-1,23.5); \node at (-1.3,22.5) {$a$};
\draw (-0.5,18) -- (1.5,18); \node at (0.5,17.6) {$a-2$}; \draw (-1,18.5) -- (-1,20.5); \node at (-1.7,19.5) {$a-2$};
\end{tikzpicture}
\caption{Deterministic elementary steps of the RSK algorithm from Definition \ref{def:refined}; $b$ is an arbitrary integer not equal to  $a$.}
\label{fig:exampRefRSK}
\end{figure}

\begin{figure}
\begin{tikzpicture}[>=stealth,scale=0.65]
\draw (2.5,28) -- (4.5,28); \node at (3.5,27.6) {$a$}; \draw (2,28.5) -- (2,30.5); \node at (1.7,29.5) {$a$};
\draw (-0.5,28) -- (1.5,28); \node at (0.5,27.6) {$a-1$}; \draw (2,25.5) -- (2,27.5); \node at (1.3,26.5) {$b$};
\node at (5.5,28) {\textcolor{blue}{OR}};
\draw (9.5,28) -- (11.5,28); \node at (10.5,27.6) {$a$}; \draw (9,28.5) -- (9,30.5); \node at (8.7,29.5) {$a$};
\draw (6.5,28) -- (8.5,28); \node at (7.5,27.6) {$b$}; \draw (9,25.5) -- (9,27.5); \node at (8.3,26.5) {$a-1$};
\draw[ultra thick,->,red] (12.5,28) -- (14.5,28); \node at (13.5,27.6) {$\mathbf{1-t}$};
\draw (18.5,28) -- (20.5,28); \node at (19.5,27.6) {$a$}; \draw (18,28.5) -- (18,30.5); \node at (17.7,29.5) {$a$};
\draw (15.5,28) -- (17.5,28); \draw (18,25.5) -- (18,27.5);
\draw (18.5,31) -- (20.5,31); \node at (19.5,30.6) {\textcolor{red}{$a+1$}}; \draw (21,28.5) -- (21,30.5); \node at (20.3,29.5) {\textcolor{red}{$a+1$}};
\draw (2.5,35) -- (4.5,35); \node at (3.5,34.6) {$a$}; \draw (2,35.5) -- (2,37.5); \node at (1.7,36.5) {$a$};
\draw (-0.5,35) -- (1.5,35); \node at (0.5,34.6) {$a-1$}; \draw (2,32.5) -- (2,34.5); \node at (1.3,33.5) {$b$};
\node at (5.5,35) {\textcolor{blue}{OR}};
\draw (9.5,35) -- (11.5,35); \node at (10.5,34.6) {$a$}; \draw (9,35.5) -- (9,37.5); \node at (8.7,36.5) {$a$};
\draw (6.5,35) -- (8.5,35); \node at (7.5,34.6) {$b$}; \draw (9,32.5) -- (9,34.5); \node at (8.3,33.5) {$a-1$};
\draw[ultra thick,->,red] (12.5,35) -- (14.5,35); \node at (13.5,34.6) {$\mathbf{t}$};
\draw (18.5,35) -- (20.5,35); \node at (19.5,34.6) {$a$}; \draw (18,35.5) -- (18,37.5); \node at (17.7,36.5) {$a$};
\draw (15.5,35) -- (17.5,35);  \draw (18,32.5) -- (18,34.5);
\draw (18.5,38) -- (20.5,38); \node at (19.5,37.6) {\textcolor{red}{$a$}}; \draw (21,35.5) -- (21,37.5); \node at (20.7,36.5) {\textcolor{red}{$a$}};
\draw (2.5,42) -- (4.5,42); \node at (3.5,41.6) {$a$}; \draw (2,42.5) -- (2,44.5); \node at (1.7,43.5) {$a$};
\draw (-0.5,42) -- (1.5,42); \node at (0.5,41.6) {$c$}; \draw (2,39.5) -- (2,41.5); \node at (1.3,40.5) {$b$};
\draw[ultra thick,->,red] (9,42) -- (11,42); \node at (10,41.3) {$\mathbf{\frac{1-t}{1-t^{c(B)}}}$};
\draw (18.5,42) -- (20.5,42); \node at (19.5,41.6) {$a$}; \draw (18,42.5) -- (18,44.5); \node at (17.7,43.5) {$a$};
\draw (15.5,42) -- (17.5,42); \node at (16.5,41.6) {$c$}; \draw (18,39.5) -- (18,41.5); \node at (17.3,40.5) {$b$};
\draw (18.5,45) -- (20.5,45); \node at (19.5,44.6) {\textcolor{red}{$a+1$}}; \draw (21,42.5) -- (21,44.5); \node at (20.3,43.5) {\textcolor{red}{$a+1$}};
\draw (2.5,49) -- (4.5,49); \node at (3.5,48.6) {$a$}; \draw (2,49.5) -- (2,51.5); \node at (1.7,50.5) {$a$};
\draw (-0.5,49) -- (1.5,49); \node at (0.5,48.6) {$c$}; \draw (2,46.5) -- (2,48.5); \node at (1.3,47.5) {$b$};
\draw[ultra thick,->,red] (9,49) -- (11,49); \node at (10,48.3) {$\mathbf{\frac{t-t^{c(B)}}{1-t^{c(B)}}}$};
\draw (18.5,49) -- (20.5,49); \node at (19.5,48.6) {$a$}; \draw (18,49.5) -- (18,51.5); \node at (17.7,50.5) {$a$};
\draw (15.5,49) -- (17.5,49); \node at (16.5,48.6) {$c$}; \draw (18,46.5) -- (18,48.5); \node at (17.3,47.5) {$b$};
\draw (18.5,52) -- (20.5,52); \node at (19.5,51.6) {\textcolor{red}{$a$}}; \draw (21,49.5) -- (21,51.5); \node at (20.7,50.5) {\textcolor{red}{$a$}};
\end{tikzpicture}
\caption{Randomized elementary steps of the RSK algorithm from Definition \ref{def:refined}; $b$ and $c$ are arbitrary numbers less than $a-1$.}
\label{fig:exampRefRSK2}
\end{figure}
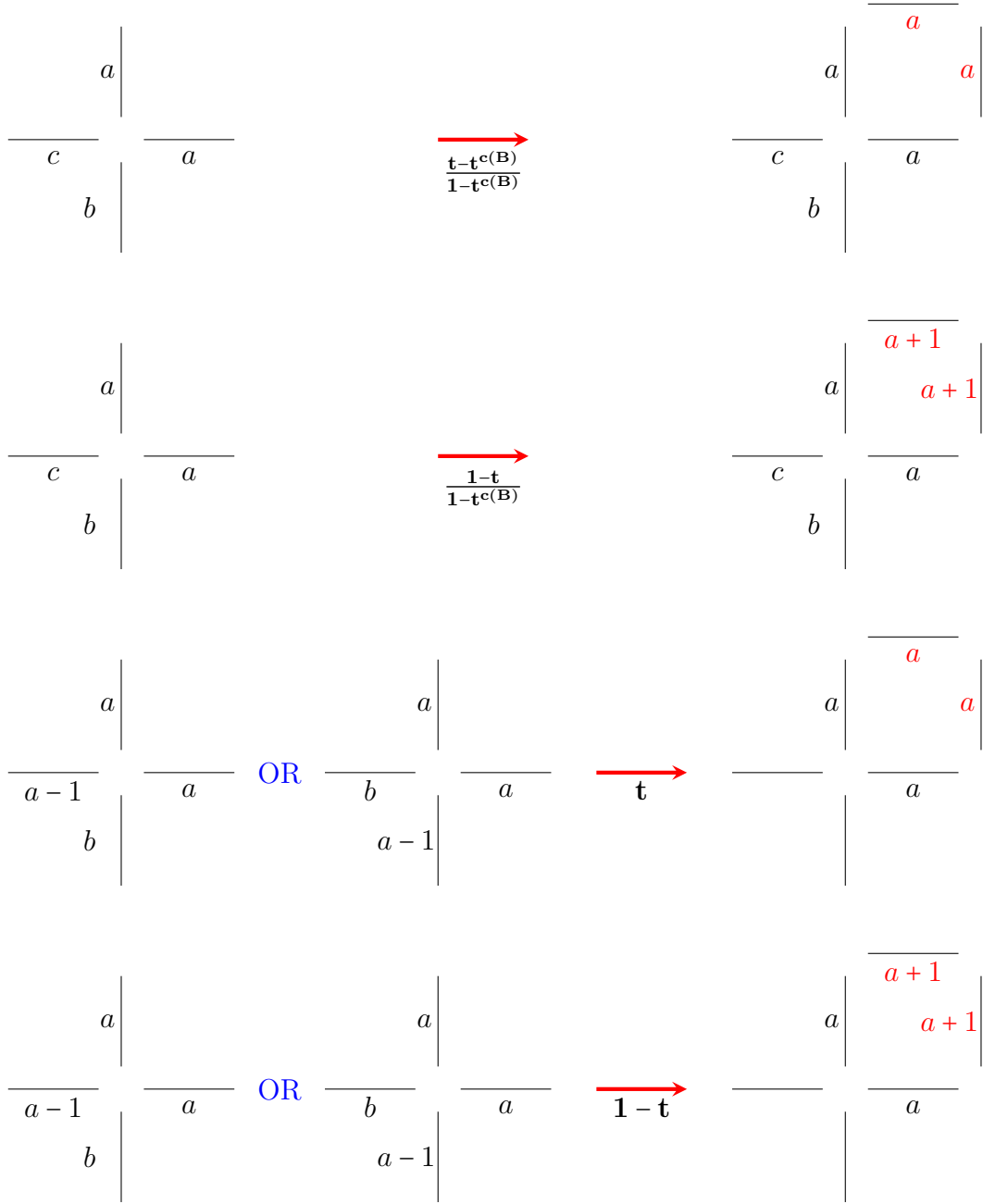

Definition \ref{def:RSK-main1} and Lemma \ref{lem:refined} directly imply that the algorithm from Definition \ref{def:refined} indeed outputs a signature $\tilde \rho$ with probability $U(\mu \to \tilde \rho \mid \la \to \nu)$. As a corollary, we obtain
\begin{equation}
\label{eq:sumUeq1}
\sum_{\tilde \rho} U(\mu \to \tilde \rho \mid \la \to \nu) = 1 \qquad \text{for fixed } \la, \nu, \mu, \text{ such that } \la \prec \nu, \mu.
\end{equation}

Let us address several properties of the coefficients $U(\mu \to \rho \mid \la \to \nu)$.

\begin{proposition}[Symmetry]
\label{prop:symmetry}
We have
$$
U(\mu \to \rho \mid \la \to \nu) =  U(\nu \to \rho \mid \la \to \mu)
$$
\end{proposition}
\begin{proof}
Note that all rules of Definition \ref{def:RSK-main1} are symmetric with respect to the swap of horizontal and vertical directions. This immediately implies the statement.
\end{proof}

\begin{proposition}[Markov projection]
\label{prop:Markov-proj}
In the algorithm of Definition \ref{def:refined}, for any $k \in \Z$ the probability distribution of the set $\{ \tilde \rho'_a \}_{a \le k}$ is determined by sets $\{ \la'_a \}_{a \le k}$, $\{ \mu'_a \}_{a \le k}$, $\{ \nu'_a \}_{a \le k}$ only.
\end{proposition}
\begin{proof}
Let $I$ be the minimal non-negative integer such that $I=m$ or $h_{I, 0}> k$, and $J$ be the minimal non-negative integer such that $J=n$ or $v_{0, J}> k$. Then $\{ \Lambda(m, n)'_a \}_{a \le k} = \{ \Lambda(I, J)'_a \}_{a \le k}$, since all (horizontal or vertical) edges in the lattice rectangle $[I, m] \times [J, n]$ have labels $>k$. However, the distribution of $\Lambda(I, J)$ is determined only by the sets $\{ \la'_a \}_{a \le k}$, $\{ \mu'_a \}_{a \le k}$, $\{ \nu'_a \}_{a \le k}$, hence the statement follows.
\end{proof}

The key role is played by the \textit{flip property} which will be addressed in the next subsection.

\subsection{Flip property}
\label{sec:flip}

For $\la \in \Sig_N$ define the {\it reversed signature} $-\la:=(-\la_{N}, \ldots, -\la_{1}) \in Sig_N$. Clearly, $\la \prec \nu \Leftrightarrow -\nu \prec - \la$. If $\Lambda$ is an admissible
function on all $V_{m,n}$ with $\Lambda(0, 0) = \la$, $\Lambda(m, 0) = \nu$, $\Lambda(0, n) = \mu$, $\Lambda(m, n) = \rho$ (see Section \ref{sec:descRSK}), then we can define the admissible function $\overline{\Lambda}$ with $\overline{\Lambda}(0, 0) = -\rho$, $\overline{\Lambda}(m, 0) = -\mu$, $\overline{\Lambda}(0, n) = -\nu$, $\overline{\Lambda}(m, n) = -\la$ by
setting
\begin{align*}
\overline{\Lambda}_{i, j} := -\Lambda_{m-i, n-j}
\end{align*}

\begin{proposition}[Flip Property] We have
\label{lemma:flip}
\begin{align}
U(\mu \to \rho \mid \la \to \nu) \frac{\phi_{\nu/\la}\psi_{\mu/\la}}{\phi_{\rho/\mu}\psi_{\rho/\nu}} = U(-\nu \to -\la \mid -\rho \to -\mu).
\label{eq:flip}
\end{align}
\end{proposition}
\begin{proof}
The left-hand side and the right-hand side of equality \eqref{eq:flip} can be expressed as, respectively,
$f(t) \prod_{k \in \mathbb{Z}} f_{k}(C_{k})$ and $g(t) \prod_{k \in \mathbb{Z}} g_{k}(C_{k})$, where $f(t), g(t)$ are polynomials in $t$, and $f_{k}, g_{k}$ are rational functions in variable $C_{k} = \exp_{t}\left(\lambda'_{k} - \lambda'_{k+1}\right)$ such that $f_{k}(0) = g_{k}(0) = 1$. Note that in a weight of the form $ (1-t)/ (1 - t^{p}C_k)$ the numerator contributes to $f(t)$, while the denominator contributes to $f_k(C_{k})$ function. Analogously, in a weight of the form $ (t - t^{p}C_k) / (1 - t^{p}C_k)$ the factor $t$ contributes to $f(t)$, while the rest contributes to $f_k(C_{k})$. $p$ here is some integer which depends on the assignment of labels. Coefficients $\phi_{\nu/\la}, \psi_{\mu/\la}, \phi_{\rho/\mu}, \psi_{\rho/\nu}$ don't contribute anything to $f(t), g(t), g_{k}(t)$: their whole contribution goes to $f_{k}(C_{k})$.

In the proof, we analyze the combinatorics of the admissible function $\Lambda$ with $\Lambda(0, 0) = \la$, $\Lambda(m, 0) = \nu$, $\Lambda(0, n) = \mu$, $\Lambda(m, n) = \rho$. First, let us analyze $f(t)$ and $g(t)$. We claim that $f(t) = g(t) = t^{N}(1-t)^{M}$, where numbers $N$, $M$ have the following combinatorial interpretations.

Connect zero-boxes $B_{1}$, $B_{2}$ by an edge if the north-eastern corner of $B_{1}$ coincides with the south-western corner of $B_{2}$ and $\overline{r}(B_{1}) = r(B_{2})-1$. Then $N$ is the number of connected components in the graph of zero-boxes with such edges, since from each component only the south-western box contributes $t$ to $f(t)$ and only the north-eastern box contributes $t$ to $g(t)$.

Connect one-boxes $B_{1}$, $B_{2}$ by an edge if either $B_{1}$ is adjacent to the southern or western edge of $B_2$, or $\overline{r}(B_{1}) = r(B_{2})-1$ and the north-eastern corner of $B_{1}$ coincides with the south-western corner of $B_{2}$. Note that if a one-box $B$ is connected to its immediate southern and western (northern and eastern) neighbors, then it is connected to the south-western (north-eastern, respectively) neighbor. Then $M$ is the number of connected components in the graph of one-boxes with such edges. It is easy to check that each such component contains a unique box not connected to any its southern, western and south-western neighbors. This box contributes $1-t$ to $f(t)$.  Each component also contains a unique box not connected to any of its northern, eastern and north-eastern neighbors. This box contributes $1-t$ to $g(t)$. See Figure \ref{fig:FlipProofEx} for an example of an arrangement of boxes.
\begin{figure}[h]
\includegraphics[width=0.6\textwidth]{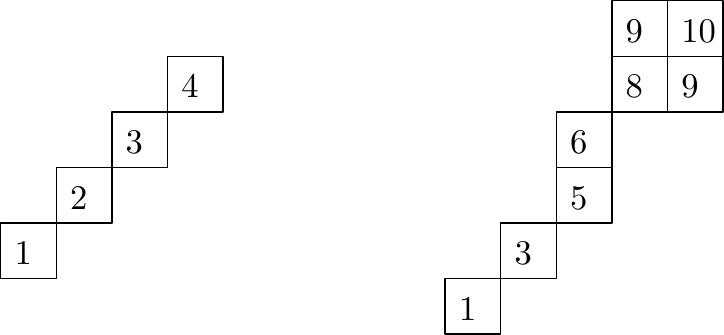}
\caption{On the left: an example of connected component for the graph of zero-boxes. On the right: an example of connected component for the graph of one-boxes. In both cases number in each box is the corresponding $r(B)$.}
\label{fig:FlipProofEx}
\end{figure}

Now we need to check that $f_{k} = g_{k}$. There are four types of boxes that could potentially contribute terms with $C_{k}$ to either the left-hand side (types (1) and (2) below) or the right-hand side (types (3) and (4) below) of \eqref{eq:flip}:
\begin{enumerate}
\item
One-box $B=b_{i, j}$ with $r(B) = k$ and $h_{i-1, j}, v_{i, j-1} < k-1$. Then $h_{i, 0}=v_{0, j} =k$, $h_{i, n}, v_{m, j} \geq k+1$.

\item
Zero-box $B=b_{i, j}$ with $r(B) = k$ and $h_{i-1, j}, v_{i, j-1} < k-1$. Then $h_{i, 0}= h_{i, n} = v_{0, j} = v_{m, j} = k$.

\item
One-box $B=b_{i, j}$ with $\overline{r}(B) = k-1$ and $h_{i+1, j+1}, v_{i+1, j+1} > k$. Then $h_{i, n}=v_{m, j} = k-1$, $h_{i, 0}, v_{0, j} \leq k-2$.

\item
Zero-box $B=b_{i, j}$ with $\overline{r}(B) = k-1$ and $h_{i+1, j+1}, v_{i+1, j+1} > k$. Then $h_{i, 0}= h_{i, n} = v_{0, j} = v_{m, j} = k-1$.

\end{enumerate}
It is an easy (though somewhat tedious) exercise to check that only one of the four types can be present, and there can be only one box of this type. For instance, suppose that there are $B_{1}=b_{i_{1}, j_{1}}$ of type $1$ and $B_{2}=b_{i_{2}, j_{2}}$ of type $4$. Then $h_{i_{1}, 0} = k > k-1 = h_{i_{2}, 0} \Rightarrow i_{2} < i_{1}$, $v_{0, j_{1}} = k > k-1 = v_{0, j_{2}} \Rightarrow j_{2} < j_{1}$. But then $k-1 = h_{i_{2}, j_{2}+1} \leq h_{i_{1}-1, j_{1}} < k-1$. Contradiction.

Consider the corresponding cases:
\begin{enumerate}
\item
$g_{k}=1$, in $f_k$ the contribution from the box  is $\frac{1}{1-C_{k}}$, the contribution of $\psi_{\mu/\lambda}$ is $1-C_{k}$, the contributions of $\phi_{\nu/\la}$, $\phi_{\rho/\mu}^{-1}$, $\psi_{\rho/\la}^{-1}$ are $1$.

\item
$g_{k}=1$, in $f_k$ the contribution from the box  is $\frac{1 - t^{-1}C_{k}}{1-C_{k}}$, the contribution of $\psi_{\mu/\lambda}$ is $1-C_{k}$, the contribution of  $\psi_{\rho/\la}^{-1}$ is $\frac{1}{1-t^{-1}C_{k}}$, the contributions of $\phi_{\nu/\la}$, $\phi_{\rho/\mu}^{-1}$ are $1$.

\item
$g_{k} = \frac{1}{1-tC_{k}}$,  in $f_k$ the contribution of $\phi_{\rho/\mu}^{-1}$ is also $\frac{1}{1-tC_{k}}$, the contributions of $\phi_{\nu/\la}$, $\psi_{\mu/\la}$, $\psi_{\rho/\nu}^{-1}$ are $1$.

\item
$g_{k} = \frac{1-tC_{k}}{1-t^{2}C_{k}}$, in $f_k$ the contribution of $\phi_{\rho/\mu}^{-1}$ is $\frac{1}{1-t^{2}C_{k}}$, the contribution of $\phi_{\nu/\la}$ is $1-tC_{k}$, the contributions of $\psi_{\mu/\la}$, $\psi_{\rho/\nu}^{-1}$ are $1$.

\end{enumerate}

If none of the four types of boxes are present, but not all contributions from $\phi_{\nu/\la}$, $\psi_{\mu/\la}$, $\phi_{\rho/\mu}^{-1}$, $\psi_{\rho/\nu}^{-1}$  are $1$, then one of the following cases takes place:
\begin{enumerate}
\item
$h_{i, 0} = h_{i, n} = k-1 < h_{i+1, 0}-1, h_{i+1, n}-1$ for some $i$, and no $v_{0, j}$ or $v_{m, j}$ is $k$ or $k-1$. Then the contribution of $\phi_{\nu/\la}$ is $1-tC_{k}$, the contribution of $\phi_{\rho/\mu}^{-1}$ is $\frac{1}{1-tC_{k}}$, the contributions of   $\psi_{\mu/\la}$, , $\psi_{\rho/\nu}^{-1}$ are $1$.
\item
$v_{0, j} = v_{m, j} = k > v_{0, j-1}+1, v_{m, j-1}+1$ for some $j$, and no $h_{i, 0}$ or $h_{i, n}$ is $k-1$ or $k$. Then the contribution of $\psi_{\mu/\la}$ is $1-C_{k}$, the contribution of $\psi_{\rho/\nu}^{-1}$ is $\frac{1}{1-C_{k}}$, the contributions of  $\phi_{\nu/\la}$, $\phi_{\rho/\mu}^{-1}$ are $1$.
\item
$h_{i, 0} = h_{i, n} = k-1 < h_{i+1, 0}-1, h_{i+1, n}-1$ and $v_{0, j} = v_{m, j} = k > v_{0, j-1}+1, v_{m, j-1}+1$ for some $i, j$. Then the contribution of $\psi_{\mu/\la}$ is $1-C_{k}$, the contribution of $\psi_{\rho/\nu}^{-1}$ is $\frac{1}{1-tC_{k}}$, the contribution of  $\phi_{\nu/\la}$ is $1-tC_{k}$, the contribution of $\phi_{\rho/\mu}^{-1}$ is $\frac{1}{1-C_{k}}$.
\item
$h_{i, 0} = v_{0, j} = k-1<h_{i+1, 0}-1, v_{0, j+1}-1$ for some $i, j$. Then there is a (unique) one-box $B=b_{I, J}$ with $r(B) = k-1$ such that $h_{I, n} = v_{m, J} = k$, and $v_{m, J-1} \le v_{I, J-1}<k-1$. Then the contribution of $\phi_{\nu/\la}$ is $1-tC_{k}$, the contribution of $\psi_{\rho/\nu}^{-1}$ is $\frac{1}{1-tC_{k}}$, the contributions of  $\phi_{\nu/\la}$, $\phi_{\rho/\mu}^{-1}$ are $1$.

\end{enumerate}

This concludes the proof that $f_{k} = g_{k}$ for any $k$, which implies the proposition.

\end{proof}

\subsection{Another description of the dynamics}
\label{sec:1cDescr}

In this section we present a slightly different description of $U(\mu \to \rho \mid \lambda \to \nu)$ by essentially ``collapsing'' the vertical coordinate. This description seems simpler than the one from Section \ref{sec:descRSK}. However, it is less suitable for certain questions: In particular, the symmetry property \ref{prop:symmetry} is not immediate from such a description.

For $1 \leq i \leq m$, let $\mathcal{U}_{i} = \prod_{j=0}^{n-1}w(b_{i-1, j})$, $d_{i} = h_{i-1, 0}$, $u_{i} = h_{i-1, n}$. Then $U(\mu \to \rho \mid \lambda \to \nu) = \prod_{i=1}^{m}\mathcal{U}_{i}$. We obtain $\nu$ from $\lambda$ by making $m$ particle moves, and we can say that they consecutively trigger $m$ random particle moves for $\mu$, such that the resulting signature is $\rho$ with probability $U(\mu \to \rho \mid \lambda \to \nu)$. More precisely,  first $\Lambda(0, 0) \to \Lambda(1, 0)$ triggers  $\Lambda(0, n) \to \Lambda(1, n)$  with probability $\mathcal{U}_{1}$, then $\Lambda(1, 0) \to \Lambda(2, 0)$ triggers  $\Lambda(1, n) \to \Lambda(2, n)$ with probability $\mathcal{U}_{2}$, $\ldots$, finally $\Lambda(m-1, 0) \to \Lambda(m, 0)$  triggers  $\Lambda(m-1, n) \to \Lambda(m, n)$ with probability $\mathcal{U}_{m}$. Here transition $\Lambda(i-1, 0) \to \Lambda(i, 0)$ means that on the lower level the first particle at position $d_i$ moves by $+1$, while transition $\Lambda(i-1, n) \to \Lambda(i, n)$ means that on the upper level the first particle at position $u_i$ moves by $+1$. Let $C(k)=t ^{\la'_{k}-\la'_{k+1}}$, $\widehat{C}(k)=t ^{\mu'_{k}-\mu'_{k+1}}$ and we will use convention $d_{0} = u_{0} = -\infty$. Let $k_{i}$ be the smallest integer $k \geq d_{i}$ such that $\Lambda(i-1, n)'_{k} > \Lambda(i-1, n)'_{k+1} = \Lambda(i-1, 0)'_{k+1}$. In the Schur case ($t=0$) the move of the first particle at position $d_i$ on the lower level by $+1$ necessarily triggers the move of the first particle at position $k_i$ on the upper level by $+1$ (see e.g. \cite[Section 4.3]{MP}).
\begin{proposition}
\begin{equation*}
\mathcal{U}_{i}:=
    \begin{cases}
     \frac{1-t}{1-C(d_{i})}1_{u_{i} = k_{i}} + \frac{t-C(d_{i})}{1-C(d_{i})}1_{u_{i} = d_{i}} & \parbox[t]{.6\columnwidth}{if $u_{i-1} < d_{i}-1$, $\widehat{C}(d_{i}) = t^{-1}C(d_{i})$.} \\ \\
      (1-t)1_{u_{i} = k_{i}} + t1_{u_{i} = d_{i}} & \parbox[t]{.6\textwidth}{if $k_{i} > d_{i}$ and either
   \begin{enumerate}
   \item $u_{i-1} < d_{i}-1$, $\widehat{C}(d_{i}) = C(d_{i})$,  or
   \item $u_{i-1}=d_{i-1} = d_{i}-1$, $\widehat{C}(d_{i}) = t^{-1}C(d_{i})$.
   \end{enumerate}}
  \\ \\
         1_{u_{i} = k_{i}} & \parbox[t]{.6\textwidth}{if either
   \begin{enumerate}
   \item $u_{i-1} \geq d_{i}$, or
   \item $k_{i} = d_{i}$, or
   \item $d_{i-1} < u_{i-1} = d_{i}-1$.
   \end{enumerate}} \\ \\
    1_{u_{i} = d_{i}} & \parbox[t]{.6\textwidth}{if $u_{i-1} = d_{i-1} = d_{i} -1$, $k_{i} > d_{i}$, $\widehat{C}(d_{i}) = C(d_{i})$}.
            \end{cases}
\end{equation*}

\end{proposition}
\begin{proof}
Let $\bar{\la} = \Lambda(i-1, 0)$, $\bar{\mu} = \Lambda(i-1, n)$. Then $\bar{\la} \prec \bar{\mu}$.
$k_{i} = d_{i} \Leftrightarrow d_{i} = \bar{\la}_{h} = \bar{\mu}_{h} < \bar{\la}_{h-1}$ for some
$h \Leftrightarrow  v_{i-1, j} \neq d_{i}$ for all $0 \leq j \leq n-1$. In such case all the boxes in the $i$-th column are trivial, so  $\mathcal{U}_{1} = 1$ and $h_{i-1, j} = d_{i}$ for all $0 \leq j \leq n-1$, in particular, $u_{i} = d_{i}$. For the following cases suppose that $k_{i} > d_{i}$, i.e. $v_{i-1, j} = d_{i}$ for some (unique) $0 \leq j \leq n-1$. Then $B = b_{i-1, j}$ is either a zero-box or a one-box, and $\mathcal{U}_{i} = w(B)$.
Suppose first that $B$ is a zero-box. Then there are no other nontrivial boxes in the $i$-th column. Then $u_{i} = d_{i}$ and there are several possibilities:
\begin{enumerate}
\item
$v_{i-1, j-1}, h_{i-2, j} < d_{i} -1$. Then $\mathcal{U}_{i} = \frac{t-C(d_{i})}{1-C(d_{i})}$;
$u_{i-1} = h_{i-2, j} < d_{i}-1$; $v_{0, j} = v_{i-1, j} = d_{i}$ and $v_{0, j-1} \leq v_{i-1, j-1} < d_{i}-1$, so
$\widehat{C}(d_{i}) = t^{-1}C(d_{i})$.

\item
$v_{i-1, j-1} < h_{i-2, j} = d_{i} -1$. Then $\mathcal{U}_{i} = t$; $d_{i-1} = u_{i-1} = h_{i-2, j} = d_{i}-1$; $v_{0, j} = v_{i-1, j} = d_{i}$ and $v_{0, j-1} \leq v_{i-1, j-1} < d_{i}-1$, so
$\widehat{C}(d_{i}) = t^{-1}C(d_{i})$.

\item
$h_{i-2, j} < v_{i-1, j-1} = d_{i} -1$. Then $\mathcal{U}_{i} = t$;  $u_{i-1} = h_{i-2, j} < d_{i}-1$;
$v_{0, j} = v_{i-1, j} = d_{i}$ and $v_{0, j-1} = v_{i-1, j-1} = d_{i}-1$, so
$\widehat{C}(d_{i}) = C(d_{i})$.

\item
$h_{i-2, j} = v_{i-1, j-1} = d_{i} -1$. Then $\mathcal{U}_{i} = 1$; $b_{i-2, j-1}$ is also a zero-box, so $u_{i-1} = d_{i-1} = d_{i}-1$;
$v_{0, j} = v_{i-1, j} = d_{i}$ and $v_{0, j-1} = v_{i-1, j-1} = d_{i}-1$, so
$\widehat{C}(d_{i}) = C(d_{i})$.

\end{enumerate}

Suppose now that $B$ is a one-box, the lowest of the stack of $r$ consecutive one-boxes in the $i$-th column. Then $v_{i-1, j+\ell} = h_{i-1, j+\ell+1}-1 = d_{i} + \ell$ for any $0 \leq \ell \leq r-1$,  $v_{i-1, j+r} > d_{i} +r$. So $u_{i} = h_{i-1, j+r} = d_{i}+r = k_{i}$, and there are several possibilities

\begin{enumerate}
\item
$v_{i-1, j-1}, h_{i-2, j} < d_{i} -1$. Then $\mathcal{U}_{i} = \frac{1-t}{1-C(d_{i})}$; $u_{i-1} = h_{i-2, j} < d_{i}-1$; $v_{0, j} = v_{i-1, j} = d_{i}$ and $v_{0, j-1} \leq v_{i-1, j-1} < d_{i}-1$, so
$\widehat{C}(d_{i}) = t^{-1}C(d_{i})$.

\item
$h_{i-2, j} < v_{i-1, j-1} = d_{i} -1$. Then $\mathcal{U}_{i} = 1-t$;
$u_{i-1} = h_{i-2, j} < d_{i}-1$;
$v_{0, j} = v_{i-1, j} = d_{i}$ and $v_{0, j-1} = v_{i-1, j-1} = d_{i}-1$, so
$\widehat{C}(d_{i}) = C(d_{i})$.

\item
$v_{i-1, j-1} < h_{i-2, j} = h_{i-2, j+1} = d_{i} -1$. Then $\mathcal{U}_{i} = 1-t$;  $d_{i-1} = h_{i-2, j} = h_{i-2, j+1} = u_{i-1} = d_{i}-1$; $v_{0, j} = v_{i-1, j} = d_{i}$ and $v_{0, j-1} \leq v_{i-1, j-1} < d_{i}-1$, so
$\widehat{C}(d_{i}) = t^{-1}C(d_{i})$.

\item
$v_{i-1, j-1} < h_{i-2, j} = h_{i-2, j+1}-1 = d_{i} -1$. Then $\mathcal{U}_{i} = 1$; $u_{i-1} \geq h_{i-2, j+1} = d_{i}$.

\item
$h_{i-2, j} = v_{i-1, j-1} = d_{i} -1$. Then $\mathcal{U}_{i} = 1$; $b_{i-2, j-1}$ is also a one-box, so $d_{i-1} \leq h_{i-2, j-1} < d_{i}-1 \leq u_{i-1}$.

\end{enumerate}

\end{proof}

See Figure \ref{fig:1coordEx} for an example of this description.

\begin{figure}
\begin{tikzpicture}[>=stealth,scale=0.7]
\node at (0,0) {(3,3,1,0)};
\draw (1.5,0) -- (3.5,0); \node at (2.5,-0.4) {$1$};
\node at (5,0) {(3,3,2,0)};
\draw (6.5,0) -- (8.5,0); \node at (7.5,-0.4) {$3$};
\node at (10,0) {(4,3,2,0)};
\draw (11.5,0) -- (13.5,0); \node at (12.5,-0.4) {$4$};
\node at (15,0) {(5,3,2,0)};
\node at (2.5,0.8) {$\mathbf{t}$}; \node at (7.5,0.8) {$\mathbf{\frac{1-t}{1-t^2}}$};
\node at (0,2) {(5,3,\textcolor{red}{2},\textcolor{red}{1})};
\draw (1.5,2) -- (3.5,2); \node at (2.5,1.6) {$1$};
\node at (5,2) {(\textcolor{red}{5},\textcolor{red}{3},2,2)};
\draw (6.5,2) -- (8.5,2); \node at (7.5,1.6) {$5$};
\node at (10,2) {(\textcolor{red}{6},3,2,2)};
\draw (11.5,2) -- (13.5,2); \node at (12.5,1.6) {$6$};
\node at (15,2) {(7,3,2,2)};
\end{tikzpicture}
\caption{An example of one coordinate description. The coordinates/particles which can jump with strictly positive probability are colored in red. Probabilities of jumps are shown in bold.}
\label{fig:1coordEx}
\end{figure}

\subsection{HL-RSK algorithm with input}
\label{sec:rsk-input}

Let us give a useful (and more classical) variation of Definition \ref{def:RSK-main1}. Consider signatures $\la, \nu \in \Sig_{N-1}$, $\mu, \rho \in \Sig_N$, $\la \prec \mu \prec \rho$, $\la \prec \nu \prec \rho$, and $r \in \Z_{\ge 0}$ such that $|\rho| - |\mu| = |\nu| - |\la|+r$. Define $\tilde \nu : = \{ -V \} \cup \nu$, $\tilde \la := \{ -V-r \}$, $\tilde \nu, \tilde \la \in \Sig_N$, for an arbitrary integer $(-V)$ which is strictly smaller than all coordinates of $\mu$ and $\nu$.

Coefficients $U^r (\mu \to \rho \mid \la \to \nu)$ are defined as
\begin{equation}
\label{eq:brrr}
U^r (\mu \to \rho \mid \la \to \nu) := U (\mu \to \rho \mid \tilde \la \to \tilde \nu);
\end{equation}
it is clear that any choice of $V$ produces the same number.

Similarly, let us define a coefficient $\hat U^r ( \nu \to \la \mid \rho \to \mu )$ via:
\begin{equation}
\label{eq:brrr2}
\hat U^r ( \nu \to \la \mid \rho \to \mu ) := U ( - \tilde \nu \to - \tilde \la \mid - \rho \to - \mu)
\end{equation}

For Young diagrams $\mu, \nu, \la, \rho$, and $r \in \Z_{\ge 0}$, with the same interlacing conditions and $|\rho| - |\mu| = |\nu| - |\la|+r$ we again use \eqref{eq:brrr} and \eqref{eq:brrr2} to define $U^r (\mu \to \rho \mid \la \to \nu)$ and $\hat U^r ( \nu \to \la \mid \rho \to \mu )$. For that we need to consider these Young diagrams as signatures $\nu, \la \in \Sig_{N-1}$ and $\rho, \mu \in \Sig_N$ for $N$ large enough. This means that we add certain amount of zeros to the lengths of rows of $\mu, \nu, \la, \rho$; it is a direct check that the obtained quantities do not depend on the choice of $N$.

Equation \eqref{eq:sumUeq1} implies that for any choice of $\la, \mu, \nu$ and $r$ we have:
\begin{equation}
\label{eq:sumUeq1A}
\sum_{\rho \in \Y} U^r (\mu \to \rho \mid \la \to \nu) =1,
\end{equation}
and also for any choice of $\mu, \nu, \rho$ we have
\begin{equation}
\label{eq:sumUeq1B}
\sum_{\la \in \Y, r \in \Z_{\ge 0}} \hat U^r ( \nu \to \la \mid \rho \to \mu ) =1.
\end{equation}
Thus, we have a (randomized) algorithm which takes as an input three Young diagrams $\la, \mu, \nu$, $\la \prec \mu$, $\la \prec \nu$, and a non-negative integer $r$ and outputs a random Young diagram $\rho$ with probability $U^r (\mu \to \rho \mid \la \to \nu)$. This is a (randomized) version of the Fomin growth diagram for the Hall-Littlewood functions. Similarly, we have a (randomized) algorithm with input $\mu, \nu, \rho$, $\mu \prec \rho$, $\nu \prec \rho$, which outputs a pair $( \la, r) \in \Y \times \Z_{\ge 0}$ with probability $\hat U^r ( \nu \to \la \mid \rho \to \mu )$. In our setting this algorithm is the ``correct'' generalization of the inversion of the RSK algorithm.

Properties of these randomized algorithms for Young diagrams easily follow from the established properties of the algorithm for signatures. Let us formulate them.


Equations \eqref{eq:def-psi}, \eqref{eq:def-phi} imply that
\begin{equation*}
\psi_{\mu / \tilde \la} = (1-t) \psi_{\mu / \la}, \qquad \phi_{\tilde \nu / \tilde \la} = (1 - t \mathbf{1}_{r \ge 1}) \phi_{\nu / \la}, \qquad \psi_{\rho / \tilde \nu} = (1-t) \psi_{\rho / \nu}.
\end{equation*}
Therefore, Proposition \ref{lemma:flip} entails that for any Young diagrams $\la,\mu,\nu,\rho$ and $r \in \Z_{\ge 0}$ we have
\begin{equation}
\label{eq:flipY}
\left( 1 - t \mathbf{1}_{r \ge 1} \right) \phi_{\nu / \la} \psi_{\mu / \la} U^r (\mu \to \rho \mid \la \to \nu) = \phi_{\rho / \mu} \psi_{\rho / \nu} \hat U^r ( \nu \to \la \mid \rho \to \mu )
\end{equation}

Proposition \ref{prop:symmetry} implies that
\begin{equation*}
U^r (\mu \to \rho \mid \la \to \nu) = U^r (\nu \to \rho \mid \la \to \mu),
\end{equation*}
while Proposition \ref{prop:Markov-proj} implies that the projection of a randomized RSK algorithm to any fixed number of first columns of Young diagrams is Markovian.

\subsection{Hall-Littlewood operators and a bijective proof of the Cauchy identity}
\label{sec:bij-proof}

Once we know a Hall-Littlewood RSK algorithm satisfying the flip property (Proposition \ref{lemma:flip}), one can prove the Cauchy identity in form \eqref{eq:skew-Cauchy-1varBB} via the following short computation:
\begin{multline}
\label{eq:bijectiveProof}
\sum_{\la \in Sig_{N}: \la \prec \mu, \la \prec \nu, |\mu| - |\la| = k, |\nu|-|\la| = l} \phi_{\nu / \la} \psi_{\mu / \la} = \sum_{\la \in Sig_{N}: \la \prec \mu, \la \prec \nu, |\mu| - |\la| = k, |\nu|-|\la| = l} \phi_{\nu / \la} \psi_{\mu / \la} \sum_{\rho} U(\mu \to \rho \mid \la \to \nu)
\\ = \sum_{\rho \in Sig_{N}: \mu \prec \rho, \nu \prec \rho, |\rho|-|\nu| = k, |\rho|-|\mu|=l} \psi_{\rho / \nu} \phi_{\rho / \mu} \sum_{\la} U( - \nu \to -\la \mid -\rho \to -\mu) = \sum_{\rho \in Sig_{N}: \mu \prec \rho, \nu \prec \rho, |\rho|-|\nu| = k, |\rho|-|\mu|=l} \psi_{\rho / \nu} \phi_{\rho / \la},
\end{multline}
where the first and the third equality follow from \eqref{eq:sumUeq1} and the second equality follows from the flip property. Note that in the second equality both sides contain exactly the same terms: This is the reason why we call this proof bijective.

Similarly, the Cauchy identity in more standard form \eqref{eq:skew-Cauchy-1varAA} can be proved via
\begin{multline}
\label{eq:bijectiveProof2}
\sum_{r \ge 0} ( 1 - t \mathbf{1}_{r \ge 1}) \sum_{\la \in Sig_{N-1}: \la \prec \mu, \la \prec \nu, |\mu|- |\la| = k-r, |\nu|-|\la| = l-r} \phi_{\nu / \la} \psi_{\mu / \la}
\\ = \sum_{r \ge 0} ( 1 - t \mathbf{1}_{r \ge 1}) \sum_{\la \in Sig_{N-1}: \la \prec \mu, \la \prec \nu, |\mu|- |\la| = k-r, |\nu|-|\la| = l-r} \phi_{\nu / \la} \psi_{\mu / \la} \sum_{\rho} U^r (\mu \to \rho \mid \la \to \nu)
\\ = \sum_{\rho \in Sig_{N}: \mu \prec \rho, \nu \prec \rho, |\rho|-|\mu| = k, |\rho|-|\nu|=l} \psi_{\rho / \nu} \phi_{\rho / \mu} \sum_{\la,r} \hat U^r ( \nu \to \la \mid \rho \to \mu) = \sum_{\rho \in Sig_{N}: \mu \prec \rho, \nu \prec \rho, |\rho|-|\nu| = k, |\rho|-|\mu|=l} \psi_{\rho / \nu} \phi_{\rho / \mu},
\end{multline}
where the first equality follows from \eqref{eq:sumUeq1A}, the third equality follows from \eqref{eq:sumUeq1B} and the second equality
is a term by term equality which follows from the flip property in form \eqref{eq:flipY}.

Also, using \eqref{eq:sumUeq1A} and summing over $k$ and $l$, we get
\begin{multline}
\label{eq:quadCauchyRSK}
\sum_{\rho \in \Sig_N: \mu \prec \rho, \nu \prec \rho} \psi_{\rho / \nu} \phi_{\rho / \mu} x^{|\rho|-|\nu|} y^{|\rho|-|\mu|} = \sum_{r \ge 0} ( 1 - t \mathbf{1}_{r \ge 1} ) (xy)^r \sum_{\la \in \Sig_{N-1}: \la \prec \mu, \la \prec \nu} \phi_{\nu / \la} \psi_{\mu / \la} x^{|\mu|-|\la|} y^{|\nu|-|\la|} \\ \times \sum_{\rho \in Sig_{N}: \mu \prec \rho, \nu \prec \rho; |\rho|-|\mu|=|\nu|-|\la|+r} U^r ( \mu \to \rho \mid \la \to \nu);
\end{multline}
we will use this equality in the next section.

Let us reformulate the skew Cauchy identity in one more way. Essentially following Fomin \cite{F}, let us consider the vector space $\Q [t;\Sig_N]$ over $\Q(t)$ with a linear basis consisting of all signatures of length $N$. Let us denote by $\bar \kappa$ the vector which corresponds to signature $\kappa$. Define a family of linear operators on this space by the following rules:
\begin{equation*}
A_k := \bar \la \to \sum_{ \mu \in \Sig_N: \la \prec \mu, |\mu| - |\la| =k} \phi_{ \mu / \la} \bar \mu, \qquad
B_l := \bar \mu \to \sum_{\la \in \Sig_N : \la \prec \mu, |\mu| - |\la| =l} \psi_{\mu / \la} \bar \la, \qquad k,l \in \Z_{\ge 0}.
\end{equation*}
For formal variables $x$ and $y$ define \textit{Hall-Littlewood} raising and lowering operators:
\begin{equation*}
A(x):= \sum_{k=0}^{\infty} A_k x^k, \qquad B(y) := \sum_{l=0}^{\infty} B_l y^l.
\end{equation*}

Then the symmetry properties of the Hall-Littlewood functions and \eqref{eq:skew-Cauchy-1varBB} can be equivalently rewritten as properties of the Hall-Littlewood operators:
\begin{equation}
\label{eq:commut-op}
A(x_1) A(x_2) = A(x_2) A(x_1), \qquad B(y_1) B(y_2) = B(y_2) B(y_1), \qquad A(x) B(y) = A(y) B(x).
\end{equation}
Informally, the Hall-Littlewood RSK algorithm provides a combinatorial refinement of the commutation relation between $A(x)$ and $B(y)$, providing a bijective proof for it. Another well-known combinatorial refinement of the commutation of operators of this sort is the Yang-Baxter / star-triangle relation. We plan to address connections between these two structures in a subsequent publication.

\subsection{Hall-Littlewood RSK field and sampling of Hall-Littlewood processes}
\label{sec:field}

For any $\la \in Sig_{N-1}$, $\mu, \nu \in Sig_N$ such that $\la \prec \mu$, $\la \prec \nu$, and any $r \in \Z_{\ge 0}$ we have constructed a random signature of length $N$ which takes value $\rho$ with probability $U^r (\mu \to \rho \mid \la \to \nu)$ (see \eqref{eq:brrr}). Iteration of this construction for all values of $N$ leads to an interesting probabilistic model described in this section.

Let $\{a_i \}_{i\ge 1}$, $\{ b_j \}_{j \ge 1}$ be two sequences of positive reals such that $a_i b_j <1$ for all $i,j$. Let us call a \textit{state} a map $\Lambda: \Z_{\ge 0} \times \Z_{\ge 0} \to \bigsqcup_{j \ge 0} \Sig_j$, such that $\Lambda (i,j) \in \Sig_j$ and $\Lambda (i,j) \prec \Lambda (i+1,j)$, $\Lambda (i,j) \prec \Lambda (i, j+1)$, for all $(i,j) \in \Z_{\ge 0} \times \Z_{\ge 0}$.

\begin{definition}[Hall-Littlewood RSK field]
\label{def:integr-RSK-field}

The \textit{integrable HL-RSK field} is a random state (see the definition above) constructed in the following way:

1) Set $\Lambda(i,0) = \emptyset$, $\Lambda(0,j) = (0,\dots, 0) \in \Sig_j$ for all $i, j \in \Z_{\ge 0}$.

2) Take a sample of a family of independent random variables $\{ r_{i,j} \}_{i,j \ge 1}$, where the individual distributions are given by
$$
P(r_{i,j}= d) = (1 - t  1_{d \ge 1})(a_i b_j)^d \frac{1 - a_i b_j}{1 - t a_i b_j}, \qquad d=0,1,2, \dots
$$

3) Iterate the following procedure for $n=2,3, \dots$.

Assume that values of $\Lambda (i,j)$ are defined for all $(i,j)$ such that $i+j < n$, and define them for $(i,j)$ such that $i+j=n$. Note that signatures $\Lambda(i-1,j-1)$, $\Lambda( i,j-1)$, and $\Lambda (i-1,j)$ are already defined, so define $\Lambda (i, j)$ as the output of the HL-RSK algorithm (in the form \eqref{eq:brrr}) with inputs  $\Lambda(i-1,j-1)$, $\Lambda( i,j-1)$, $\Lambda (i-1,j)$, and $r_{i,j}$. The steps of the HL-RSK algorithm are independent for distinct pairs $(i,j)$.
\end{definition}

\begin{proposition}
\label{prop:pathHLprocess}
Let $m_1 \ge m_2 \ge \dots \ge m_k$ and $n_1 \le n_2 \dots \le n_k$ be positive integers. Then the collection of signatures
\begin{multline*}
\varnothing = \Lambda (m_1,0) \prec \Lambda (m_1, n_1) \succ \Lambda (m_2, n_1) \prec \Lambda (m_2, n_2) \succ \dots \\ \succ \Lambda (m_k, n_{k-1}) \prec \Lambda (m_k, n_k) \succ \Lambda (0, n_k) = \underbrace{(0, \ldots, 0)}_{n_{k}}
\end{multline*}
is distributed as the Hall-Littlewood process determined by specializations
\begin{align*}
\rho_1^- = \rho (a_{m_2 + 1},\dots,a_{m_1}), \quad \rho_2^- = \rho (a_{m_3 + 1}, \dots, a_{m_{2}}), \quad \ldots , \quad \rho_k^- = \rho (a_1,\dots,a_{m_k}), \\ \rho_1^+ = \rho (b_{1},\dots,b_{n_1}), \quad \rho_2^+ = \rho (b_{n_{1}+1}, \dots,b_{n_{2}}), \quad \ldots, \quad \rho_k^+ = \rho (b_{n_{k-1}+1},\dots,b_{n_k}).
\end{align*}
\end{proposition}

\begin{proof}

Consider an arbitrary down-right path $\mathcal S$ in $\Z_{\ge 0} \times \Z_{\ge 0}$ starting on the $y$-axis and ending on the $x$-axis. It is enough to prove the proposition in the case when the set $\{ (m_i, n_i) \}$ coincides with the set of all integer points of $\mathcal S$: The claim for an arbitrary subset of these points follows from the combinatorial formulae for the Hall-Littlewood polynomials \eqref{eq:HLcombFormQ}, \eqref{eq:HLcombFormP}.

Note that if the set $\{( m_i, n_i) \}$ coincides with the set of all integer points from some down-right path, then we need to show that probability of any configuration of signatures is a certain product of $\phi$'s and $\psi$'s from \eqref{eq:def-psi}, \eqref{eq:def-phi}, since all the skew Hall-Littlewood functions entering the definition depend on one variable only. For example, for the down-right path
$$
(0,2) \to (1,2) \to (1,1) \to (3,1) \to (3,0)
$$
we need to prove
\begin{multline*}
\mathrm{Prob} ( \Lambda(3,1) = \la(3,1), \Lambda(1,2) = \la(1,2), \Lambda(1,1) = \la(1,1), \Lambda(1,2) = \la(1,2) ) \\ = \phi_{ \la(1,2) / (0, 0)} a_1^{|\la(1,2)|} \psi_{ \la(1,2) / \la(1,1) } b_2^{ |\la(1,2)|- |\la(1,1)|} \phi_{ \la(2,1) / \la(1,1)} a_2^{|\la(2,1)| - |\la(1,1)|} \\ \times \phi_{ \la(3,1) / \la(2,1)} a_3^{|\la(3,1)| - |\la(2,1)|} \psi_{ \la(3,1) / \varnothing } b_1^{ |\la(3,1)|},
\end{multline*}
for all choices of signatures $\la(i,j)$.

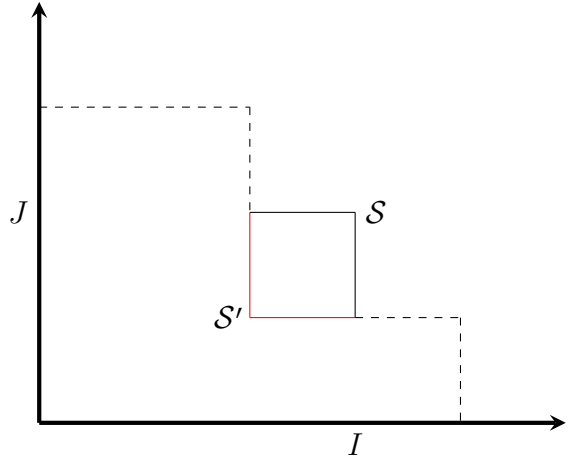
\begin{figure}
\begin{tikzpicture}[>=stealth,scale=0.7]
\draw [ultra thick,->] (0,0) -- (10,0); \draw [ultra thick,->] (0,0) -- (0,8);
\draw [dashed] (0,6) -- (4,6); \draw [dashed] (4,6) -- (4,4);
\draw (4,4) -- (6,4); \draw (6,4) -- (6,2); \draw [dashed] (6,2) -- (8,2); \draw [dashed] (8,2) -- (8,0);
\draw [color=red] (4,4) -- (4,2); \draw [color=red] (4,2) -- (6,2);
\node at (6.4,4) {$\mathcal S$}; \node at (3.6,2) {$\mathcal S'$};
\node at (6,-0.4) {$I$}; \node at (-0.4,4) {$J$};
\end{tikzpicture}
\caption{Paths $\mathcal S$ and $\mathcal S'$ differ by the box with the north-east corner $(I,J)$. Dotted lines represent (arbitrary) common parts of these paths. }
\label{fig:HLprocProof}
\end{figure}

Returning to a general path $\mathcal S$, we prove the statement by induction on the number of points from $\Z_{\ge 0} \times \Z_{\ge 0}$ which lie below $\mathcal S$. The base case is immediate. For the induction step, consider an arbitrary north-eastern corner of $\mathcal S$; denote its position by $(I,J)$. Consider the down-right path $\mathcal S'$ which is obtained from $\mathcal S$ by excluding the box with the north-eastern corner $(I,J)$ (see Figure \ref{fig:HLprocProof} ). Applying the induction hypothesis to the path $\mathcal S'$, we get that the probability of a specific configuration $\{ \la(i,j) \}$ along $\mathcal S'$ is
\begin{multline*}
\mathrm{Prob} ( \Lambda (i,j) = \la (i,j); (i,j) \in \mathcal S') \\ = A \phi_{ \la ( I,J-1 ) / \la (I-1,J-1)} a_I^{|\la ( I,J-1 )| - |\la (I-1,J-1)| } \psi_{ \la ( I-1,J ) / \la (I-1,J-1)} b_J^{|\la ( I-1,J )| - |\la (I-1,J-1)|} B,
\end{multline*}
where $A,B$ come from the definition of the Hall-Littlewood process and do not depend on $\la_{I-1,J-1}$. Next, Definition \ref{def:integr-RSK-field} implies
\begin{multline*}
\mathrm{Prob} ( \Lambda (i,j) = \la (i,j); (i,j) \in \mathcal S' \ \text{or} \ (i,j) =(I,J) ) \\ = A \phi_{ \la ( I,J-1 ) / \la (I-1,J-1)} a_I^{|\la ( I,J-1 )| - |\la (I-1,J-1)| } \psi_{ \la ( I-1,J ) / \la (I-1,J-1)} b_J^{|\la ( I-1,J )| - |\la (I-1,J-1)|} B \\ \times (1 - t \mathbf{1}_{R \ge 1}) (a_I b_J)^R \frac{1-a_I b_J}{1-t a_I b_J} U^R (\la(I-1,J) \to \la (I,J) \mid \la(I-1,J-1) \to \la (I,J-1)),
\end{multline*}
where $R: = |\la (I,J)| - |\la (
I-1,J)| - |\la (I,J-1)| + |\la (I-1,J-1)|$.

Summing over all $\la(I-1,J-1)$ and (crucially) using \eqref{eq:quadCauchyRSK}, we obtain
\begin{multline*}
\mathrm{Prob} ( \Lambda (i,j) = \la (i,j); (i,j) \in \mathcal S) \\ = A \phi_{ \la ( I,J ) / \la (I-1,J)} a_I^{|\la ( I,J )| - |\la (I-1,J)| } \psi_{ \la ( I,J ) / \la (I,J-1)} b_J^{|\la ( I,J )| - |\la (I,J-1)|} B,
\end{multline*}
which concludes the proof.

\end{proof}

As we have shown in Section \ref{sec:2}, there are various formulas for observables of the Hall-Littlewood processes which should provide a lot of information about various limits of this probabilistic model. Thus, we refer to the HL-RSK field from Definition \ref{def:integr-RSK-field} as \textit{integrable}.

\section{Degenerations}
\label{sec:degen}

Let $\{ \Lambda (i,j) \}_{i,j \in \Z_{\ge 0}}$ be the HL-RSK field (see Definition \ref{def:integr-RSK-field}).
In this section we consider its various degenerations.

\subsection{Schur RSK field}

For $t=0$ we obtain the well known object coming from the column insertion version of the classical RSK algorithm. The asymptotic analysis of this model was done in \cite{Joh} and many subsequent papers. Various degenerations lead to the last passage percolation problem and to the totally asymmetric simple exclusion process (TASEP). The HL-RSK field can be thought of as a natural one parameter deformation of these objects. In particular, it degenerates into the asymmetric simple exclusion process (ASEP), see Section \ref{sec:asep}.

\subsection{Half-continuous RSK field}

A ``continuous time'' version of RSK-algorithm was previously constructed in \cite{BP}. The corresponding field appears from the HL-RSK field as a result of the following limit transition. Consider $a_i = \eps \hat a_i$, $i=1,2,3,\dots$, and grid $\eps \Z_{\ge 0} \times \Z_{\ge 0}$ instead of $\Z_{\ge 0} \times \Z_{\ge 0}$. In $\eps \to 0$ limit the $x$-axis becomes continuous, while the $y$-direction remains discrete.
See \cite[Section 6]{BBW} for a more detailed discussion about this object and its relation to vertex models.

\subsection{Stochastic six vertex model}
\label{sec:st6v}

Let $d_0 (i,j)$ be the number of coordinates of $\Lambda(i,j)$ which are equal to 0. Equivalently, if we think about $\Lambda(i,j)$ as a Young diagram, then $d_0 (i,j)$ equals $j$ minus the length of the first column of $\Lambda(i,j)$. Note that the interlacing constraints entail $d_0 (i-1,j) \ge d_0 (i,j) \ge d_0 (i-1,j) -1$, $d_0 (i,j-1) \le d_0 (i,j) \le d_0 (i,j-1)+1$ (the asymmetry is the result of the fact that the lengths of signatures increase in $y$-direction, but not in $x$-direction).

Proposition \ref{prop:Markov-proj} implies that $ d_0 (i,j)$ depends on $d_0 (i-1,j)$, $d_0 (i,j-1)$, $d_0 (i-1,j-1)$ and $e_{i,j} := 1_{r_{i,j} \ge 1}$ in a Markovian way. More explicitly, using the description of the algorithm from Section \ref{sec:1cDescr} and $\mathrm{Prob} (e_{i,j} =0) = \frac{1-a_i b_j}{1 - t a_i b_j}$, $\mathrm{Prob} (e_{i,j} =1) = \frac{(1- t) a_i b_j}{1 - t a_i b_j}$, it is a direct check that we have the following rules for any $m \in \Z_{\ge 1}$:
\begin{itemize}

\item
If $d_0 (i-1,j-1) = m , d_0 (i,j-1) = m-1, d_0 (i-1,j) = m+1$, then $\mathrm{Prob} ( d_0 (i,j) = m) =1$.

\item
If $d_0 (i-1,j-1) = m , d_0 (i,j-1) = m, d_0 (i-1,j) = m$, then $\mathrm{Prob} ( d_0 (i,j) = m) =1$.

\item
If $d_0 (i-1,j-1) = m , d_0 (i,j-1) = m, d_0 (i-1,j) = m+1$, then $\mathrm{Prob} ( d_0 (i,j) = m+1) = \frac{1-a_i b_j}{1 - t a_i b_j}$, $\mathrm{Prob} ( d_0 (i,j) = m) = \frac{(1-t) a_i b_j}{1 - t a_i b_j}$.

\item
If $d_0 (i-1,j-1) = m, d_0 (i,j-1) = m-1, d_0 (i-1,j) = m$, then $\mathrm{Prob} ( d_0 (i,j) = m) = \frac{1-t}{1 - t a_i b_j}$, $\mathrm{Prob} ( d_0 (i,j) = m-1) = \frac{t (1- a_i b_j)}{1 - t a_i b_j}$.

\end{itemize}

This probabilistic model can be interpreted as an instance of the \textit{stochastic six vertex model} introduced in \cite{GS} and recently studied in \cite{BCG}. Namely, let us draw an edge between vertices $(i-1/2,j-1/2)$ and $(i-1/2,j+1/2)$ if $d_0 (i-1,j) \ne d_0 (i,j)$, and also draw an edge between vertices $(i-1/2,j-1/2)$ and $(i+1/2,j-1/2)$ if $d_0 (i,j-1) \ne d_0 (i,j)$. Then we obtain a random configuration on edges. The rules listed above imply that if we know whether the edges $(i-3/2, j-1/2) \to (i-1/2,j-1/2)$ and $(i-1/2, j-3/2) \to (i-1/2,j-1/2)$ are present in the configuration, the probabilities of the presence of edges $(i-1/2, j-1/2) \to (i-1/2,j+1/2)$ and $(i-1/2, j-1/2) \to (i+1/2,j-1/2)$ are determined and do not depend on the rest of the configuration. Thus, the probability of the whole configuration can be computed as the product of weights of vertices (=probabilities of local steps); the weights are depicted in Figure \ref{fig:1arrows}.

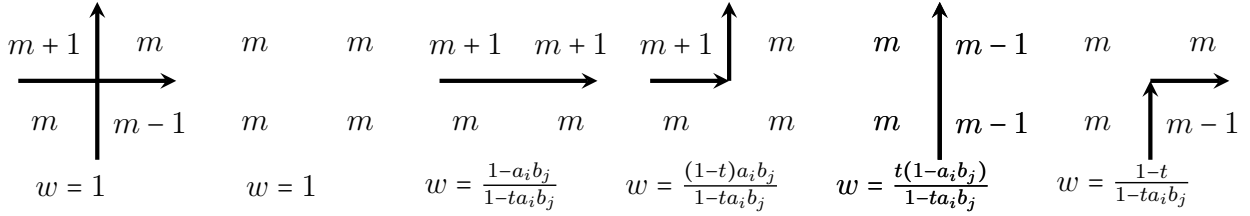
\begin{figure}
\begin{tikzpicture}[>=stealth,scale=0.7]
\draw[ultra thick,->] (0,3) -- (3,3);
\draw[ultra thick,->] (1.5,1.5) -- (1.5,4.5);
\node at (0.5,2.2) {$m$};
\node at (2.5,2.2) {$m-1$};
\node at (0.5,3.7) {$m+1$};
\node at (2.5,3.7) {$m$};
\node at (1,1) {$w = 1$};
\node at (4.5,2.2) {$m$};
\node at (6.5,2.2) {$m$};
\node at (4.5,3.7) {$m$};
\node at (6.5,3.7) {$m$};
\node at (5,1) {$w = 1$};
\draw[ultra thick,->] (8,3) -- (11,3);
\node at (8.5,2.2) {$m$};
\node at (10.5,2.2) {$m$};
\node at (8.5,3.7) {$m+1$};
\node at (10.5,3.7) {$m+1$};
\node at (9,1) {$w = \frac{1-a_{i}b_{j}}{1-ta_{i}b_{j}}$};
\draw[ultra thick,->] (12,3) -- (13.5,3); \draw[ultra thick,->] (13.5,3) -- (13.5,4.5);
\node at (12.5,2.2) {$m$};
\node at (14.5,2.2) {$m$};
\node at (12.5,3.7) {$m+1$};
\node at (14.5,3.7) {$m$};
\node at (13,1) {$w = \frac{(1-t)a_{i}b_{j}}{1-ta_{i}b_{j}}$};
\draw[ultra thick,->] (17.5,1.5) -- (17.5,4.5);
\node at (16.5,2.2) {$m$};
\node at (18.5,2.2) {$m-1$};
\node at (16.5,3.7) {$m$};
\node at (18.5,3.7) {$m-1$};
\node at (17,1) {$w = \frac{t(1-a_{i}b_{j})}{1-ta_{i}b_{j}}$};
\draw[ultra thick,->] (17.5,1.5) -- (17.5,4.5);
\node at (16.5,2.2) {$m$};
\node at (18.5,2.2) {$m-1$};
\node at (16.5,3.7) {$m$};
\node at (18.5,3.7) {$m-1$};
\node at (17,1) {$w = \frac{t(1-a_{i}b_{j})}{1-ta_{i}b_{j}}$};
\draw[ultra thick,->] (21.5,1.5) -- (21.5,3); \draw[ultra thick,->] (21.5,3) -- (23,3);
\node at (20.5,2.2) {$m$};
\node at (22.5,2.2) {$m-1$};
\node at (20.5,3.7) {$m$};
\node at (22.5,3.7) {$m$};
\node at (21,1) {$w = \frac{1-t}{1-ta_{i}b_{j}}$};
\end{tikzpicture}
\caption{Weights of vertices in the stochastic six-vertex model}
\label{fig:1arrows}
\end{figure}

This is exactly the definition of the stochastic six vertex model in a quadrant, see e.g. \cite{BCG}, \cite[Section 3 and Section 5.1]{BBW} for a more detailed discussion. It is immediate that $d_0 (i,j)$ is the \textit{height function} for such a configuration on edges.

The application of Proposition \ref{prop:pathHLprocess} implies that the height function of the stochastic six vertex model is distributed according to the Hall-Littlewood process, which gives an alternative proof of \cite[Theorem 4.3]{BBW}.



Theorem \ref{prop:gen-HL-proc} combined with Proposition \ref{prop:pathHLprocess} provides observables for any points on the down-right path in $\Z_{\ge 0}$. Let us formulate an explicit statement for two points.

\begin{proposition}
\label{prop:stoh6vertMom}
Let $h (i,j)$ be the height function for the stochastic six vertex model with weights given by Figure \ref{fig:1arrows}.
For any $m_1 \ge m_2$ and $n_1 \le n_2$, and arbitrary integers $k_1, k_2 \ge 0$ we have
\begin{multline*}
\mathbf E t^{k_1 h (m_1,n_1) + k_2 h (m_2,n_2)} = [ z_1^{-1} \dots z_{k_1}^{-1} z_{k_1+1}^{-1} \dots z_{k_1 + k_2}^{-1}] t^{n_1 k_1 + n_2 k_2} \prod_{1 \le i < j \le k_1+k_2} \frac{ 1 - z_i^{-1} z_j}{1 - t^{-1} z_i^{-1} z_j} \\ \times \frac{1}{z_1 \dots z_{k_1} z_{k_1+1} \dots z_{k_1 + k_2}} \prod_{r=1}^{k_1+k_2} \left( \prod_{j=1}^{n_1} \frac{1+z_r b_j}{1+t z_r b_j} \prod_{i=1}^{m_2} \frac{1+t^{-1} z_r^{-1} a_i}{1 + z_r^{-1} a_i} \right) \prod_{r=1}^{k_1} \prod_{i=m_2+1}^{m_1} \frac{1+t^{-1} z_r^{-1} a_i}{1 + z_r^{-1} a_i} \\ \times \prod_{r=k_1+1}^{k_2} \prod_{j=n_1+1}^{n_2} \frac{1+z_r b_j}{1+t z_r b_j}.
\end{multline*}
\end{proposition}
\begin{proof}
The height function is given by $h(i,j) = d_0 (i,j) = j - \Lambda'_1 (i,j)$. Applying Theorem \ref{prop:gen-HL-proc}, we arrive at the claimed formula.
\end{proof}

\subsection{Two-layer stochastic vertex model}
\label{sec:2vertMod}

In this section we introduce a natural generalization of a stochastic six vertex model (see Section \ref{sec:st6v}).

Let $d_0 (i,j)$ and $d_1 (i,j)$ be the number of coordinates of $\Lambda(i,j)$ which are equal to 0 and 1, respectively. Equivalently, if we think about $\Lambda(i,j)$ as a Young diagram, then we have $d_0 (i,j) = j - \Lambda'_1 (i,j)$, and $d_1 (i,j) = \Lambda'_1 (i,j) - \Lambda'_2 (i,j)$. We will also use notation $h_0 (i,j) := d_0(i,j)$, $h_1 (i,j) := d_0 (i,j) + d_1 (i,j)$.
Note that interlacing constraints entail $d_0 (i-1,j) \ge d_0 (i,j) \ge d_0 (i-1,j) -1$, $d_0 (i,j-1) \le d_0 (i,j) \le d_0 (i,j+1)$, $h_1 (i-1,j) \ge h_1 (i,j) \ge h_1 (i-1,j) -1$, $h_1 (i,j-1) \le h_1 (i,j) \le h_1 (i,j+1)$
(the asymmetry coming because the lengths of signatures increase in $y$-direction, but not in $x$-direction).

\begin{proposition}
$(d_0 (i,j), d_1 (i,j))$ is determined by $(d_0 (i-1,j), d_1 (i-1,j))$, $(d_0 (i,j-1), d_1 (i,j-1))$, $(d_0 (i-1,j-1), d_1 (i-1,j-1))$ in a Markovian way. The explicit transition probabilities for this Markov chain are written in Figure \ref{fig:colored-2arrows}. In Figure \ref{fig:colored-2arrows} we draw a black arrow between boxes if they contain distinct values of $d_0$, and a red arrow if they contain distinct values of $d_0+ d_1$.
\end{proposition}
\begin{proof}
Note that we can consider only three variants for the value of input: $P(r_{ij}=0) = \frac{1-ab}{1-tab}$, $P(r_{ij}=1) = \frac{(1-t)ab(1-ab)}{1-tab}$, $P(r_{ij} \ge 2) = \frac{(1-t)a^2 b^2}{1-tab}$. The rest is a direct (though a little bit tedious) check with the use of the description given in Section \ref{sec:1cDescr}.
\end{proof}

\begin{figure}
\begin{tikzpicture}[>=stealth,scale=0.7]
\draw[ultra thick,->] (0,3) -- (3,3);
\draw[ultra thick,->, dashed, draw=red,fill=red] (0,3.2) -- (3,3.2);
\draw[ultra thick,->] (1.5,1.5) -- (1.5,4.5);
\draw[ultra thick,->, dashed, draw=red,fill=red] (1.7,1.5) -- (1.7,4.5);
\node at (0.5,2.5) {$m$}; \node at (0.5,2) {$k$};
\node at (2.5,2.5) {$m-1$}; \node at (2.5,2) {$k$};
\node at (0.5,4.2) {$m+1$}; \node at (0.5,3.7) {$k$};
\node at (2.5,4.2) {$m$}; \node at (2.5,3.7) {$k$};
\node at (1,1) {$w = 1$};
\draw[ultra thick,->] (4,3) -- (7,3);
\draw[ultra thick,->] (5.5,1.5) -- (5.5,4.5);
\node at (4.5,2.5) {$m$}; \node at (4.5,2) {$k$};
\node at (6.5,2.5) {$m-1$}; \node at (6.5,2) {$k+1$};
\node at (4.5,4.2) {$m+1$}; \node at (4.5,3.7) {$k-1$};
\node at (6.5,4.2) {$m$}; \node at (6.5,3.7) {$k$};
\node at (5,1) {$w = 1$};
\draw[ultra thick,->, dashed, draw=red,fill=red] (8,3.2) -- (11,3.2);
\draw[ultra thick,->, dashed, draw=red,fill=red] (9.7,1.5) -- (9.7,4.5);
\node at (8.5,2.5) {$m$}; \node at (8.5,2) {$k$};
\node at (10.5,2.5) {$m-1$}; \node at (10.5,2) {$k$};
\node at (8.5,4.2) {$m+1$}; \node at (8.5,3.7) {$k$};
\node at (10.5,4.2) {$m$}; \node at (10.5,3.7) {$k$};
\node at (9,1) {$w = 1$};
\node at (12.5,2.5) {$m$}; \node at (12.5,2) {$k$};
\node at (14.5,2.5) {$m$}; \node at (14.5,2) {$k$};
\node at (12.5,4.2) {$m$}; \node at (12.5,3.7) {$k$};
\node at (14.5,4.2) {$m$}; \node at (14.5,3.7) {$k$};
\node at (13,1) {$w = 1$};
\draw[ultra thick,->] (16,3) -- (17.5,3); \draw[ultra thick,->] (17.5,3) -- (17.5,4.5);
\node at (16.5,2.5) {$m$}; \node at (16.5,2) {$k$};
\node at (18.5,2.5) {$m-1$}; \node at (18.5,2) {$k$};
\node at (16.5,4.2) {$m$}; \node at (16.5,3.7) {$k$};
\node at (18.5,4.2) {$m$}; \node at (18.5,3.7) {$k$};
\node at (17,1) {$w = \frac{(1-t)ab}{1- t ab}$};
\draw[ultra thick,->] (20,3) -- (21.5,3); \draw[ultra thick,->] (21.5,3) -- (21.5,4.5);
\draw[ultra thick,->, dashed, draw=red,fill=red] (20,3.2) -- (23,3.2);
\draw[ultra thick,->, dashed, draw=red,fill=red] (21.7,1.5) -- (21.7,4.5);
\node at (20.5,2.5) {$m$}; \node at (20.5,2) {$k$};
\node at (22.5,2.5) {$m-1$}; \node at (22.5,2) {$k+1$};
\node at (20.5,4.2) {$m+1$}; \node at (20.5,3.7) {$k$};
\node at (22.5,4.2) {$m$}; \node at (22.5,3.7) {$k$};
\node at (21,1) {$w = \frac{(1-t)ab}{1- t ab}$};
\draw[ultra thick,->] (0,8) -- (3,8);
\node at (0.5,7.5) {$m$}; \node at (0.5,7) {$k$};
\node at (2.5,7.5) {$m$}; \node at (2.5,7) {$k$};
\node at (0.5,9.2) {$m+1$}; \node at (0.5,8.7) {$k-1$};
\node at (2.5,9.2) {$m+1$}; \node at (2.5,8.7) {$k-1$};
\node at (1,6) {$w = \frac{1-ab}{1-t ab}$};
\draw[ultra thick,->] (4,8) -- (7,8);
\draw[ultra thick,->, dashed, draw=red,fill=red] (4,8.2) -- (7,8.2);
\draw[ultra thick,->, dashed, draw=red,fill=red] (5.7,6.5) -- (5.7,9.5);
\node at (4.5,7.5) {$m$}; \node at (4.5,7) {$k$};
\node at (6.5,7.5) {$m$}; \node at (6.5,7) {$k+1$};
\node at (4.5,9.2) {$m+1$}; \node at (4.5,8.7) {$k$};
\node at (6.5,9.2) {$m+1$}; \node at (6.5,8.7) {$k+1$};
\node at (5,6) {$w = \frac{1-ab}{1-t ab}$};
\draw[ultra thick,->, dashed, draw=red,fill=red] (8,8.2) -- (11,8.2);
\node at (8.5,7.5) {$m$}; \node at (8.5,7) {$k$};
\node at (10.5,7.5) {$m$}; \node at (10.5,7) {$k$};
\node at (8.5,9.2) {$m$}; \node at (8.5,8.7) {$k+1$};
\node at (10.5,9.2) {$m$}; \node at (10.5,8.7) {$k+1$};
\node at (9,6) {$w = \frac{1-ab}{1-t ab}$};
\draw[ultra thick,->, dashed, draw=red,fill=red] (12,8.2) -- (13.5,8.2);
\draw[ultra thick,->, dashed, draw=red,fill=red] (13.5,8.2) -- (13.5,9.7);
\node at (12.5,7.5) {$m$}; \node at (12.5,7) {$k$};
\node at (14.5,7.5) {$m$}; \node at (14.5,7) {$k$};
\node at (12.5,9.2) {$m$}; \node at (12.5,8.7) {$k+1$};
\node at (14.5,9.2) {$m$}; \node at (14.5,8.7) {$k$};
\node at (13,6) {$w = \frac{(1-t)ab}{1- t ab}$};
\draw[ultra thick,->] (16,8) -- (19,8); \draw[ultra thick,->] (17.5,6.5) -- (17.5,9.5);
\draw[ultra thick,->, dashed, draw=red,fill=red] (16,8.2) -- (19,8.2);
\node at (16.5,7.5) {$m$}; \node at (16.5,7) {$k$};
\node at (18.5,7.5) {$m$}; \node at (18.5,7) {$k$};
\node at (16.5,9.2) {$m$}; \node at (16.5,8.7) {$k+1$};
\node at (18.5,9.2) {$m$}; \node at (18.5,8.7) {$k+1$};
\node at (17,6) {$w = \frac{1-ab}{1-t ab}$};
\draw[ultra thick,->] (20,8) -- (23,8); \draw[ultra thick,->] (21.5,6.5) -- (21.5,9.5);
\draw[ultra thick,->, dashed, draw=red,fill=red] (20,8.2) -- (21.7,8.2);
\draw[ultra thick,->, dashed, draw=red,fill=red] (21.7,8.2) -- (21.7,9.7);
\node at (20.5,7.5) {$m$}; \node at (20.5,7) {$k$};
\node at (22.5,7.5) {$m$}; \node at (22.5,7) {$k$};
\node at (20.5,9.2) {$m$}; \node at (20.5,8.7) {$k+1$};
\node at (22.5,9.2) {$m$}; \node at (22.5,8.7) {$k$};
\node at (21,6) {$w = \frac{(1-t)ab}{1- t ab}$};
\draw[ultra thick,->] (1.5,11.5) -- (1.5,14.5);
\node at (0.5,12.5) {$m$}; \node at (0.5,12) {$k$};
\node at (2.5,12.5) {$m-1$}; \node at (2.5,12) {$k$};
\node at (0.5,14.2) {$m$}; \node at (0.5,13.7) {$k$};
\node at (2.5,14.2) {$m-1$}; \node at (2.5,13.7) {$k$};
\node at (1,11) {$w = \frac{t(1-ab)}{1-t ab}$};
\draw[ultra thick,->] (5.5,11.5) -- (5.5,13.0); \draw[ultra thick,->] (5.5,13) -- (7,13.0);
\node at (4.5,12.5) {$m$}; \node at (4.5,12) {$k$};
\node at (6.5,12.5) {$m-1$}; \node at (6.5,12) {$k$};
\node at (4.5,14.2) {$m$}; \node at (4.5,13.7) {$k$};
\node at (6.5,14.2) {$m$}; \node at (6.5,13.7) {$k$};
\node at (5,11) {$w = \frac{1-t}{1-t ab}$};
\draw[ultra thick,->] (9.5,11.5) -- (9.5,14.5);
\draw[ultra thick,->, dashed, draw=red,fill=red] (9.7,11.5) -- (9.7,14.5);
\draw[ultra thick,->, dashed, draw=red,fill=red] (8,13) -- (11,13);
\node at (8.5,12.5) {$m$}; \node at (8.5,12) {$k$};
\node at (10.5,12.5) {$m-1$}; \node at (10.5,12) {$k$};
\node at (8.5,14.2) {$m$}; \node at (8.5,13.7) {$k$};
\node at (10.5,14.2) {$m-1$}; \node at (10.5,13.7) {$k$};
\node at (9,11) {$w = \frac{t(1-ab)}{1-t ab}$};
\draw[ultra thick,->] (13.5,11.5) -- (13.5,13.0); \draw[ultra thick,->] (13.5,13) -- (15,13.0);
\draw[ultra thick,->, dashed, draw=red,fill=red] (13.7,11.5) -- (13.7,14.5);
\draw[ultra thick,->, dashed, draw=red,fill=red] (12,13.2) -- (15,13.2);
\node at (12.5,12.5) {$m$}; \node at (12.5,12) {$k$};
\node at (14.5,12.5) {$m-1$}; \node at (14.5,12) {$k$};
\node at (12.5,14.2) {$m$}; \node at (12.5,13.7) {$k+1$};
\node at (14.5,14.2) {$m$}; \node at (14.5,13.7) {$k$};
\node at (13,11) {$w = \frac{1-t}{1-t ab}$};
\draw[ultra thick,->, dashed, draw=red,fill=red] (21.5,11.5) -- (21.5,14.5);
\node at (20.5,12.5) {$m$}; \node at (20.5,12) {$k$};
\node at (22.5,12.5) {$m$}; \node at (22.5,12) {$k-1$};
\node at (20.5,14.2) {$m$}; \node at (20.5,13.7) {$k$};
\node at (22.5,14.2) {$m$}; \node at (22.5,13.7) {$k-1$};
\node at (21,11) {$w = \frac{t(1-ab)}{1-t ab}$};
\draw[ultra thick,->, dashed, draw=red,fill=red] (17.5,11.5) -- (17.5,13.0); \draw[ultra thick,->, dashed, draw=red,fill=red] (17.5,13) -- (19,13.0);
\node at (16.5,12.5) {$m$}; \node at (16.5,12) {$k$};
\node at (18.5,12.5) {$m$}; \node at (18.5,12) {$k-1$};
\node at (16.5,14.2) {$m$}; \node at (16.5,13.7) {$k$};
\node at (18.5,14.2) {$m$}; \node at (18.5,13.7) {$k$};
\node at (17,11) {$w = \frac{1-t}{1-t ab}$};
\draw[ultra thick,->] (1.5,16.5) -- (1.5,19.5); \draw[ultra thick,->] (0,18) -- (3,18);
\draw[ultra thick,->, dashed, draw=red,fill=red] (1.7,16.5) -- (1.7,19.5);
\node at (0.5,17.5) {$m$}; \node at (0.5,17) {$k$};
\node at (2.5,17.5) {$m-1$}; \node at (2.5,17) {$k$};
\node at (0.5,19.2) {$m+1$}; \node at (0.5,18.7) {$k-1$};
\node at (2.5,19.2) {$m$}; \node at (2.5,18.7) {$k-1$};
\node at (1,16) {$w = \frac{t(1-ab)}{1-t ab}$};
\draw[ultra thick,->] (5.5,16.5) -- (5.5,19.5); \draw[ultra thick,->] (4,18) -- (7,18);
\draw[ultra thick,->, dashed, draw=red,fill=red] (5.7,16.5) -- (5.7,18.2); \draw[ultra thick,->, dashed, draw=red,fill=red] (5.7,18.2) -- (7,18.2);
\node at (4.5,17.5) {$m$}; \node at (4.5,17) {$k$};
\node at (6.5,17.5) {$m-1$}; \node at (6.5,17) {$k$};
\node at (4.5,19.2) {$m+1$}; \node at (4.5,18.7) {$k-1$};
\node at (6.5,19.2) {$m$}; \node at (6.5,18.7) {$k$};
\node at (5,16) {$w = \frac{1-t}{1-t ab}$};
\draw[ultra thick,->] (9.5,16.5) -- (9.5,19.5);
\draw[ultra thick,->, dashed, draw=red,fill=red] (9.7,16.5) -- (9.7,19.5);
\node at (8.5,17.5) {$m$}; \node at (8.5,17) {$k$};
\node at (10.5,17.5) {$m-1$}; \node at (10.5,17) {$k$};
\node at (8.5,19.2) {$m$}; \node at (8.5,18.7) {$k$};
\node at (10.5,19.2) {$m-1$}; \node at (10.5,18.7) {$k$};
\node at (9,16) {$w = \frac{t(1-ab)}{1-t ab}$};
\draw[ultra thick,->] (13.5,16.5) -- (13.5,18); \draw[ultra thick,->] (13.5,18) -- (15,18);
\draw[ultra thick,->, dashed, draw=red,fill=red] (13.7,16.5) -- (13.7,18.2); \draw[ultra thick,->, dashed, draw=red,fill=red] (13.7,18.2) -- (15,18.2);
\node at (12.5,17.5) {$m$}; \node at (12.5,17) {$k$};
\node at (14.5,17.5) {$m-1$}; \node at (14.5,17) {$k$};
\node at (12.5,19.2) {$m$}; \node at (12.5,18.7) {$k$};
\node at (14.5,19.2) {$m$}; \node at (14.5,18.7) {$k$};
\node at (13,16) {$w = \frac{1-t}{1-t ab}$};
\draw[ultra thick,->] (17.5,16.5) -- (17.5,19.5);
\draw[ultra thick,->, dashed, draw=red,fill=red] (16,18) -- (19,18); 
\node at (16.5,17.5) {$m$}; \node at (16.5,17) {$k$};
\node at (18.5,17.5) {$m-1$}; \node at (18.5,17) {$k+1$};
\node at (16.5,19.2) {$m$}; \node at (16.5,18.7) {$k+1$};
\node at (18.5,19.2) {$m-1$}; \node at (18.5,18.7) {$k+2$};
\node at (17,16) {$w = \frac{t(1-ab)}{1-t ab}$};
\draw[ultra thick,->] (21.5,16.5) -- (21.5,18); \draw[ultra thick,->] (21.5,18) -- (23,18);
\draw[ultra thick,->, dashed, draw=red,fill=red] (20,18) -- (21.5,18); \draw[ultra thick,->, dashed, draw=red,fill=red] (21.5,18) -- (21.5,19.5);
\node at (20.5,17.5) {$m$}; \node at (20.5,17) {$k$};
\node at (22.5,17.5) {$m-1$}; \node at (22.5,17) {$k+1$};
\node at (20.5,19.2) {$m$}; \node at (20.5,18.7) {$k+1$};
\node at (22.5,19.2) {$m$}; \node at (22.5,18.7) {$k$};
\node at (21,16) {$w = \frac{1-t}{1-t ab}$};
\draw[ultra thick,->] (0,23) -- (3,23);
\draw[ultra thick,->, dashed, draw=red,fill=red] (1.5,21.5) -- (1.5,23.2); \draw[ultra thick,->, dashed, draw=red,fill=red] (1.5,23.2) -- (3,23.2);
\node at (0.5,22.5) {$m$}; \node at (0.5,22) {$k$};
\node at (2.5,22.5) {$m$}; \node at (2.5,22) {$k-1$};
\node at (0.5,24.2) {$m+1$}; \node at (0.5,23.7) {$k-1$};
\node at (2.5,24.2) {$m+1$}; \node at (2.5,23.7) {$k-1$};
\node at (1,21) {$w = \frac{1-ab}{1-t ab} \frac{1-t}{1-t^k}$};
\draw[ultra thick,->] (4,23) -- (7,23);
\draw[ultra thick,->, dashed, draw=red,fill=red] (5.5,21.5) -- (5.5,24.5);
\node at (4.5,22.5) {$m$}; \node at (4.5,22) {$k$};
\node at (6.5,22.5) {$m$}; \node at (6.5,22) {$k-1$};
\node at (4.5,24.2) {$m+1$}; \node at (4.5,23.7) {$k-1$};
\node at (6.5,24.2) {$m+1$}; \node at (6.5,23.7) {$k-2$};
\node at (5,21) {$w = \frac{1-ab}{1-t ab} \frac{t-t^k}{1-t^k}$};
\draw[ultra thick,->] (8,23) -- (9.5,23); \draw[ultra thick,->] (9.5,23) -- (9.5,24.5);
\draw[ultra thick,->, dashed, draw=red,fill=red] (9.5,21.5) -- (9.5,23); \draw[ultra thick,->, dashed, draw=red,fill=red] (9.5,23) -- (11,23);
\node at (8.5,22.5) {$m$}; \node at (8.5,22) {$k$};
\node at (10.5,22.5) {$m$}; \node at (10.5,22) {$k-1$};
\node at (8.5,24.2) {$m+1$}; \node at (8.5,23.7) {$k-1$};
\node at (10.5,24.2) {$m$}; \node at (10.5,23.7) {$k$};
\node at (9,21) {$w = \frac{(1-t)ab}{1-t ab}$};
\draw[ultra thick,->] (12,23) -- (15,23);
\draw[ultra thick,->, dashed, draw=red,fill=red] (12,23.2) -- (15,23.2);
\node at (12.5,22.5) {$m$}; \node at (12.5,22) {$k$};
\node at (14.5,22.5) {$m$}; \node at (14.5,22) {$k$};
\node at (12.5,24.2) {$m+1$}; \node at (12.5,23.7) {$k$};
\node at (14.5,24.2) {$m+1$}; \node at (14.5,23.7) {$k$};
\node at (13,21) {$w = \frac{1-ab}{1-t ab}$};
\draw[ultra thick,->] (16,23) -- (17.5,23);\draw[ultra thick,->] (17.5,23) -- (17.5,24.5);
\draw[ultra thick,->, dashed, draw=red,fill=red] (16,23.2) -- (19,23.2);
\node at (16.5,22.5) {$m$}; \node at (16.5,22) {$k$};
\node at (18.5,22.5) {$m$}; \node at (18.5,22) {$k$};
\node at (16.5,24.2) {$m+1$}; \node at (16.5,23.7) {$k$};
\node at (18.5,24.2) {$m+1$}; \node at (18.5,23.7) {$k+1$};
\node at (17,21) {$w = \frac{(1-t)ab(1-ab)}{1-t ab}$};
\draw[ultra thick,->] (20,23) -- (21.5,23);\draw[ultra thick,->] (21.5,23) -- (21.5,24.5);
\draw[ultra thick,->, dashed, draw=red,fill=red] (20,23.2) -- (21.7,23.2); \draw[ultra thick,->, dashed, draw=red,fill=red] (21.75,23.2) -- (21.7,24.5);
\node at (20.5,22.5) {$m$}; \node at (20.5,22) {$k$};
\node at (22.5,22.5) {$m$}; \node at (22.5,22) {$k$};
\node at (20.5,24.2) {$m+1$}; \node at (20.5,23.7) {$k$};
\node at (22.5,24.2) {$m$}; \node at (22.5,23.7) {$k$};
\node at (21,21) {$w = \frac{(1-t)a^2 b^2}{1-t ab}$};
\end{tikzpicture}

\caption{An arrow model for the Markov evolution of first two columns of HL-RSK field. $(d_0 (i,j), d_1 (i,j))$ is the top-right corner, the weight of each picture is in the bottom of it.  }
\label{fig:colored-2arrows}
\end{figure}
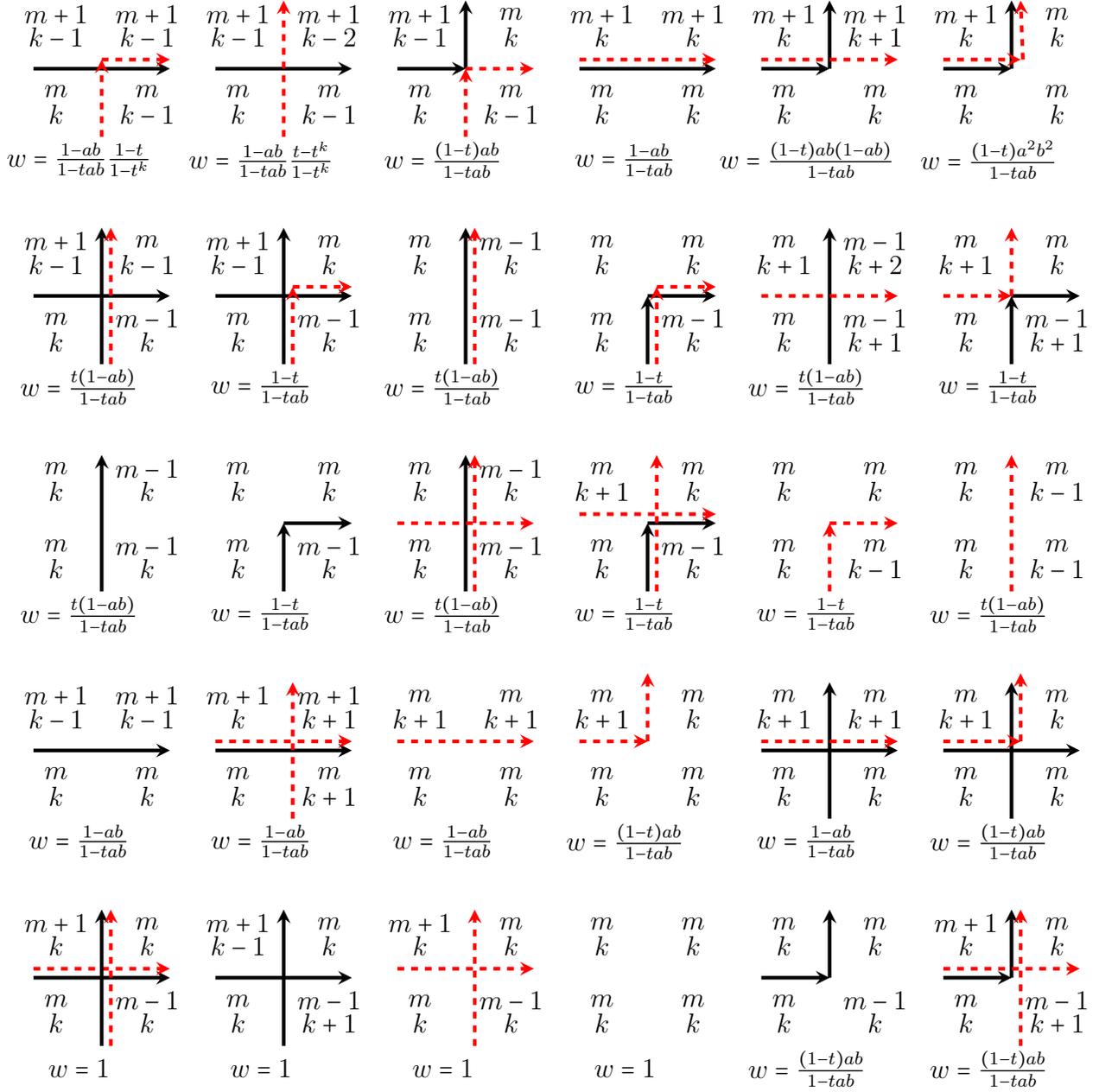

Analogously to the dynamics of the first column, we can now interpret the colored edges from Figure \ref{fig:colored-2arrows} as a two-layer stochastic vertex model. An important feature of this model is that some of the weights (in the top row of Figure \ref{fig:colored-2arrows}) depend on the value $d_1$. One may call this model \textit{integrable}, because the machinery of Hall-Littlewood processes provides a lot of observables for it.

In more detail, note that $h_0(m,n)$ is the number of vertical black arrows coming through segment $[(0;n); (m,n)]$, and note that $h_1 (m,n)$ is the number of vertical red arrows coming through segment $[(0;n); (m,n)]$. The study of these height functions is equivalent to the study of the configuration of arrows.
We have $h_0 (m,n) = d_0 (m,n) = n - \Lambda'_1 (m,n) $, and $h_1 (m,n) = d_0 (m,n) + d_1 (m,n) = n - \Lambda_1'(m,n) - \Lambda'_2 (m,n)$.
Theorem \ref{prop:gen-HL-proc} provides observables for any number of points on the down-right path in the quadrant. To make the expression simpler, let us write an explicit formula only in the case when points are on the same horizontal line.




\begin{proposition}
\label{prop:form-for-2verMod}
Fix an integer $k$, and consider an arbitrary sequence $(s(1), \dots, s(k))$, $s(i) \in \{0,1 \}$, $i=1, \dots, k$. Let $\{ z_{u;f} \}$, $u=1, \dots, k$, $f=1, \dots, s(u)+1$, (so $f$ takes either one or two possible values depending on $s(u)$) be formal variables.
For any integers $m_1 \ge m_2 \ge \dots \ge m_k$ and $n$ we have
\begin{multline*}
\mathbf E t^{ h_{s(1)} (m_1, n) + h_{s(2)} (m_2, n)+ \dots + h_{s(k)} (m_k,n)} = \left[ \prod_{u=1}^k \prod_{f=1}^{s(u)+1} z_{u;f}^{-1} \right] t^{n k} \prod_{1 \le i<j \le k} \prod_{f=1}^{s(i)+1} \prod_{\hat f=1}^{s(j)+1} \frac{1- z_{i;f}^{-1} z_{j;\hat f}}{1 - t^{-1} z_{i;f}^{-1} z_{j;\hat f}}
\\ \times \prod_{l=1}^k \left( \mathbf{1}_{s(l)=1} \frac{- (z_{l;1} - z_{l;2})^2}{2 z_{l;1}^2 z_{l;2}^2} + \mathbf{1}_{s(l)=0} \frac{1}{z_{l;1}} \right) \prod_{l=1}^k \left( \prod_{j=1}^n \prod_{f=1}^{s(l)+1} \frac{1+z_{l;f} b_j}{1 + t z_{l;f} b_j} \prod_{i=1}^{m_l} \prod_{f=1}^{s(l)+1} \frac{1+t^{-1} z_{l;f}^{-1} a_i}{1 + z_{l;f}^{-1} a_i} \right).
\end{multline*}
\end{proposition}
\begin{proof}
This is an immediate corollary of Theorem \ref{prop:gen-HL-proc}.
\end{proof}

\subsection{Multi-layer stochastic vertex model}
\label{sec:multi-layer-vertMod}

Similarly with the previous section, one can consider the dynamics which appears as a restriction of the HL-RSK field to the first $r$ columns of Young diagrams. It is determined by numbers $(d_0 (m,n), \dots, d_{l-1} (m,n))$, where $d_i (m,n)$ is the number of particles in $\Lambda (m,n)$ which are equal to $i$. Analogously with the previous sections, one can interpret this dynamics as a stochastic multi-layer vertex model. We do not write the rules of this dynamics explicitly (of course, they are completely determined by HL-RSK algorithm from Section \ref{sec:1cDescr}). However, we (informally) remark that in the multi-layer case only ``neighboring'' arrows interact, and all possible interactions between arrows are present in the two-layer case.

Such a multi-colored model can be described in terms of the height functions $h_0 (m,n) = d_0 (m,n)$, $h_1 (m,n) = d_0 (m,n)+ d_1 (m,n)$, ..., $h_{l-1} (m,n) = d_0 (m,n)+d_1 (m,n) + \dots + d_{l-1} (m,n)$. Theorem \ref{prop:gen-HL-proc} gives an explicit formula for expressions of the form
$$
\mathbf E t^{ \sum_{i=1}^p \sum_{j=0}^{l-1} k_{i,j} h_{j} (m_i,n_i)}
$$
for any natural numbers $\{ k_{i,j} \}$ and $m_1 \ge m_2 \ge \dots \ge m_p$, $n_1 \le n_2 \le \dots \le n_p$. This gives all joint moments of the random vector $\{ t^{h_j (m_i,n_i)} \}_{i,j}$. Note that each coordinate of this vector takes values between 0 and 1, so in principle the observables provide enough information for completely determining the distributions of the height functions.

\subsection{ASEP}
\label{sec:asep}

There exists a limit transition from a stochastic six vertex model to the asymmetric simple exclusion process (ASEP), see \cite[Section 2.2]{BCG}. Using this limit transition, one can deduce Tracy-Widom formulas for ASEP and their multi-point generalisations through the machinery of Hall-Littlewood processes. Let us briefly describe the procedure.

A stochastic six vertex model can be interpreted as a particle configuration evolution. Namely, let us place a particle in $(x, N)$, $x \in \Z_{\ge 0} + 1/2$, $N \in \Z_{\ge 0}$, if $d_0 (x+1/2, N) = d_0 (x-1/2, N) - 1$. Equivalently, these are the places where vertical lines considered in Section \ref{sec:st6v} intersect with lines given by equations $y=N$. For a fixed $N$ the particles form a random configuration $\mathbf{X}(N) \subset \Z_{\ge 0} +1/2$, and with the growth of $N$ the configuration changes (for example, new particles appear).

For the limit we consider it will be convenient to make a shift of this configuration: Let us define $\mathbf{\tilde X}(N) := \mathbf{X}(N) - N$; that is, each particle is shifted to the left by $N$.

Consider the case of homogeneous equal parameters $a_i=a$, $b_j = a$, for all $i,j \ge 1$, and let $ab= a^2 = 1 - (1-t) \eps$ for $\eps \ge 0$. Note that if $\eps=0$, then $d_0 (i,N) = N-i$, for all $N=1,2,\dots$, $j=1,2,\dots, N$, because the local rules of growth always imply $d_0 (i,j) = d_0 (i+1,j+1)$. It is the height function of a particle configuration $\mathbf{\tilde X}(N)$ with particles at points $-1/2, -3/2, -5/2, \dots, -N+1/2$.

Consider now small $\eps >0$ and two neighboring levels $\{d_0 (i,N) \}_{i \ge 1}$, and $\{d_0 (i,N+1) \}_{i \ge 1}$. We are interested in transition probabilities between these two levels which are constant or linear in $\eps$. The only transition probability which does not depend on $\eps$ is again given by $d_0 (i,N) = d_0 (i+1,N+1)$ for all $i$; the transition probability equals $1 - o(1)$. It is readily verified that transition probabilities which are linear in $\eps$ have the following form: we have $d_0 (i,N) = d_0 (i+1,N+1)$ for all $i \ge 1$ except one. In the exceptional value $\hat i$ we can have the following cases:
\begin{itemize}
\item
$d_0 (\hat i+1,N+1) = d_0 ( \hat i, N)+1$, if we have $d_0 (\hat i -1, N) = d_0 ( \hat i, N)+1 = d_0 ( \hat i+1, N) +1$. In this case the transition probability has the form $\eps + o( \eps)$.

\item $d_0 (\hat i+1,N+1) = d_0 ( \hat i, N)-1$, if we have $d_0 (\hat i -1, N) = d_0 ( \hat i, N) = d_0 ( \hat i+1, N) +1$. In this case the transition probability has the form $t \eps + o( \eps)$.
\end{itemize}

For a fixed $\tau>0$ let us consider the particle configuration $\mathbf{ \tilde X} (N)$ in the limit $\eps \to 0$, $N = [ \tau \eps^{-1} ] \to \infty$. The resulting particle configuration $\mathbf{\hat X (\tau)}$ will consist of infinite amount of particles living on $\Z+1/2$. The initial condition becomes $\mathbf{ \hat X} (0) =\{-1/2, -3/2, -5/2, \dots \}$, which is called step initial condition. $\mathbf{\hat X (\tau)}$ is a continuous time dynamics; the jump rates are determined by linear in $\eps$ transition probabilities listed above. In terms of particles, they mean that a particle at $x$ jumps by 1 to the right with intensity $1$ if there is no particle at $x+1$, and a particle at $x$ jumps to the left with intensity $t$, if there is no particles at $x-1$. This is exactly the rules of ASEP dynamics (see \cite{A} for a more detailed explanation of this limit transition).

Theorem \ref{prop:gen-HL-proc} supplies the multi-point formulas for the observables of ASEP in distinct points for a fixed moment of time $\tau$. They are not new: Similar one-point formulas were first obtained in \cite{TW1}, \cite{TW2}, and subsequently in \cite{BCS}. Similar multi-point formulas were obtained in \cite[Section 10.1]{BP2}. In our setting these formulas appear from the general formalism of Hall-Littlewood processes.

\begin{proposition}
\label{prop:momAsep}
Consider the ASEP $\mathbf{\hat X (\tau)}$ on $\Z+1/2$ in which particles are initially placed in $-1/2, -3/2, -5/2, \dots$ (step initial condition); the rate of a particle jump to the right is $1$, the rate of a particle jump to the left is $t$. For $\tilde m \in \Z$ let $ h (\tilde m;\tau)$ be the number of particles to the right of $\tilde m$ at time $\tau$. We have
\begin{multline*}
t^{ k_1 h ( \tilde m_1) + k_2 h (\tilde m_2)} = [ z_1^{-1} \dots z_{k_1}^{-1} z_{k_1+1}^{-1} \dots z_{k_1 + k_2}^{-1}] \prod_{1 \le i < j \le k_1+k_2} \frac{ 1 - z_i^{-1} z_j}{1 - t^{-1} z_i^{-1} z_j} \\ \times \frac{1}{z_1 \dots z_{k_1} z_{k_1+1} \dots z_{k_1 + k_2}}
\prod_{r=1}^{k_1+k_2} \exp \left( \frac{ - \tau z_r (1-t)^2}{(1+z_r)(1+t z_r)} \right) \prod_{r=1}^{k_1} \left( \frac{z_r + t^{-1}}{z_r +1} \right)^{\tilde m_1} \prod_{r=k_1+1}^{k_1+k_2} \left( \frac{z_r + t^{-1}}{z_r +1} \right)^{\tilde m_2}.
\end{multline*}
\end{proposition}
\begin{proof}
Consider points $(m_1,N)$, $(m_2,N)$ in a stochastic six vertex model. In the limit transition to ASEP: $a_i=a$, $b_j = a$, for all $i,j \ge 1$, $ab= a^2 = 1 - (1-t) \eps$, consider $N = \lfloor \tau \eps^{-1} \rfloor$, $m_1 = m_1(N) = N+ \tilde m_1$, $m_2 = m_2(N) = N+ \tilde m_2$, where $\tilde m_1 \ge \tilde m_2$ are fixed integers. Let us apply Proposition \ref{prop:stoh6vertMom} to observables in points $(m_1,N)$, $(m_2,N)$ in a stochastic six vertex model, and make $\eps \to 0$ limit. Using
$$
\frac{1+b z}{1+t b z} \frac{1 +t^{-1} z^{-1} a}{ 1 + z^{-1} a} = \left( 1 - \frac{\eps z (1-t)^2}{(1+z) (1+tz)} + o(\eps) \right)^{\lfloor \tau / \eps \rfloor} \to \exp \left ( \frac{ - \tau z (1-t)^2}{(1+z)(1+tz)} \right),
$$
we arrive at the claimed formula.
\end{proof}

\subsection{Two-layer ASEP}
\label{sec:2asep}

In this section we introduce a new two-layer extension of ASEP; we call it the \textit{two-layer ASEP}.

It is obtained by the limit transition from Section \ref{sec:asep} in the two-layer stochastic vertex model introduced in Section \ref{sec:2vertMod}. The resulting object will be the dynamics of particle configurations on $\Z+1/2$. However, now the particles will be of two different types.

We will use notation from Section \ref{sec:2vertMod}. Let us place a black particle in $(x, N)$, $x \in \Z_{\ge 0} + 1/2$, $N \in \Z_{\ge 0}$, if $d_0 (x+1/2, N) = d_0 (x-1/2, N) - 1$, and let us place a red particle in $(x, N)$ if $h_1 (x+1/2, N) = h_1 (x-1/2, N) - 1$. One can have two particles of different colors in the same site, but not of the same color. For a fixed $N$ the particles form random configurations $\mathbf{X}(N) \subset \Z_{\ge 0} +1/2$, $\mathbf{Y}(N) \subset \Z_{\ge 0} +1/2$, and with the growth of $N$ the configuration changes (for example, new particles appear). Let us recall that $\mathbf{\tilde X}(N) = \mathbf{X}(N) - N$, and, analogously, let us set $\mathbf{\tilde Y}(N) := \mathbf{Y}(N) - N$.

Let us consider the case of homogeneous equal parameters $a_i=a$, $b_j = b = a$, for all $i,j \ge 1$, $ab= a^2 = 1 - (1-t) \eps$, and the limit $\eps \to 0$, as in Section \ref{sec:asep}. Again, we are interested in constant and linear in $\eps$ terms of transition probabilities between $\{ d_0 (i,N), d_1 (i,N) \}_{i \ge 1}$ and $\{ d_0 (i,N+1), d_1 (i,N+1) \}_{i \ge 1}$. Using Figure \ref{fig:colored-2arrows}, we see that the probability of $d_0 (i,N) = d_0 (i+1,N+1), d_1 (i,N) = d_1 (i+1, N+1)$ for all $i \ge 1$, has the form $1- o(1)$ --- indeed, this happens in all cases when the weight in Figure \ref{fig:colored-2arrows} does not tend to 0 (that is, does not contain the factor $(1-ab)$). Linear in $\eps$ terms are given by $d_0 (i,N) = d_0 (i+1,N+1), d_1 (i,N) = d_1 (i+1, N+1)$ for all $i \ge 1$ except one value $\hat i$, and the values for these terms come from all pictures from Figure \ref{fig:colored-2arrows} which contain the factor $(1-ab)$.

For a fixed $\tau>0$ let us consider the particle configuration $\mathbf{ \tilde X} (N)$ in the limit $\eps \to 0$, $N = [ \tau \eps^{-1} ] \to \infty$. The resulting particle configurations $\mathbf{\hat X (\tau)}$, $\mathbf{\hat Y (\tau)}$ will consist of infinite amount of particles living on $\Z+1/2$. The initial condition becomes $\mathbf{ \hat X} (0) = \mathbf{ \hat Y} (0) =\{-1/2, -3/2, -5/2, \dots \}$. The process $(\mathbf{\hat X (\tau)}, \mathbf{\hat Y (\tau)})$ is a continuous time dynamics. For $\tilde m \in \Z$ let $h_0 (\tilde m; \tau)$ be the number of black particles to the right of $\tilde m$ at time $\tau$, and let $h_1 (\tilde m; \tau)$ be the number of red particles to the right of $\tilde m$ at time $\tau$. At each moment of time particles placed in two neighboring vertices might be transformed into another configuration in these two vertices with certain jump rates. The jump rates are determined by linear in $\eps$ terms. Note that these terms depend on parameter $k = k (\tilde m) = h_1 (\tilde m; \tau) - h_0 (\tilde m; \tau)$, which affects the jumps between $\tilde m-1/2$ and $\tilde m + 1/2$. Doing a case by case analysis, we arrive at expressions depicted in Figure \ref{fig:2levAsepp} (which coincides with the description given in Introduction).

Note that if $k(\tilde m)=1$, then the jump rate of the form $(t-t^k)/ (1-t^k)$ equals 0. This means that if we number black and, separately, red particles from right to left, then the black particle number $k$ cannot be to the right of the red particle number $k$.

\begin{figure}
\begin{tikzpicture}[>=stealth,scale=0.6]
\draw[black,fill=black] (1,1) circle (0.2);
\draw[black] (1.8,1) circle (0.2);
\node at (2.5,1) {$\mapsto$};
\draw[black] (3,1) circle (0.2);
\draw[black,fill=black] (3.8,1) circle (0.2);
\node at (2.5,-0.3) {$v=1$};
\draw[black,fill=black] (6,1) circle (0.2); \fill[draw=red, thick, pattern=north west lines, pattern color=red] (6,0.5) circle (0.2);
\fill[draw=red, thick, pattern=north west lines, pattern color=red] (6.8,0.5) circle (0.2);
\node at (7.5,1) {$\mapsto$};
\fill[draw=red, thick, pattern=north west lines, pattern color=red] (8,0.5) circle (0.2);
\draw[black,fill=black] (8.8,1) circle (0.2); \fill[draw=red, thick, pattern=north west lines, pattern color=red] (8.8,0.5) circle (0.2);
\node at (7.5,-0.3) {$v=1$};
\fill[draw=red, thick, pattern=north west lines, pattern color=red] (11,1) circle (0.2);
\draw[black] (11.8,1) circle (0.2);
\node at (12.5,1) {$\mapsto$};
\draw[black] (13,1) circle (0.2);
\fill[draw=red, thick, pattern=north west lines, pattern color=red] (13.8,1) circle (0.2);
\node at (12.5,-0.3) {$v=1$};
\draw[black,fill=black] (16,1) circle (0.2); \fill[draw=red, thick, pattern=north west lines, pattern color=red] (16,0.5) circle (0.2);
\draw[black,fill=black] (16.8,1) circle (0.2);
\node at (17.5,1) {$\mapsto$};
\draw[black,fill=black] (18,1) circle (0.2);
\draw[black,fill=black] (18.8,1) circle (0.2); \fill[draw=red, thick, pattern=north west lines, pattern color=red] (18.8,0.5) circle (0.2);
\node at (17.5,-0.3) {$v=1$};
\draw[black] (21,1) circle (0.2);
\draw[black, fill=black] (21.8,1) circle (0.2);
\node at (22.5,1) {$\mapsto$};
\draw[black, fill=black] (23,1) circle (0.2);
\draw[black] (23.8,1) circle (0.2);
\node at (22.5,-0.3) {$v=t$};
\draw[black,fill=black] (1.8,4) circle (0.2); \fill[draw=red, thick, pattern=north west lines, pattern color=red] (1,3.5) circle (0.2);
\fill[draw=red, thick, pattern=north west lines, pattern color=red] (1.8,3.5) circle (0.2);
\node at (2.5,4) {$\mapsto$};
\fill[draw=red, thick, pattern=north west lines, pattern color=red] (3,3.5) circle (0.2);
\draw[black,fill=black] (3,4) circle (0.2); \fill[draw=red, thick, pattern=north west lines, pattern color=red] (3.8,3.5) circle (0.2);
\node at (2.5,2.7) {$v=t$};
\fill[draw=red, thick, pattern=north west lines, pattern color=red] (6.8,4) circle (0.2);
\draw[black] (6,4) circle (0.2);
\node at (7.5,4) {$\mapsto$};
\draw[black] (8.8,4) circle (0.2);
\fill[draw=red, thick, pattern=north west lines, pattern color=red] (8,4) circle (0.2);
\node at (7.5,2.7) {$v=t$};
\draw[black,fill=black] (11,4) circle (0.2); \fill[draw=red, thick, pattern=north west lines, pattern color=red] (11,3.5) circle (0.2);
\draw[black,fill=black] (11.8,4) circle (0.2);
\node at (12.5,4) {$\mapsto$};
\draw[black,fill=black] (13,4) circle (0.2);
\draw[black,fill=black] (13.8,4) circle (0.2); \fill[draw=red, thick, pattern=north west lines, pattern color=red] (13.8,3.5) circle (0.2);
\node at (12.5,2.7) {$v=t$};
\draw[black,fill=black] (16.8,4) circle (0.2);
\fill[draw=red, thick, pattern=north west lines, pattern color=red] (16,4) circle (0.2);
\node at (17.5,4) {$\mapsto$};
\fill[draw=red, thick, pattern=north west lines, pattern color=red] (18.8,4) circle (0.2);
\draw[black,fill=black] (18,4) circle (0.2);
\node at (17.5,2.7) {$v=t$};
\draw[black,fill=black] (21.8,4) circle (0.2); \fill[draw=red, thick, pattern=north west lines, pattern color=red] (21.8,3.5) circle (0.2);
\draw[black] (21,4) circle (0.2);
\node at (22.5,4) {$\mapsto$};
\draw[black] (23.8,4) circle (0.2);
\draw[black,fill=black] (23,4) circle (0.2); \fill[draw=red, thick, pattern=north west lines, pattern color=red] (23,3.5) circle (0.2);
\node at (22.5,2.7) {$v=t$};
\draw[black,fill=black] (1,7) circle (0.2); \fill[draw=red, thick, pattern=north west lines, pattern color=red] (1,6.5) circle (0.2);
\draw[black] (1.8,7) circle (0.2);
\node at (2.5,7) {$\mapsto$};
\draw[black] (3,7) circle (0.2);
\draw[black,fill=black] (3.8,7) circle (0.2); \fill[draw=red, thick, pattern=north west lines, pattern color=red] (3.8,6.5) circle (0.2);
\node at (2.5,5.7) {$v=1$};
\fill[draw=red, thick, pattern=north west lines, pattern color=red] (6,6.5) circle (0.2);
\draw[black,fill=black] (6,7) circle (0.2); \draw[black] (6.8,7) circle (0.2);
\node at (7.5,7) {$\mapsto$};
\fill[draw=red, thick, pattern=north west lines, pattern color=red] (8.8,7) circle (0.2);
\draw[black,fill=black] (8,7) circle (0.2);
\node at (7.5,5.7) {$v=1-t$};
\draw[black,fill=black] (11,7) circle (0.2); \fill[draw=red, thick, pattern=north west lines, pattern color=red] (11.8,7) circle (0.2);
\node at (12.5,7) {$\mapsto$};
\draw[black,fill=black] (13.8,7) circle (0.2); \fill[draw=red, thick, pattern=north west lines, pattern color=red] (13.8,6.5) circle (0.2);
\draw[black] (13,7) circle (0.2);
\node at (12.5,5.7) {$v=\frac{1-t}{1-t^k}$};
\fill[draw=red, thick, pattern=north west lines, pattern color=red] (16.8,7) circle (0.2);
\draw[black,fill=black] (16,7) circle (0.2);
\node at (17.5,7) {$\mapsto$};
\draw[black,fill=black] (18.8,7) circle (0.2);
\fill[draw=red, thick, pattern=north west lines, pattern color=red] (18,7) circle (0.2);
\node at (17.5,5.7) {$v=\frac{t-t^k}{1-t^k}$};
\end{tikzpicture}
\caption{Jump rates of the two-layer ASEP. An empty circle denotes an empty site, while filled black and shaded red circles denote black and red particles. $k = k(\tilde m)$ is a parameter of the configuration at this point. }
\label{fig:2levAsepp}
\end{figure}
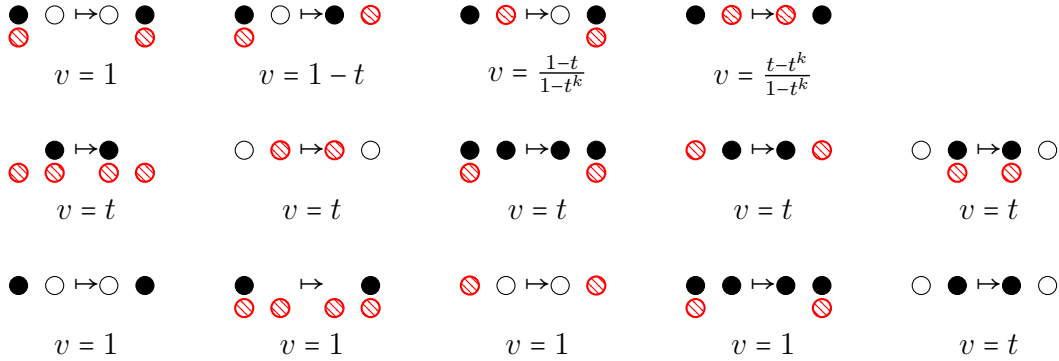

As we have already seen, the black particles evolve in a Markovian way as ASEP. The evolution of red particles is more complicated: It depends on the evolution of black ones.

The two-layer ASEP is an integrable system in the sense that we have formulas for observables which contain (at least, in principle) all information about the distribution of the model. A general formula for observables is given in Proposition \ref{prop:2ASEPform2} in the introduction. Proposition \ref{prop:2ASEPform2} is obtained from Theorem \ref{prop:gen-HL-proc} analogously to Proposition \ref{prop:momAsep}.



\subsection{A multi-layer ASEP}
\label{sec:multi-asep}

One can similarly consider the ASEP-like limit of the multi-layer stochastic vertex model from Section \ref{sec:multi-layer-vertMod}. The resulting continuous dynamics will consist of particles of many types; we do not attempt to describe explicitly the rules of their evolution in the current paper (though they are completely determined by the HL-RSK field).

As a conjecture, the formulas of Theorem \ref{prop:gen-HL-proc} should be sufficient for proving that the one-point fluctuations of the multi-layer ASEP converge to solutions of the multi-layer stochastic heat equation introduced in \cite{OCW}.

\end{document}